\documentclass{article}

\usepackage{mysty}

\newcommand{\Zak}{Z_{a,k}}

\newcommand{\Zbk}{Z_{b,k}}

\makeatletter
\let\reftagform@=\tagform@
\def\tagform@#1{\maketag@@@{(\ignorespaces\textcolor{black}{#1}\unskip\@@italiccorr)}}
\newcommand{\iref}[1]{\textup{\reftagform@{\tcr{\ref{#1}}}}}
\makeatother

\newcommand*{\mcap}{\mathbin{\scalebox{1.25}{\ensuremath{\cap}}}}%
\usepackage[round]{natbib}

    \usepackage[final]{neurips_2019-2}
\usepackage{nicefrac}       

\title{Best Pair Formulation \& Accelerated Scheme for Non-convex Principal Component Pursuit}

\author{
  Aritra Dutta \\
  KAUST\\
  Thuwal, KSA \\
  \texttt{aritra.dutta@kaust.edu.sa} 
\And
  Filip Hanzely \\
  KAUST\\
  Thuwal, KSA \\
  \texttt{filip.hanzely@kaust.edu.sa} \\
\And
  Jingwei Liang \\
  University of Cambridge\\
  Cambridge, UK\\
  \texttt{jl993@cam.ac.uk} \\
  \And
Peter Richt\'{a}rik \\
  KAUST\\
  Thuwal, KSA \\
  \texttt{peter.richtarik@kaust.edu.sa} \\
}

\begin{document}

\maketitle

\begin{abstract}
The {\emph best pair} problem aims to find a pair of points that minimize the distance between two disjoint sets. In this paper, we formulate the classical robust principal component analysis~(RPCA) as the best pair; which was not considered before. We design an accelerated proximal gradient scheme to solve it, for which we show global convergence, as well as the local linear rate. Our extensive numerical experiments on both real and synthetic data suggest that the algorithm outperforms relevant baseline algorithms in the literature. 
\end{abstract}

\section{Introduction}\label{sec:introduction}
Let $A\in\mathbb{R}^{m\times n}$ be a given matrix, the generalized low-rank recovery model can be written as
 \begin{eqnarray}\label{prblm:GLR}
\textstyle \min_{L\in\mathbb{R}^{m\times n}}\mathcal{F} (A,L)+\lambda\mathcal{R}(L),
\end{eqnarray}
where $\mathcal{F}(A, L)$ is a loss function, $\mathcal{R}(L)\eqdef\sum_{i=1}^n\mathcal{R}_i(L)$ is a suitable regularizer, and $\lambda>0$ is a balancing parameter. By an appropriate choice of the loss function and the regularizer, \eqref{prblm:GLR} can express a wide range of low-rank approximation problems of matrices. For example, by setting $\mathcal{F}(A,L)=\|A-L\|_F^2,\lambda = 1$, and $\mathcal{R}(L)=\iota_{{\rm rank}(L)\leq r}(L)$ --- the characteristic function \eqref{eq:charac-func} 
of the set $\{L \in \bbR^{m\times n}:{\rm rank}(L)\le r\}$, \eqref{prblm:GLR}, specializes to:
\begin{eqnarray}\label{pca}
\textstyle \min_{L\in\mathbb{R}^{m\times n}} \|A-L\|_F^2+\iota_{{\rm rank}(L)\leq r}(L),
\end{eqnarray}
which is a {\em best approximation} formulation of the classical principal component analysis~(PCA). The solution to problem \eqref{pca} is given by:~$\hat{L}=U{\mathbf H}_r(\Sigma)V^\top,$ where $U\Sigma V^\top = A$ is a singular value decomposition~(SVD) of $A$ and ${\mathbf H}_r(\cdot)$ is the hard-thresholding operator that keeps the $r$ largest singular values. 
Although PCA is vastly used and a successful designing tool in different engineering applications, it can only handle the presence of uniformly distributed noise and is rather sensitive to sparse outliers in the data matrix~\citep{LinChenMa,APG,candeslimawright}. 
To overcome this shortcoming and to deal with sparse errors, \citep{rpca_1,candeslimawright} replaced the Frobenius norm in (\ref{pca}) by the $\ell_0$~pseudo norm, 
and introduced the celebrated {\em principal component pursuit}~(PCP) problem:
\begin{eqnarray}\label{pcp}
\textstyle \min_{L\in\mathbb{R}^{m\times n}}\|A-L\|_{\ell_0}+\lambda{\rm rank}(L). 
\end{eqnarray}
However, the above problem is non-convex and NP-hard. One of the most commonly used, tractable surrogate reformulations of \eqref{pcp} is replacing the rank function with nuclear norm$\|L\|_\star$ 
and $\ell_0$ pseudo norm with $\ell_1$-norm $\|A-L\|_{\ell_1}$ \citep{caicandesshen, RechtFazelParrilo2007}.  
Exploiting this idea, {\it Robust PCA} (RPCA) was introduced as a convex surrogate of the PCP problem  \citep{APG,LinChenMa,candeslimawright}:
\begin{equation}\label{rpca}
\textstyle    \min_{L\in\mathbb{R}^{m\times n}}\|A-L\|_{\ell_1}+\lambda \|L\|_\star.
\end{equation}
It was shown in~\citep{rpca_1,candeslimawright} that under a rank-sparsity incoherence assumption, problem \eqref{pcp} can be provably solved via \eqref{rpca}, as the solutions of them lie close to each other with high probability. 

Besides~\eqref{rpca}, there are other formulations of RPCA. One of the most popular way is to introduce an auxiliary variable, $S$, and add an additional constraint $L+S = A$, which yields:
\begin{eqnarray}\label{rpca1}
\textstyle \min_{L,S\in\mathbb{R}^{m\times n}}\|S\|_{\ell_1}+\lambda\|L\|_\star
\quad \text{subject to} \quad L+S = A. 
\end{eqnarray}
This {\em constrained} formulation enables several avenues to solve RPCA, such as, the exact and inexact augmented Lagrangian method of multipliers by Lin et al. \citep{LinChenMa}, accelerated proximal gradient method \citep{APG}, alternating direction method \citep{adm_rpca}, alternating projection with intermediate denoising \citep{NIPS2014_5430}, dual approach \citep{dual_rpca}, and SpaRCS \citep{SpaRCS},   manifold optimization by Yi et al.\ \citep{RPCAgd} and Zhang and Yang \citep{zhangpca}, are a few popular ones. We refer to \citep{rpca_methods} for a comprehensive review of RPCA algorithms. 

For the discussion above, $A$ is fully observed with no data missing. One can consider that $A$ is partially observed, that is, there exists a projection operator (or simply a Bernoulli binary mask) $\proj_{\Omega}$ on the set of observed data entries $\Omega\subseteq [m]\times [n]$ and is defined by 
\beq\label{eq:binary}
\textstyle{\left(\proj_{\Omega} [A]\right)_{ij} 
=\begin{cases} 
A_{ij} & (i,j)\in\Omega , \\
0 & \text{otherwise.}
\end{cases}}
\eeq
The partial observed version of \eqref{rpca1} reads
\begin{eqnarray}\label{rpca2}
\textstyle \min_{L,S\in\mathbb{R}^{m\times n}}\|S\|_{\ell_1}+\lambda\|L\|_\star
\quad \text{subject to} \quad \proj_{\Omega}(L+S) = \proj_{\Omega}(A) . 
\end{eqnarray}

Besides \eqref{rpca1} and \eqref{rpca2}, other tractable reformulations of \eqref{pcp} still exist. For example, if the rank and target sparsity is user-inferred then it is common practice to relax the equality constraint in \eqref{rpca1} and consider it in the objective function as a penalty. This, together with explicit constraints on the target rank, $r$, and target sparsity level, $\alpha$,~({\it user-inferred} hyperparameters), leads to the GoDec formulation \citep{godec}. One can also extend the above model to the case of partially observed data that leads to a more general class of problems that is commonly known as the {\em robust matrix completion (RMC)} problem \citep{Chen_RMC,rmc_taoyuan, CherapanamjeriJN17,CherapanamjeriGJ17} 
that contains the variant proposed in \citep{godec} as a special case. With $S=0$, the matrix completion (MC) problem is also a special case of the RMC problem \citep{candesplan,Jain,caicandesshen,JN15,candes_MC,KeshavanMC,candes_MC2,Marecek_MC,LMAFIT}. 
Lastly, when the whole matrix is observed, the RMC problem is nothing but \eqref{rpca1}.

Recently, \citep{duttahanzely} reformulated~\eqref{pcp} as a non-convex feasibility problem, which does not require any objective function, convex relaxation, or surrogate convex constraints. Rather, it exploits the following idea: the solution to the PCP problem lies in the intersection of two sets---one convex and one non-convex, if one considers both the target rank $r$ and the target sparsity $\alpha$ as hyperparameters. 
Let $X = \BPa{ {\footnotesize \begin{matrix} S \\ L  \end{matrix} } } \in \bbR^{2m \times n}$ and $K = [\Id, \Id]$ where $\Id$ is the identity operator of $\bbR^{m\times n}$, define
\[
\begin{aligned}
\calX \eqdef \Ba{X : KX = A} , \enskip
\calY \eqdef \Ba{X : \rank(L) \leq r, \norm{S_{i, \cdot}}_{0} \leq\alpha m , \norm{S_{\cdot, j}}_{0} \leq \alpha n ,~ i \in [m],~ j \in [n]}.
\end{aligned}
\]
Note that $\calX$ is convex and $\calY$ is non-convex\footnote{The $\alpha$-sparsity constraint on $S$ means that for $\alpha\in(0,1)$, each row and column of $S$ contains no more than $\alpha n$ and $\alpha m$ number of non-zero entries, respectively. This is slightly more complicated than directly applying $\| \cdot \|_0$ constraint. However, it often works better in practice.}. 
Given the sets, Dutta et al.\ \citep{duttahanzely} reformulated~\eqref{pcp} as non-convex feasibility problem:
\beq\label{eq:ncvx-pcp}
\textstyle \find~~ X \in \bbR^{2m \times n} \enskip \textrm{such that} \enskip X \in  \calX \cap \calY  .   
\eeq
Note that if we replace the $\Id$ in $K$ with Bernoulli binary matrix, then we obtain the reformulation of PCP problem with partial observation. 

\vspace{-5pt}
\subsection{Formulation and Contributions}
In this paper we consider reformulating the feasibility problem~\eqref{eq:ncvx-pcp} as a {\em best pair} problem. 
Given two sets $\calX, \calY \subset \bbR^{2m\times n}$, the best pair problems aims to find a pair of points $(X^\star, Y^\star) \in \calX \times \calY$ such that they have the closest distance, that is $(X^\star, Y^\star)$ a the solution of the problem below:
\beq\label{eq:X-Y}
\min_{X\in \calX, Y \in \calY} ~ \sfrac{1}{2} \norm{X-Y}^2.
\eeq 
When the intersection of $\calX$ and $\calY$ is non-empty, that is $\calX \cap \calY \neq \emptyset$, \eqref{eq:X-Y} reduces to the feasibility problem, with $X^\star = Y^\star \in \calX\cap\calY$. Given a set $\calX$, define its characteristic function by
\begin{equation}\label{eq:charac-func}
\textstyle
\iota_{\calX}(X) \eqdef
\left\{
\begin{aligned}
0 &: X \in \calX ,\\
+\infty &: \text{otherwise}.
\end{aligned}
\right.
\end{equation}
Then \eqref{eq:X-Y} can be equivalently written as
\beq\label{eq:BP}
\min_{X, Y \in \bbR^{2m\times n}} \iota_{\calX}(X) + \sfrac{1}{2} \norm{X-Y}^2 + \iota_{\calY}(Y)   .
\eeq
Observe that for a given $Y$, problem \eqref{eq:BP} becomes $\min_{X \in \bbR^{2m\times n}} \iota_{\calX}(X) + \frac{1}{2} \norm{X-Y}^2$ which is the Moreau envelope \citep{bauschke2011convex} of $\iota_{\calX}(X)$ of index $1$: 
\[
^1\Pa{\iota_{\calX}(Y)} \eqdef \min_{X\in\bbR^{2m\times n}} \sfrac{1}{2}\norm{X-Y}^2 + \iota_{\calX}(X)  .
\]
As a result, we can simplify \eqref{eq:BP} to the case of only $Y$,
\beq\label{eq:BP-ME}
    \boxed{\textstyle \min_{Y \in \bbR^{2m\times n}} \iota_{\calY}(Y) + \, ^1\Pa{\iota_{\calX}(Y)}.}
\eeq
For the rest of the paper, we focus on \eqref{eq:BP-ME} and our main contributions are summarised below:
\begin{itemize}[leftmargin=2em]
    \item \textbf{New formulation and a new algorithm for non-convex PCP.}
  We reformulate the non-convex set feasibility formulation of RPCA to a {\em best pair} problem. Although our formulation was inspired by formulation~\eqref{eq:ncvx-pcp} from \citep{duttahanzely}, to the best of our knowledge, we are the first to formulate and solve RPCA via the best pair. To this end, we design a fast and efficient algorithm---an accelerated proximal gradient method---to solve it. 
    
   \item \textbf{Theoretical convergence guarantees.} Both global and local convergence analysis of the scheme are provided. Globally, we show that our algorithm converges to a critical point. If the algorithm additionally starts sufficiently close to the optimum, we show that it converges to a global minimizer. Locally, our algorithm enjoys a fast linear rate, which we can sharply estimate. We owe this novelty to our best pair formulation. In contrast, the non-convex projection RPCA from \citep{duttahanzely} or GoDec~ \citep{godec} can only guarantee a local linear convergence. 
   
\item \textbf{Numerical experiments and applications to real-world problems.}
We apply the proposed method to several well-tested applications in computer vision. Our extensive experiments on both real and synthetic data suggest that our algorithm matches or outperforms relevant baseline algorithms in {\em fractions} of their execution time. Additionally, in the supplementary material, we provide empirical validity of the hyperparameters sensitivity of our approach. 
\end{itemize}

%
%

\subsection{Notations}
Throughout the paper, $\bbN$ is the set of non-negative integers. 
For a nonempty closed convex set $\Omega \subset \RR^n$, 
denote $\proj_{\Omega}$ the orthogonal projector onto $\Omega$. 
Let $\calR: \bbR^n \to \bbR \cup \ba{\pinf}$ be a lower semi-continuous (lsc) function, its domain is defined as $\dom(\calR) \eqdef \ba{x\in\bbR^{n} : \calR(x) < \pinf}$, and it is said to be proper if $\dom(\calR) \neq \emptyset$. 
We need the following notions from variational analysis, see e.g. \citep{rockafellar1998variational} for details. Given $x\in \dom(\calR)$, the Fr\'echet subdifferential $\partial^{F} \calR(x)$ of $\calR$ at $x$, is the set of vectors $v\in\bbR^{n}$ that satisfies $\liminf_{z\to x ,\,  z \neq x} \frac{1}{\norm{x-z}} \pa{\calR(z) - \calR(x) - \iprod{v}{z-x}} \geq 0$. 
If $x \notin \dom(\calR)$, then $\partial^{F} \calR(x) = \emptyset$. 
The limiting-subdifferential (or simply subdifferential) of $\calR$ at $x$, written as $\partial \calR(x)$, is defined as $\partial \calR(x) \eqdef \ba{ v \in \bbR^{n} : \exists \xk \to x , \calR(\xk) \to \calR(x) ,   \vk \in \partial^{F} \calR(\xk) \to v }$. 
%
Denote $\dom(\partial \calR) \eqdef \ba{x\in\bbR^n: \partial \calR(x) \neq \emptyset}$. Both $\partial^{F} \calR(x)$ and $\partial \calR(x$) are closed, with $\partial^{F} \calR(x)$ convex and $\partial^{F} \calR(x) \subset \partial \calR(x)$ \citep[Proposition~8.5]{rockafellar1998variational}. Since $\calR$ is lsc, it is (subdifferentially) regular at $x$ if and only if $\partial^{F} \calR(x) = \partial \calR(x)$ \citep[Corollary~8.11]{rockafellar1998variational}.
%
%
A necessary condition for $x$ to be a minimizer of $\calR$ is $0 \in \partial \calR(x)$. The set of critical points of $\calR$ is $\crit(\calR)=\ba{x \in \bbR^n: 0 \in \partial \calR(x)}$.

\vspace{-5pt}
\section{An accelerated proximal gradient method}\label{sec:algorithm}

In this section, we describe a gradient-based optimization method for solving \eqref{eq:BP-ME}. Denote $\proj_{\calX}, \proj_{\calY}$ the projection operators onto $\calX$ and $\calY$, respectively. 
Since $\calX$ is a non-empty closed convex set, its characteristic function $\iota_{\calX}$ is proper convex and lower semi-continuous. Owing to \citep{bauschke2011convex}, the Moreau envelope is convex differentiable with gradient reads
\[
\nabla \Pa{ ^1\pa{\iota_{\calX}(Y)} } = \pa{ \Id - \proj_{\calX} } (Y) 
\]
which is $1$-Lipschitz continuous. 
Clearly, \eqref{eq:BP-ME} admits a ``non-smooth + smooth'' structure, and in literature one prevailing algorithm to apply is the proximal gradient method \citep{lions1979splitting}, a.k.a. Forward--Backward splitting. In this paper, we consider an accelerated version of the method, see Algorithm~\ref{alg:apg}, which is based on inertial technique.

\begin{algorithm}[h]
\caption{An accelerated proximal gradient method} \label{alg:apg}
{\noindent{\bf{Initial}}}: Let $\gamma \in ]0, 2]$ and choose $Y_{0} \in \bbR^{2m \times n}, Y_{-1} = Y_{0}$. \\
\Repeat{convergence}{
\beq \label{eq:apg}
\begin{gathered}
\Zak = \Yk + \ak\pa{\Yk-\Ykm}    ,  \\
\Zbk = \Yk + \bk\pa{\Yk-\Ykm}    ,  \\
\Ykp = \proj_{\calY}\Pa{\Zak - \gamma \pa{\Zbk - \proj_{\calX}(\Zbk)}} .
\end{gathered}
\eeq
$k = k + 1$\;
}
\end{algorithm}


\begin{remark}$~$
    \begin{itemize}[leftmargin=2em]
        \item If we choose $\gamma = 1$ and $\ak, \bk \equiv 0$, Algorithm \ref{alg:apg} becomes the Backward--Backward splitting, which is the method of alternating projections for the considered feasibility problem \eqref{eq:ncvx-pcp}. Therefore, we recover the method from~\citep{duttahanzely} as a special case. 
        
        \item From \eqref{eq:ncvx-pcp} to \eqref{eq:BP-ME}, we can also consider the Moreau envelope of the non-convex set $\calY$, that is
        \[
    \textstyle \min_{X \in \bbR^{2m\times n}} \iota_{\calX}(X) + \, ^1\Pa{\iota_{\calY}(X)}  ,
        \]
        which also works well in practice. 
        
        
        
        \item Algorithm \ref{alg:apg} is a special cases of the multi-step inertial proximal gradient descent method considered in \citep{liang2016multi} for general non-convex composite optimization. 
        
    \end{itemize}
\end{remark}

Note that the two projection operators $\proj_{\calX}, \proj_{\calY}$ are very easy to compute. Given $X = \BPa{ {\footnotesize \begin{matrix} S \\ L  \end{matrix} } } $, since $\calX$ is an affine subspace, the projection of $X$ onto $\calX$ reads
$\proj_{\calX}(X) 
= 
\qfrac{1}{2} \BPa{ {\footnotesize \begin{matrix} A + S - L \\ A - S + L \end{matrix} } } $. 
If $K = [\proj_{\Omega}, \proj_{\Omega}]$ where $\proj_{\Omega}$ is the binary mask defined in \eqref{eq:binary}, then for the partial observed case, we have
\[
\proj_{\calX}(X) 
= \BPa{ {\footnotesize \begin{matrix} S \\ L  \end{matrix} } } + 
\qfrac{1}{2} \BPa{ {\footnotesize \begin{matrix} \proj_{\Omega} [ A - S - L ] \\ \proj_{\Omega} [ A - S - L ] \end{matrix} } }  .
\]
For the projection $\proj_{\calY}$ which contains a low-rank projection and sparsity projection, we refer to \citep{duttahanzely} for more details.

\subsection{Global convergence}

Since set $\calY$ is semi-algebraic \citep{bolte2010characterizations}, 
our global convergence guarantees of Algorithm \ref{alg:apg} is based on \KL property. 

\textbf{\KL property.}~~ 
Let $R : \bbR^{n} \to \bbR \cup \ba{\pinf}$ be a proper lsc function. For $\eta_1, \eta_2$ such that $\ninf < \eta_1 < \eta_2 < \pinf$, define the set 
 \[
 [ \eta_1< R < \eta_2 ] \eqdef \ba{ Y \in \bbR^{n} : \eta_1 < R(Y) < \eta_2}  .
 \]

\begin{definition}\label{defn:KLp}
Function $R$ is said to have the \KL property at $\Ybar \in \dom(R)$ if there exists $\eta \in ]0, \pinf]$, a neighbourhood $U$ of $\Ybar$ and a continuous concave function $\varphi:[0, \eta[ \to \bbR_{+}$ such that
\begin{enumerate}[label={\rm (\roman{*})}, leftmargin=3em]
\item
$\varphi(0) = 0$, $\varphi$ is $C^1$ on $]0, \eta[$, and for all $s \in ]0, \eta[$, $\varphi'(s) > 0$;
\item
for all $Y \in U\cap [R(\Ybar)<R<R(\Ybar) + \eta]$, the \KL inequality holds
\beq\label{eq:KLi}
\varphi'\Pa{R(Y)-R(\Ybar)}\dist\Pa{0, \partial R(Y)} \geq 1 .
\eeq
\end{enumerate}
Proper lsc functions which satisfy the \KL property at each point of $\dom(\partial R)$ are called KL functions.
\end{definition}

KL functions include the class of semi-algebraic functions, see \citep{Bolte07,bolte2010characterizations}. For instance, the $\ell_0$ pseudo-norm and the rank function are KL. 

\textbf{Global convergence.}~~ 
To deliver the convergence result, we rewrite \eqref{eq:BP-ME} into the following generic form
\beq\label{eq:R-plus-F}
\textstyle \min_{Y\in\bbR^{2m\times n}} \Ba{\Phi(Y) \eqdef \calR(Y) + \calF(Y)}  ,
\eeq
where we assume that
\begin{enumerate}[label= ({\textbf{A.\arabic{*}})},ref= \textbf{A.\arabic{*}}, leftmargin=3.75em]
\item\label{item:A-R}
$\calR: \bbR^n \to \bbR\cup\ba{\pinf}$ is proper lower semi-continuous, {and bounded from below};
\item\label{item:A-F}
$\calF: \bbR^{n} \to \bbR$ is convex differentiable and its gradient $\nabla \calF$ is $L$-Lipschitz continuous. 
\end{enumerate}

Let $\nu > 0$ be a constant. Define the following quantities,
\beq\label{eq:beta-alphai}
\beta_{k} \eqdef \sfrac{1 - \gamma L - \ak - \nu}{2\gamma}   ,~~ 
\ubeta \eqdef \liminf_{\kinN} \beta_{k} 
\qandq
\alpha_{k} \eqdef \sfrac{\gamma\bk^2 L^2 + \nu\ak}{2\nu\gamma} ,~~ 
\oalpha \eqdef \limsup_{\kinN} \alpha_{k}  .
\eeq

\begin{theorem}[Global convergence]\label{thm:convergence-apg}
For problem \eqref{eq:R-plus-F}, assume \iref{item:A-R}-\iref{item:A-F} hold, and that $\Phi$ is a proper lsc KL function which is bounded from below.
For Algorithm \ref{alg:apg}, choose $\nu, \gamma, \ak, \bk$ such that
\beq\label{eq:ineq-parameters}
\delta \eqdef \ubeta - \oalpha > 0 .
\eeq
Then each bounded sequence $\sequence{\Yk}$ satisfies
\begin{enumerate}[label = {\rm (\roman{*})}, ref = {\rm \roman{*}}]
\item \label{it:aa1}
$\sequence{\Yk}$ has finite length, \ie $\sum_{\kinN} \norm{\Yk-\Ykm} < \pinf$;
\item \label{it:aa2}
There exists a critical point $\Ysol \in \crit(\Phi)$ such that $\lim_{k\to \infty} \Yk = \Ysol$.
\item \label{it:aa3}
{If $\Phi$ has the KL property at a global minimizer $\Ysol$, then starting sufficiently close from $\Ysol$, any sequence $\sequence{\Yk}$ converges to  a global minimum of $\Phi$ and satisfies \iref{it:aa1}.}
\end{enumerate}
\end{theorem}

The proof of the above theorem can be found in the supplementary material. We also refer to \citep{liang2016multi} and the reference therein for more results on non-convex proximal gradient method.


\subsection{Local linear convergence}

Now we turn to the local perspective and present a local linear convergence analysis for Algorithm \ref{alg:apg}. For the constraint set $\calY$ define in \eqref{eq:ncvx-pcp}, consider the following decomposition of it
\[
\calY_{L} \eqdef \BBa{ Y = \BPa{ {\footnotesize \begin{matrix} S \\ L  \end{matrix} } }  : \rank(L) \leq r}
\qandq
 \calY_{S} \eqdef \BBa{ Y = \BPa{ {\footnotesize \begin{matrix} S \\ L  \end{matrix} } } : \textrm{$S$ is $\alpha$-sparse} }  .
\]
For the sequence $\Yk$ generated by \eqref{eq:apg}, suppose $\Yk = \BPa{ {\footnotesize \begin{matrix} S_{k} \\ L_{k}  \end{matrix} } }$. It is immediate that $\rank(L_{k}) \leq r$ holds for all $k$. For $S_{k}$, though it is always $\alpha$-sparse, the locations of non-zero elements change along the course of iteration. 
In the following, we first show that after a finite number of iterations the locations of non-zero elements of $S_{k}$ stop changing, that is $S_{k}$ will have the same support as that of $S^\star$ to which $S_{k}$ converges, and then Algorithm \ref{alg:apg} enters a linear convergence regime.

\textbf{Support identification of $S_{k}$.}~~ 
Let $\Ysol = \BPa{ {\footnotesize \begin{matrix} S^\star \\ L^\star  \end{matrix} } }$ be a critical point of \eqref{eq:BP-ME} to which $\Yk$ converges. Let $\calS$ be the subspace extended by the support of $S^\star$. Clearly, $S^\star \in \calS$ and we have the result below concerning the relation between $S_{k}$ and $\calS$.

\begin{theorem}[Support identification]\label{thm:supp-iden}
For Algorithm \ref{alg:apg}, suppose Theorem \ref{thm:convergence-apg} holds. Then $\Yk$ converges to a critical point $Y^\star$ of \eqref{eq:BP-ME}. For all $k$ large enough, we have $S_{k}  \in \calS$. 
%
\end{theorem}

Let $S^\star$ be the point that $S_{k}$ converges to, the above result simply means that after finite number of iterations, $\supp(S_{k}) = \supp(S^\star)$ holds for all $k$ large enough.  


\textbf{Local linear convergence.}~~ 
Given a critical point $\Ysol$, let $\Xsol = \proj_{\calX}(\Ysol)$, we have 
\[
\Xsol \in \calX 
\qandq
S^\star \in \calS ,~~~ L^\star \in \calY_{L}  .
\]
Note that the first two sets, $\calX, \calS$ are (affine) subspaces, hence smooth, and $\calY_{L}$ is the set of fixed-rank matrices which is $C^2$-smooth manifold \citep{lee2003smooth}. 
To derive the local linear rate, we need to utilize the smoothness of these sets. 
Let $\calM$ be a $C^2$-smooth manifold and let $\tanSp{\calM}{X}$ the tangent space of $\calM$ at $X \in \calM$, we have the following lemma which is crucial for our local linear convergence analysis. 

\begin{lemma}[{\citep[Lemma 5.1]{liang2014local}}]\label{lem:proj-M}
Let $\calM$ be a $C^2$-smooth manifold around $X$. Then for any $X' \in \calM \cap \calN$, where $\calN$ is a neighbourhood of $X$, the projection operator $\proj_{\calM}(X')$ is uniquely valued and $C^1$ around $X$, and thus $X' - X = \proj_{\tanSp{\calM}{X}}(X'-X) + o\pa{\norm{X'-X}}$. 
If moreover, $\calM = X + \tanSp{\calM}{X}$ is an affine subspace, then $X' - X = \proj_{\tanSp{\calM}{X}}(X'-X)$.
\end{lemma}

Denote the tangent spaces of $\calX, \calY$ at $\Xsol, \Ysol$ as $T_{\calX}^{\Xsol}$ and $T_{\calY}^{\Ysol}$, respectively. We refer to the supplementary material for detailed expressions of these tangent spaces. 
Denote $\proj_{T_{\calX}^{\Xsol}}$ and $\proj_{T_{\calY}^{\Ysol}}$ the projections onto the tangent spaces. Define the matrix $\calP \eqdef  \proj_{T_{\calY}^{\Ysol}} \Pa{(1 - \gamma)\Id + \gamma \proj_{T_{\calX}^{\Xsol}} } \proj_{T_{\calY}^{\Ysol}} $, and 
\[
D_{k} \eqdef \begin{pmatrix} \Yk - Y^\star \\ \Ykm - Y^\star \end{pmatrix}
\qandq
\calQ
\eqdef \begin{bmatrix} (1+a)\calP & -a\calP \\ \Id & 0  \end{bmatrix} \enskip \textrm{with} \enskip a \in [0, 1]  .
\]
Denote $\rho_{_{\calP}}, \rho_{_{\calQ}}$ the spectral radiuses of $\calP, \calQ$, respectively.

\begin{theorem}[Local linear convergence]\label{thm:local-rate}
For Algorithm \ref{alg:apg}, suppose Theorem \ref{thm:supp-iden} holds. Then $\Yk$ converges to a critical point $Y^\star$ of \eqref{eq:BP-ME}. 
Suppose $\bk=\ak\equiv a \in [0, 1]$, there exists a $K > 0$ such that for all $k \geq K$, 
\beqn
D_{k+1}
= \calQ D_{k} + o(\norm{D_{k}})  .
\eeqn
Moreover, if $\rho_{_{\calP}} < 1$, then so is $\rho_{_{\calQ}}$, and for all $k$ large enough we have $\norm{\Yk-\Ysol} = O(\rho_{_{\calQ}}^k)$. 
\end{theorem}
\begin{remark} $~$
\begin{itemize}[leftmargin=2em]
\item 
If $T_{\calX}^{\Xsol} \cap T_{\calY}^{\Ysol} = \ba{0}$, then it can be shown that $\rho_{_{\calP}} < 1$. 
\item
Given $\rho_{_{\calP}}$, $\rho_{_{\calQ}}$ can be expressed explicitly in terms of $a$ and $\rho_{_{\calP}}$. 
For the case that $\ak \to a \in [0, 1]$ and $\bk \to b \in [0, 1]$, we refer to \citep[Chapter 6]{liang2016convergence} for detailed discussion on the local linear convergence analysis. 
\end{itemize}
\end{remark}
An numerical illustration on our theoretical rate estimation and practical observation is provided in the supplementary material Section \ref{sec:proof-local}-Figure \ref{fig:rate}. 
 \begin{figure}[!ht]
    \centering
    \begin{minipage}{0.32\textwidth}
    \includegraphics[width=\textwidth]{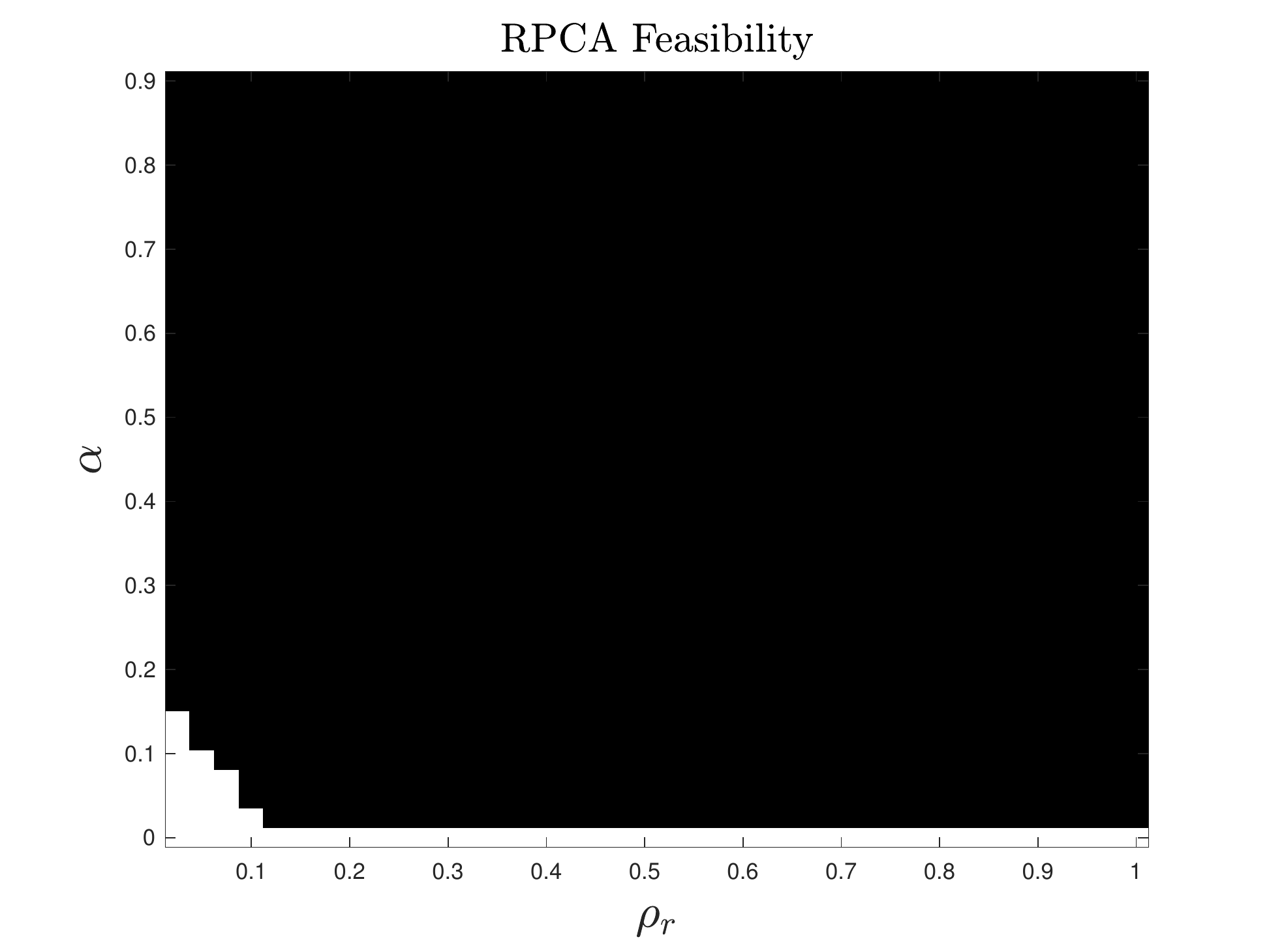}
      \end{minipage}
    \begin{minipage}{0.32\textwidth}
    \includegraphics[width = \textwidth]{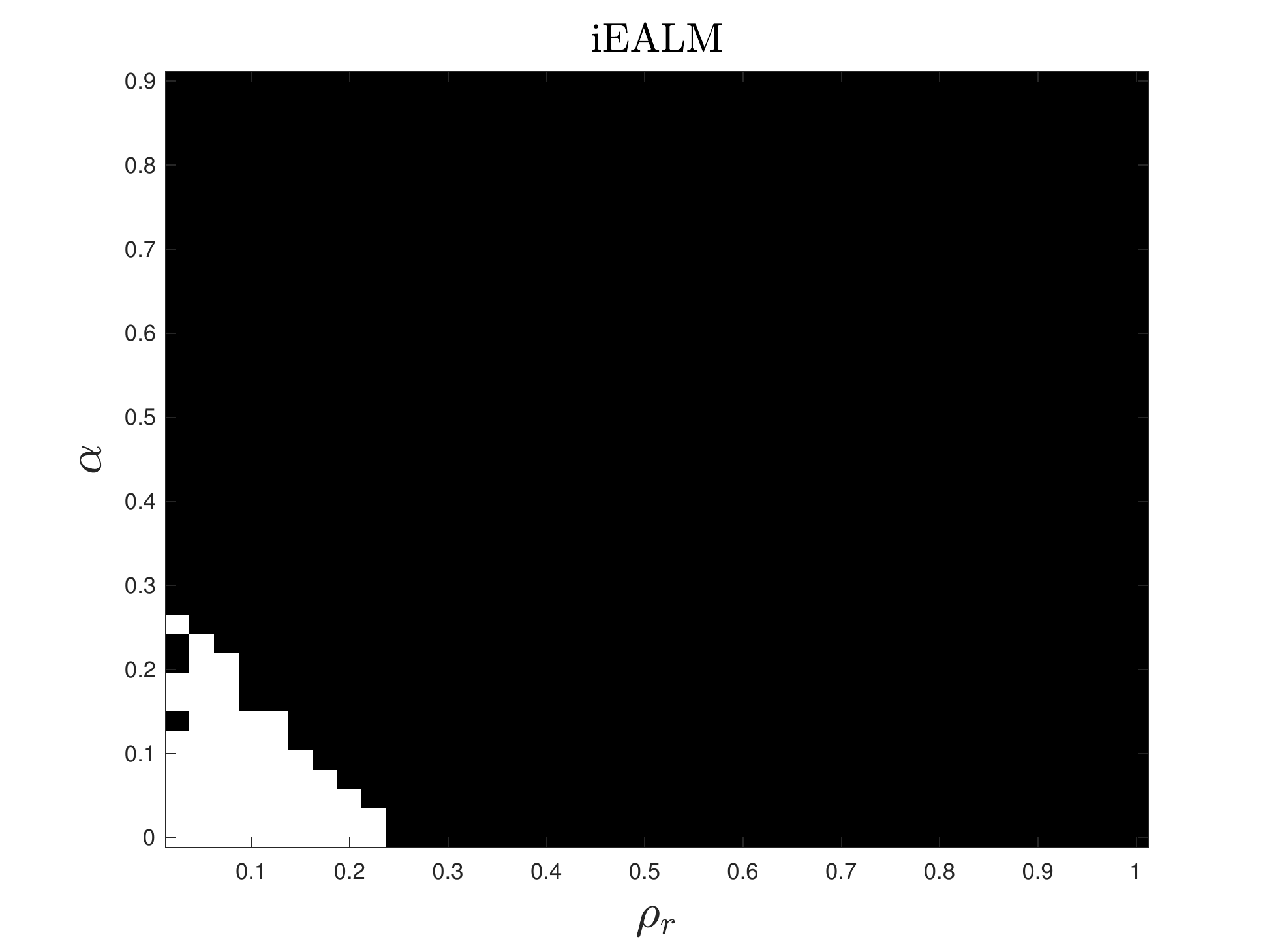}
  \end{minipage} 
   \begin{minipage}{0.32\textwidth}
    \includegraphics[width = \textwidth]{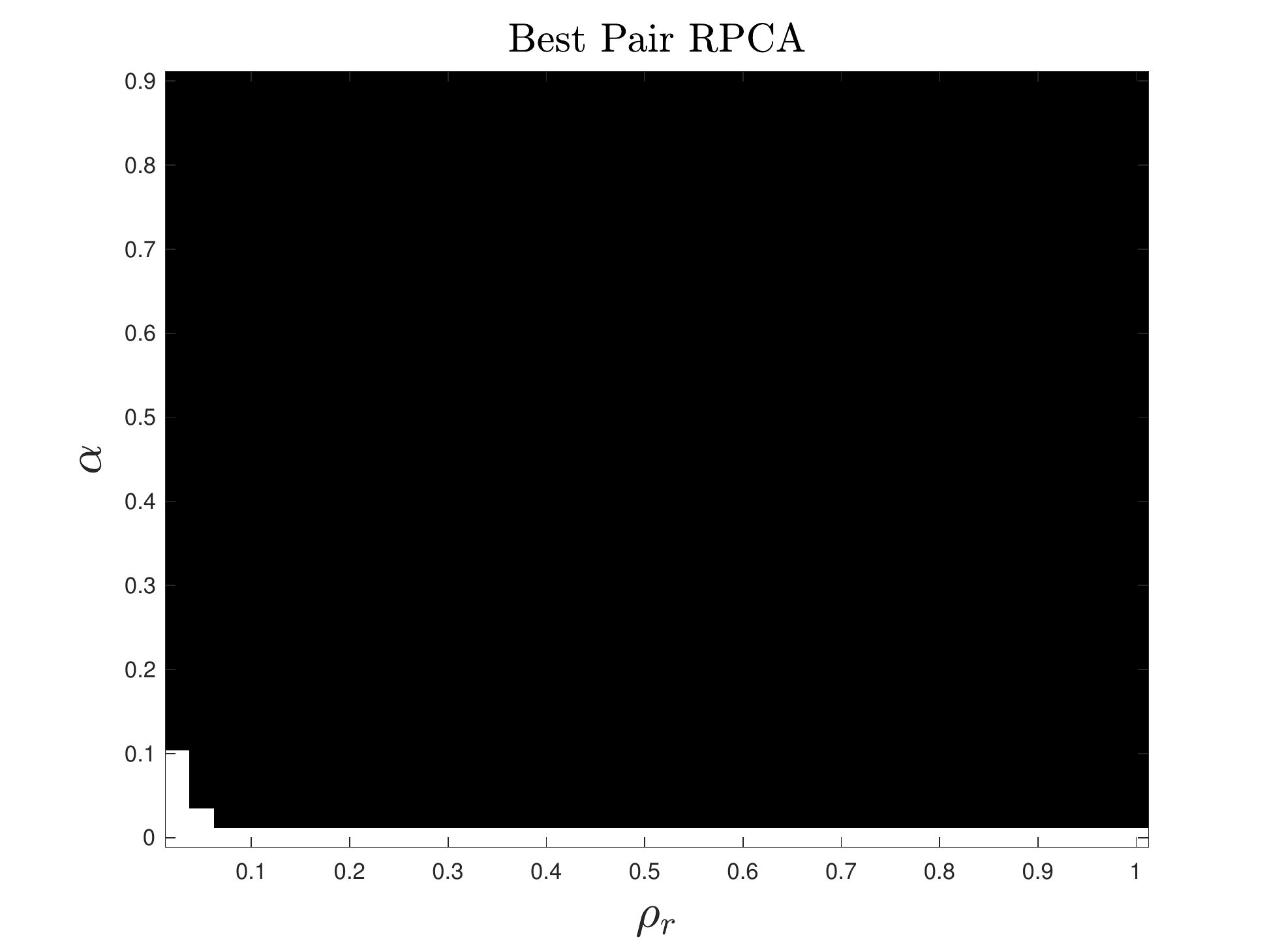}
  \end{minipage} 
\vspace{-0.1in}
\caption{\small{Phase transition diagram for RPCA F, iEALM, and APG with respect to rank and error sparsity. Here, $\rho_r={\rm rank}(L)/m$ and $\alpha$ is the sparsity measure. We have $(\rho_r,\alpha)\in (0.025,1]\times(0,1)$ with $r=5:5:200$ and $\alpha = {\tt linspace}(0,0.99,40)$. We perform 5 runs of each algorithm.}}
    \label{syntheticdata1}
\end{figure}

\vspace{-10pt}
\section{Numerical experiments}
In this section, we extensively tested our best-pair formulation on both real and synthetic data against a vast genre of PCP algorithms. The first set of algorithms that we tested against, e.g.\ iEALM and APG, determine the target rank and sparsity {\em robustly} from the given set of hyperparameters. On the other hand, for the second set of algorithms, e.g. RPCA gradient descent (RPCA GD),  Go decomposition (GoDec), and RPCA nonconvex feasibility (RPCA NCF), the target rank and sparsity are user-inferred. Although our accelerated proximal gradient algorithm  belongs to the second class, to show its effectiveness, we compare it with both classes of state-of-the-art robust PCP algorithms (see Table \ref{algo} in the supplementary material) on several computer vision applications---removal of shadows and specularities from face images, Background estimation or tracking from video sequences, and inlier detection from a grossly corrupted dataset~(see Section \ref{sec:inlier detection} in the supplementary material)\footnote{In all experiments, we use the approximate projection \citep{duttahanzely, RPCAgd, zhangpca} onto $\calY$ as the exact one is expensive: 
\begin{equation*}\label{eq:T_def}
 \mathcal{T}_{\alpha}[S] 
 \eqdef \{\proj_{\Omega_{\alpha}}(S)\in\mathbb{R}^{m\times n}: (i,j)\in\Omega_{\alpha}\;{\rm if}\;|S_{ij}|\ge|S_{(i,.)}^{(\alpha n)}|\;\;{\rm and}\; |S_{ij}|\ge|S_{(.,j)}^{(\alpha m)}|\}.
 \end{equation*} If the sparsity constraint was defined only along rowsc (or only columns), the exact projection would be cheap. However, the approximate projection produces better results, thus we stick with it.}. 

\paragraph{Results on synthetic data.}
The primary goal of these set of experiments is to understand the behavior of our proposed method on some well-understood data and to test against some state-of-the-art algorithms. To construct our test matrix $A$, for these experiments, we used the idea proposed by Wright et al.~\citep{APG}. First, we generate the low-rank matrix, $L$, as a product of two independent full-rank matrices of size $m\times r$ with $r<m$  such that elements are independent and identically distributed (i.i.d.) and sampled from a normal distribution---$\calN(0, 1)$. We generate the sparse matrix, $S$, such that its elements are chosen from the interval $[-500,500]$. We create the sparse support set by using the operator \eqref{eq:T_def}. Finally, we write $A$ as $A = L+S$. We fix $m=200$ and define $\rho_r={\rm rank}(L)/m$, where ${\rm rank}(L)$ varies. We choose the sparsity level $\alpha\in(0,1)$. 

{\em Phase transition experiments.} For each pair of $(\rho_r,\alpha)$, we apply iEALM,~RPCA NCF, and our algorithm to recover the pair $(\hat{L},\hat{S})$. 
For iEALM, we set $\lambda=1/\sqrt{m}$ and use $\mu=1.25/\|A\|_2$ and $\rho=1.5$, where $\|A\|_2$ is the spectral norm (maximum singular value) of $A$. For a given $\epsilon>0$, if the recovered matrix pair $(\hat{L},\hat{S})$, satisfies the relative error $\tfrac{\|A-\hat{L}-\hat{S}\|_F}{\|{A}\|_F}<\epsilon$ then we consider the construction is viable. In Figure \ref{syntheticdata1}, we produce the {\em phase transition diagrams} to show the fraction of perfect recovery of $A$, where white denotes {\it success} and black denotes {\it failure}. We run the experiments for 5 times and plot the results. The success of iEALM is approximately below the line $\rho_r+\alpha\approx 0.25.$ On the other hand, we note that the performance of our best pair RPCA is almost similar to that of \citep{duttahanzely}, when the sparsity level $\alpha$ is small and both approaches can efficiently provide a feasible reconstruction for any $\rho_r$ in that case. We also note that for low sparsity level, iEALM can only provide a feasible reconstruction for $\rho_r\le0.25$. Due to their robustness to any low-rank structure when $\alpha$ is low, RPCA NCF and best pair RPCA can be proved to be very effective in many real-world applications. In many real-world problems, involving the video/image data can ideally have any inherent low-rank structure and are generally corrupted by very sparse outliers of arbitrary large magnitudes. In those instances, RPCA NCF and our best pair RPCA could be very useful. We show more justification in the later section. 
\vspace{-1pt}

{\em Root mean square error measure.} To validate our performance against RPCA GD of Yi et al.\ \citep{RPCAgd}, we use a different metric---root mean square error (RMSE). Since RPCA GD does not explicitly recover a sparse matrix, $S$, it is unjustified to test it against the same relative error. Therefore, for the true low-rank,~$L$, and a low-rank recovery, $\hat{L}$, we use the metric $\nicefrac{\|L-\hat{L}\|_F}{\sqrt{mn}}$ as the measure of RMSE. From Figure \ref{syntheticdata2}, we can conclude that our best pair RPCA has less RMSE compare to that of RPCA GD. Moreover, the RMSE remains unaltered as the cardinality of support set,~$\Omega$ increases. Also, see Figure \ref{syntheticdata3} in the Appendix. 

\begin{figure}
    \centering
    \begin{minipage}{0.32\textwidth}
    \includegraphics[width=\textwidth]{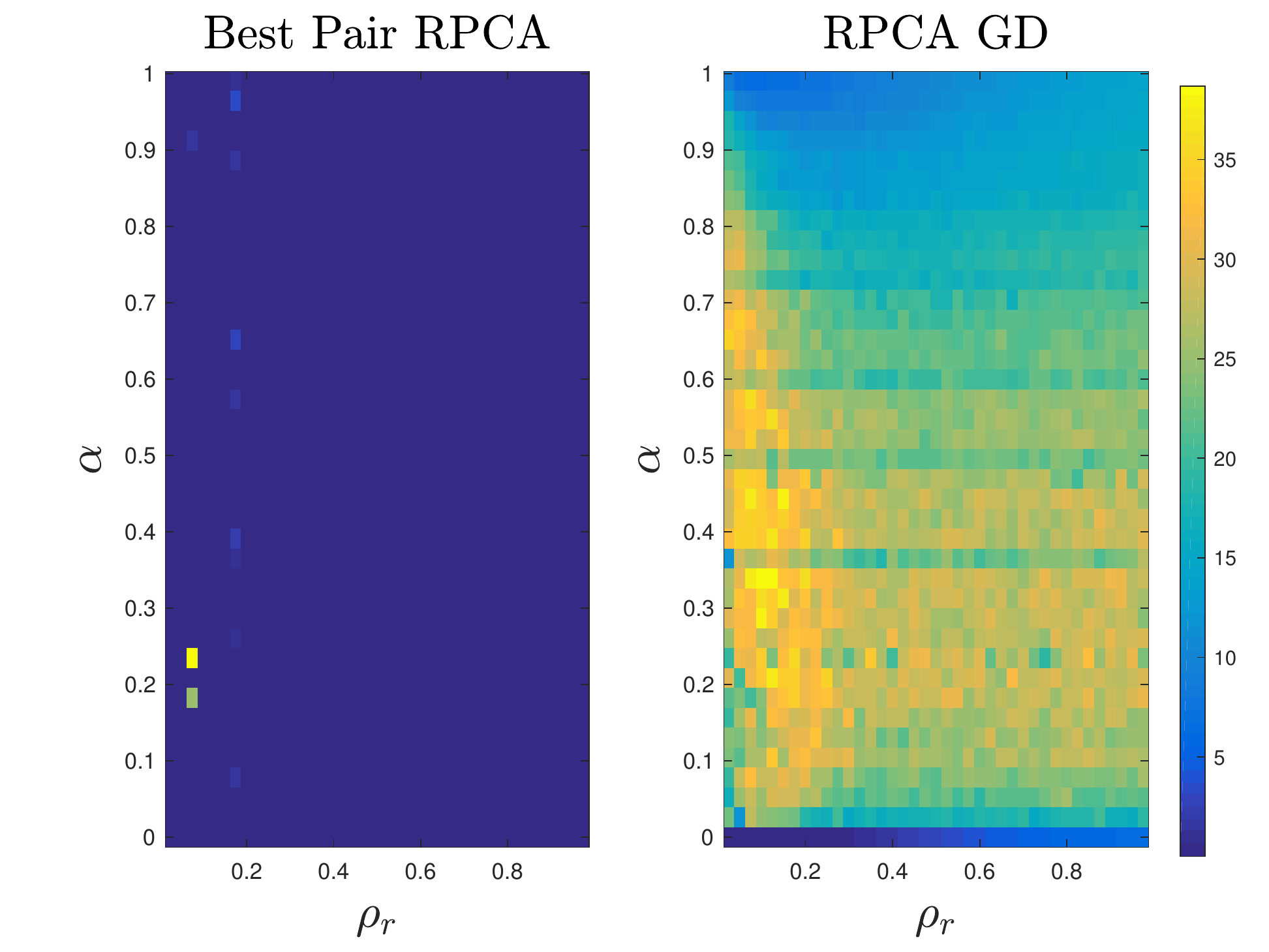}
      \end{minipage}
    \begin{minipage}{0.32\textwidth}
    \includegraphics[width = \textwidth]{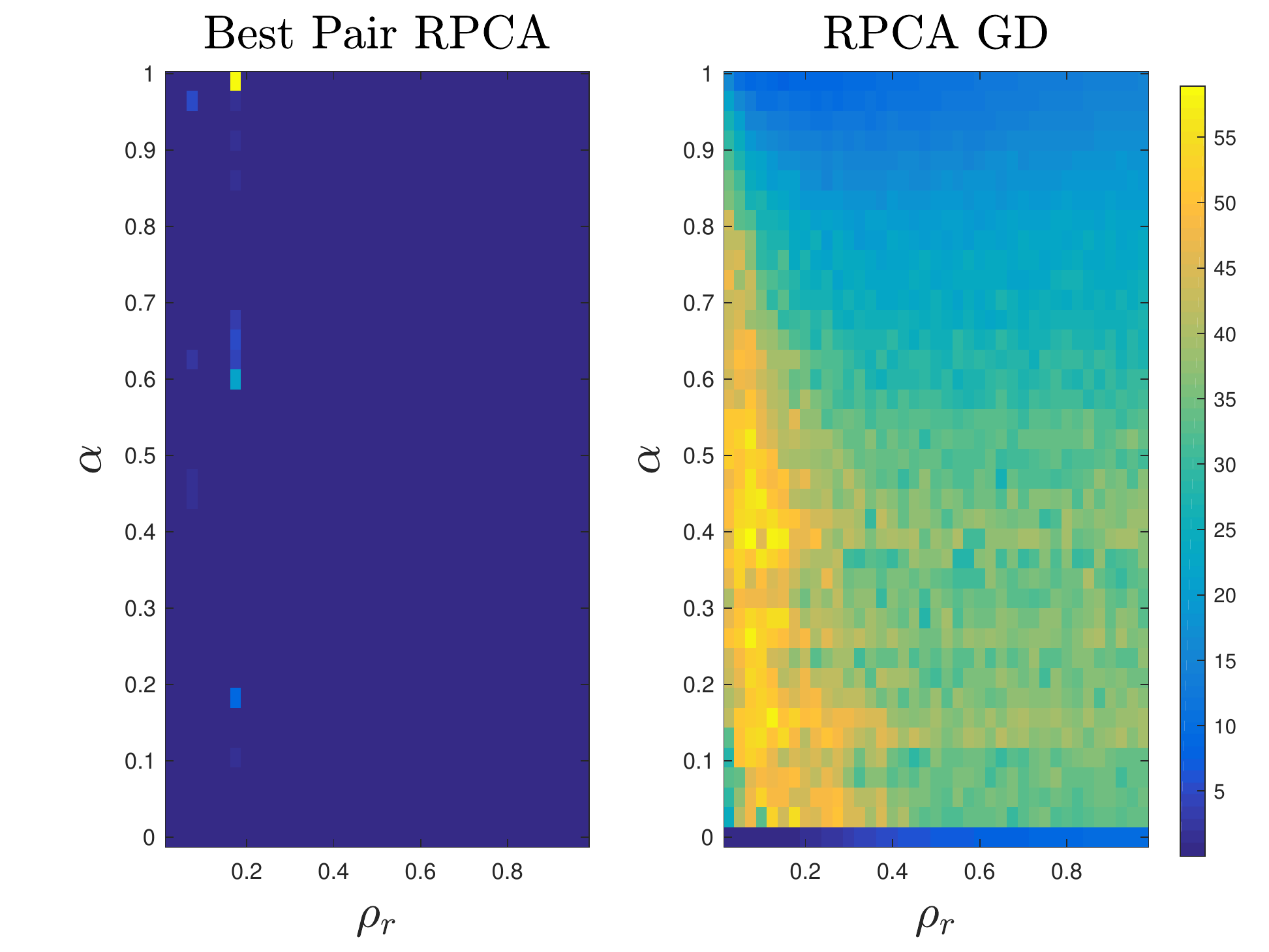}
  \end{minipage} 
   \begin{minipage}{0.32\textwidth}
   \includegraphics[width = \textwidth]{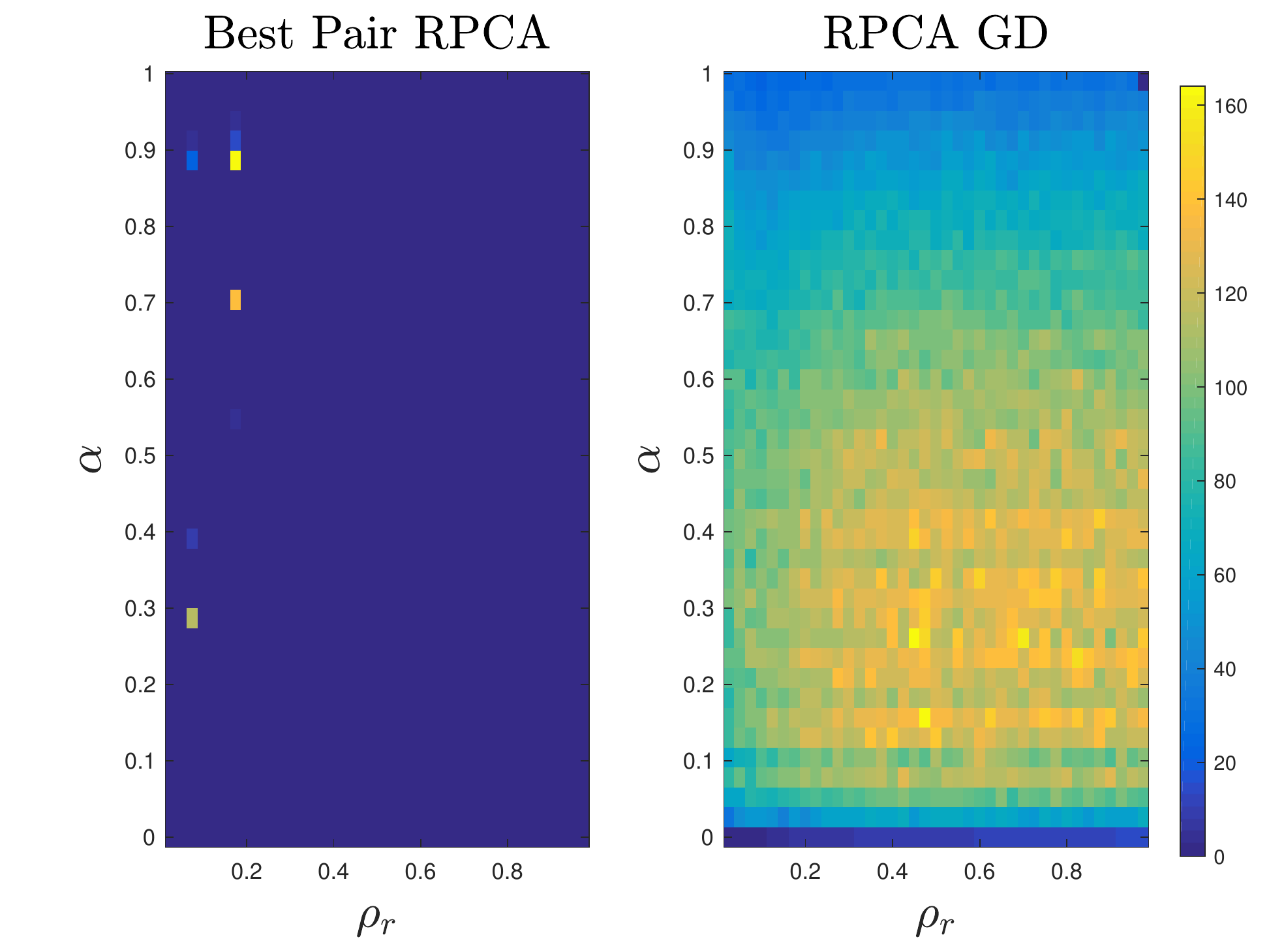}
  \end{minipage} 
\vspace{-0.1in}
\caption{\small{ RMSE to compare between RPCA GD and best-pair RPCA with respect to rank and error sparsity. We set $\rho_r={\rm rank}(L)/m$ and $\alpha$ is the sparsity measure. We have $(\rho_r,\alpha)\in (0.025,1]\times(0,1)$ with $r=5:5:200$ and $\alpha = {\tt linspace}(0,0.99,40)$.}}
    \label{syntheticdata2}
\end{figure}

\vspace{-5pt}
\paragraph{Removal of shadows and specularities.}
Set of images of an object under unknown pose and arbitrary lighting conditions, form a convex cone in the space of all possible images which may have {\em unbounded dimension} \citep{basri_jacobs,belhumeur1998set}. However, the images under distant, isotropic lighting can be approximated by a 9-dimensional linear subspace which is popularly referred to as the {\em harmonic plane}. We used three subjects {\tt B11,B12}, and {\tt B13} from the {\tt Extended Yale Face Database} \citep{yale_face} for our simulations. We used 63  downsampled images of resolution of $120\times 160$ of each subject. For APG and iEALM, we set the parameters the same as in the previous section. For RPCA GD, RPCA NCF, and our method, we set target rank $r = 9$ and sparsity level to $0.1$.~The qualitative analysis on the recovered images from Figure \ref{shadow_removal} shows while RPCA GD recovers patchy and granular face images, our best pair reformulation provides comparable reconstruction to that of iEALM, APG, and RPCA NCF. 
\begin{figure}
    \centering
    \includegraphics[width = \textwidth]{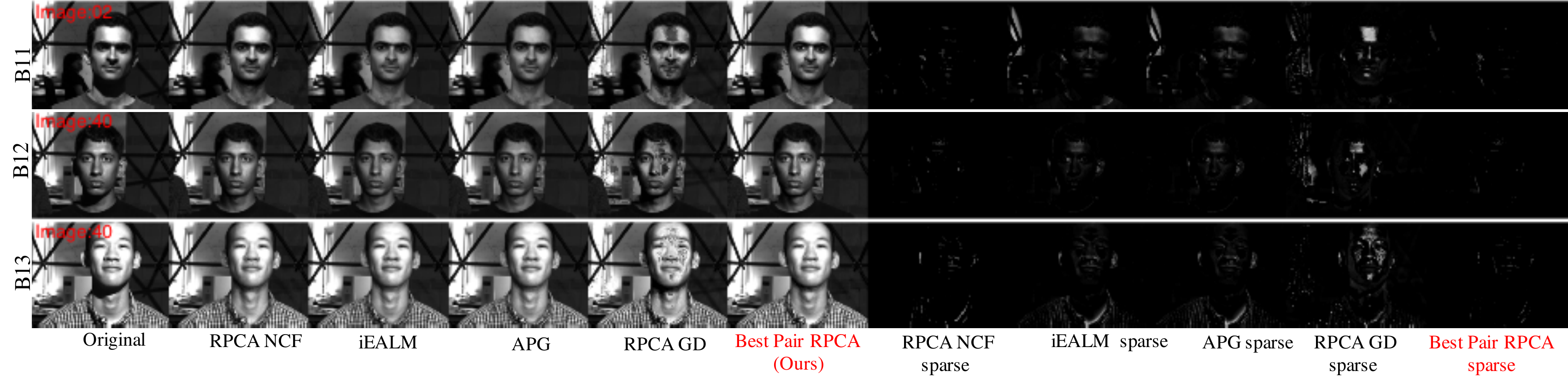}
    \vspace{-7pt}
\caption{\small{Shadow and specularities removal from face images captured under varying illumination and camera position. Our feasibility approach provides comparable reconstruction to that of iEALM, APG, and RPCA NCF.}}
    \label{shadow_removal}
\end{figure}
\begin{figure}
    \centering
    \includegraphics[width = \textwidth]{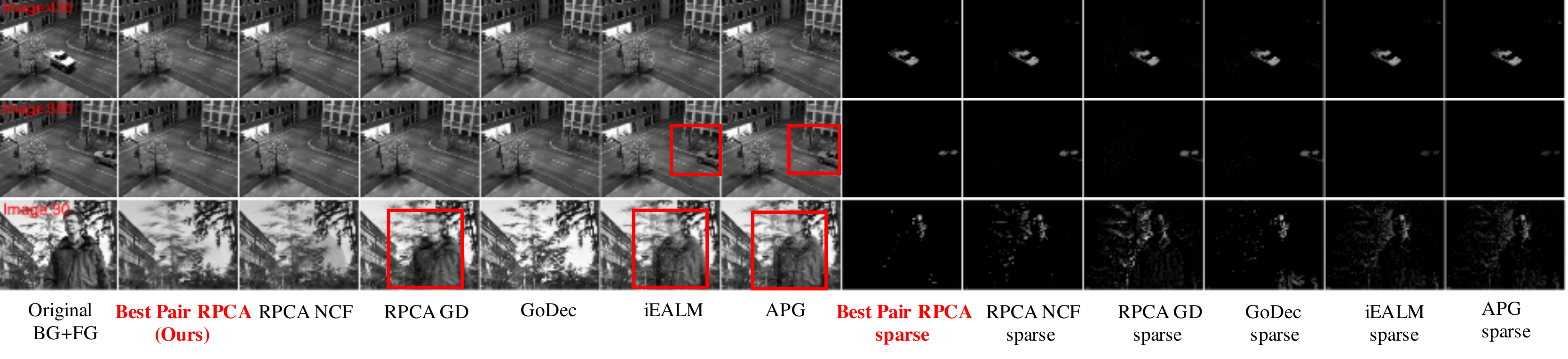}
    \vspace{-5pt}
    \caption{\small{Background estimation from video sequences. Except Best pair RPCA, RPCA NCF, and GoDEc all other methods struggle to remove the static foreground object. }}
    \label{BG_full}
\end{figure}
\vspace{-5pt}
\paragraph{Background estimation from video sequences.}
Background estimation or moving object tracking \citep{boumans_taxonomy,dutta_thesis, Bouwmans2016,dutta2017weighted, Bouwmans201431,inWLR,duttalirichtarik_modeling,dutta_online,duttali_bg} is considered as one of the classic problems in computer vision and is used as a crucial component in human activity recognition, tracking, and video analysis from surveillance cameras. When the video is captured by a static camera, minimizing the rank of the matrix $A\in\mathbb{R}^{m\times n}$, that concatenates $n$ video frames (after converting them into vectors) represents the structure of the linear subspace, $L$, that contains the background and an error, $S$, that emphasizes the foreground components. However, the exact desired rank is often tuned empirically, as the ideal rank-one background is often unrealistic as the changing illumination, occluded foreground/background objects, reflection, and noise are typically also a part of the video frames. Based on the above observation, we note that the problem can be cast typically as \eqref{rpca}. However, as we explained in some cases, when the target rank and the sparsity level is user-inferred hyperparameters, one might use a different approach as in \citep{godec,duttahanzely,RPCAgd} as well. Additionally, there might be missing/unobserved pixels in the video and that makes the problem more complex and only a few methods, such as RPCA NCF, GRASTA \citep{grasta}, RPCA GD remedy to that. Therefore, we tested our best pair RPCA to a wide range of methods. In our experiments, we use {\em two} different video sequences:~(i)~the {\tt Basic} sequence from Stuttgart synthetic dataset~\citep{cvpr11brutzer}, (ii)~the {\tt waving tree} sequence \citep{wallflower}. We extensively use the Stuttgart video sequence as it is a challenging sequence that comprises both static and dynamic foreground objects and varying illumination in the background. Additionally, it comes with foreground ground truth for each frame. For iEALM and APG, we set the parameters the same as in the previous sections. For Best pair RPCA, RPCA GD, RPCA NCF, and GoDec, we set $r=2$, target sparsity 10\% and additionally, for GoDec, we set $q=2$. For GRATSA, we set the parameters the same as those mentioned in the authors’ website \citep{grastacode}. The qualitative analysis on the background and foreground recovered on both, full observation (in Figure \ref{BG_full}) and partial observation (in Figure \ref{Bg_subsampled}), suggest that our method recovers a visually better quality background and foreground compare to the other methods. Note that, RPCA GD recovers a fragmentary foreground with more false positives compare to our method; moreover, RPCA GD, GRASTA, iEALM, and APG cannot remove the static foreground object. We provide a detailed quantitative evaluation of our best pair RPCA with respect to the $\epsilon$-proximity metric--$d_{\epsilon}(X,Y)$ as in \citep{duttahanzely} and the mean structural similarity index measure (SSIM) by \citep{mssim} in recovering the foreground objects in Figures \ref{qunatitative_BG} and \ref{qunatitative_BG_1} in Appendix. 
 \begin{figure}[!ht]
    \centering
    \includegraphics[width = \textwidth]{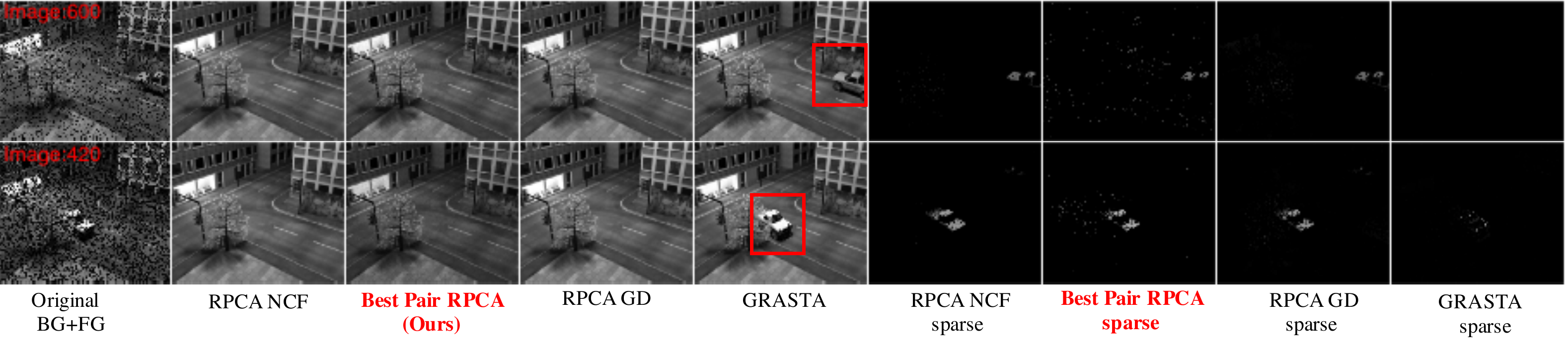}
    \vspace{-7pt}
    \caption{\small{Background estimation on subsampled Stuttgart  {\tt Basic} video sequence. We use $\Omega = 0.9(m.n)$ and $\Omega = 0.8(m.n)$, respectively.}}
    \label{Bg_subsampled}
\end{figure}


\bibliographystyle{plainnat}
{
	\small
	\bibliography{references}}

\begin{thebibliography}{58}
\providecommand{\natexlab}[1]{#1}
\providecommand{\url}[1]{\texttt{#1}}
\expandafter\ifx\csname urlstyle\endcsname\relax
  \providecommand{\doi}[1]{doi: #1}\else
  \providecommand{\doi}{doi: \begingroup \urlstyle{rm}\Url}\fi

\bibitem[gra(2012)]{grastacode}
2012.
\newblock https://sites.google.com/site/hejunzz/grasta.

\bibitem[Basri and Jacobs(2003)]{basri_jacobs}
R.~Basri and D.~Jacobs.
\newblock Lambertian reflection and linear subspaces.
\newblock \emph{IEEE Transaction on Pattern Analysis and Machine Intelligence},
  25\penalty0 (3):\penalty0 218--233, 2003.

\bibitem[Bauschke and Combettes(2011)]{bauschke2011convex}
H.~H. Bauschke and P.~L. Combettes.
\newblock \emph{Convex analysis and monotone operator theory in Hilbert
  spaces}.
\newblock Springer, 2011.

\bibitem[Belhumeur and Kriegman(1998)]{belhumeur1998set}
P.~N. Belhumeur and D.~J. Kriegman.
\newblock What is the set of images of an object under all possible
  illumination conditions?
\newblock \emph{International Journal of Computer Vision}, 28\penalty0
  (3):\penalty0 245--260, 1998.

\bibitem[Bolte et~al.(2007)Bolte, Daniilidis, and Lewis]{Bolte07}
J.~Bolte, A.~Daniilidis, and A.~Lewis.
\newblock The {{\L}}ojasiewicz inequality for nonsmooth subanalytic functions
  with applications to subgradient dynamical systems.
\newblock \emph{SIAM Journal on Optimization}, 17\penalty0 (4):\penalty0
  1205--1223, 2007.

\bibitem[Bolte et~al.(2010)Bolte, Daniilidis, Ley, and
  Mazet]{bolte2010characterizations}
J.~Bolte, A.~Daniilidis, O.~Ley, and L.~Mazet.
\newblock Characterizations of {{\L}}ojasiewicz inequalities: subgradient
  flows, talweg, convexity.
\newblock \emph{Transactions of the American Mathematical Society},
  362\penalty0 (6):\penalty0 3319--3363, 2010.

\bibitem[Bouwmans(2014)]{Bouwmans201431}
T.~Bouwmans.
\newblock Traditional and recent approaches in background mode
  foreground detection: An overview.
\newblock \emph{Computer Science Review}, 11--12:\penalty0 31 -- 66, 2014.

\bibitem[Bouwmans and Zahzah(2014)]{rpca_methods}
T.~Bouwmans and E.-H. Zahzah.
\newblock Robust {P}{C}{A} via principal component pursuit: {A} review for a
  comparative evaluation in video surveillance.
\newblock \emph{Computer Vision and Image Understanding}, 122:\penalty0 22--34,
  2014.

\bibitem[Bouwmans et~al.(2016)Bouwmans, Sobral, Javed, Jung, and
  Zahzah]{Bouwmans2016}
T.~Bouwmans, A.~Sobral, S.~Javed, S.~K. Jung, and E.-H. Zahzah.
\newblock Decomposition into low-rank plus additive matrices for
  background/foreground separation: A review for a comparative evaluation with
  a large-scale dataset.
\newblock \emph{Computer Science Review}, 2016.

\bibitem[Bouwmans et~al.(2017)Bouwmans, Maddalena, and
  Petrosino]{boumans_taxonomy}
T.~Bouwmans, L.~Maddalena, and A.~Petrosino.
\newblock Scene background initialization: A taxonomy.
\newblock \emph{Pattern Recognition Letters}, 96:\penalty0 3--11, 2017.

\bibitem[Brutzer et~al.(2011)Brutzer, Hoferlin, and Heidemann]{cvpr11brutzer}
S.~Brutzer, B.~Hoferlin, and G.~Heidemann.
\newblock Evaluation of background subtraction techniques for video
  surveillance.
\newblock In \emph{IEEE Computer Vision and Pattern Recognition}, pages
  1937--1944, 2011.

\bibitem[Cai et~al.(2010)Cai, Cand\`{e}s, and Shen]{caicandesshen}
J.~F. Cai, E.~J. Cand\`{e}s, and Z.~Shen.
\newblock A singular value thresholding algorithm for matrix completion.
\newblock \emph{SIAM Journal on Optimization}, 20\penalty0 (4):\penalty0
  1956--1982, 2010.

\bibitem[Cand\`{e}s and Plan(2009)]{candesplan}
E.~J. Cand\`{e}s and Y.~Plan.
\newblock Matrix completion with noise.
\newblock \emph{Proceedings of the IEEE}, 98\penalty0 (6):\penalty0 925--936,
  2009.

\bibitem[Cand\`{e}s and Recht(2009)]{candes_MC}
E.~J. Cand\`{e}s and B.~Recht.
\newblock Exact matrix completion via convex optimization.
\newblock \emph{Foundations of Computational Mathematics}, 9\penalty0
  (6):\penalty0 717--772, 2009.

\bibitem[Cand\`{e}s and Tao(2010)]{candes_MC2}
E.~J. Cand\`{e}s and T.~Tao.
\newblock The power of convex relaxation: Near-optimal matrix completion.
\newblock \emph{IEEE Transactions on Information Theory}, 56\penalty0
  (5):\penalty0 2053--2080, 2010.

\bibitem[Cand\`{e}s et~al.(2011)Cand\`{e}s, Li, Ma, and
  Wright]{candeslimawright}
E.~J. Cand\`{e}s, X.~Li, Y.~Ma, and J.~Wright.
\newblock Robust principal component analysis?
\newblock \emph{Journal of the Association for Computing Machinery},
  58\penalty0 (3):\penalty0 11:1--11:37, 2011.

\bibitem[Chandrasekaran et~al.(2011)Chandrasekaran, Sanghavi, Parrilo, and
  Willsky]{rpca_1}
V.~Chandrasekaran, S.~Sanghavi, P.~A. Parrilo, and A.~S. Willsky.
\newblock Rank-sparsity incoherence for matrix decomposition.
\newblock \emph{SIAM Journal on Optimization}, 21\penalty0 (2):\penalty0
  572--596, 2011.

\bibitem[Chen et~al.(2011)Chen, Xu, Caramanis, and Sanghavi]{Chen_RMC}
Y.~Chen, H.~Xu, C.~Caramanis, and S.~Sanghavi.
\newblock Robust matrix completion and corrupted columns.
\newblock In \emph{Proceedings of the 28th International Conference on
  International Conference on Machine Learning}, pages 873--880, 2011.

\bibitem[Cherapanamjeri et~al.(2017{\natexlab{a}})Cherapanamjeri, Gupta, and
  Jain]{CherapanamjeriGJ17}
Y.~Cherapanamjeri, K.~Gupta, and P.~Jain.
\newblock Nearly optimal robust matrix completion.
\newblock In \emph{Proceedings of the 34th International Conference on Machine
  Learning (ICML)}, pages 797--805, 2017{\natexlab{a}}.

\bibitem[Cherapanamjeri et~al.(2017{\natexlab{b}})Cherapanamjeri, Jain, and
  Netrapalli]{CherapanamjeriJN17}
Y.~Cherapanamjeri, P.~Jain, and P.~Netrapalli.
\newblock Thresholding based outlier robust {P}{C}{A}.
\newblock In \emph{Proceedings of the 30th Conference on Learning Theory
  (COLT)}, pages 593--628, 2017{\natexlab{b}}.

\bibitem[Drusvyatskiy and Lewis(2013)]{drusvyatskiy2013optimality}
D.~Drusvyatskiy and A.~S. Lewis.
\newblock Optimality, identifiability, and sensitivity.
\newblock \emph{Mathematical Programming}, pages 1--32, 2013.

\bibitem[Dutta(2016)]{dutta_thesis}
A.~Dutta.
\newblock \emph{Weighted Low-Rank Approximation of Matrices:Some Analytical and
  Numerical Aspects}.
\newblock PhD thesis, University of Central Florida, 2016.

\bibitem[Dutta and Li(2017)]{duttali_bg}
A.~Dutta and X.~Li.
\newblock Weighted low rank approximation for background estimation problems.
\newblock In \emph{The IEEE International Conference on Computer Vision
  Workshops (ICCVW)}, pages 1853--1861, 2017.

\bibitem[{Dutta} and {Richt\'{a}rik}(2019)]{dutta_online}
A.~{Dutta} and P.~{Richt\'{a}rik}.
\newblock Online and batch supervised background estimation via l1 regression.
\newblock In \emph{2019 IEEE Winter Conference on Applications of Computer
  Vision (WACV)}, pages 541--550, 2019.

\bibitem[Dutta et~al.(2017{\natexlab{a}})Dutta, Gong, Li, and
  Shah]{dutta2017weighted}
A.~Dutta, B.~Gong, X.~Li, and M.~Shah.
\newblock Weighted singular value thresholding and its application to
  background estimation.
\newblock \emph{arXiv:1707.00133}, 2017{\natexlab{a}}.

\bibitem[Dutta et~al.(2017{\natexlab{b}})Dutta, Li, and Richt\'{a}rik]{inWLR}
A.~Dutta, X.~Li, and P.~Richt\'{a}rik.
\newblock A batch-incremental video background estimation model using weighted
  low-rank approximation of matrices.
\newblock In \emph{The IEEE International Conference on Computer Vision
  Workshops (ICCVW)}, pages 1835--1843, 2017{\natexlab{b}}.

\bibitem[Dutta et~al.(2018{\natexlab{a}})Dutta, Hanzely, and
  Richt\'{a}rik]{duttahanzely}
A.~Dutta, F.~Hanzely, and P.~Richt\'{a}rik.
\newblock A nonconvex projection method for robust \uppercase{PCA},
  2018{\natexlab{a}}.
\newblock To appear in Thirty-Third AAAI Conference on Artificial Intelligence
  (AAAI-19), arXiv:1805.07962.

\bibitem[Dutta et~al.(2018{\natexlab{b}})Dutta, Li, and
  Richt\'{a}rik]{duttalirichtarik_modeling}
A.~Dutta, X.~Li, and P.~Richt\'{a}rik.
\newblock Weighted low-rank approximation of matrices and background modeling,
  2018{\natexlab{b}}.
\newblock arXiv:1804.06252.

\bibitem[Fei-Fei et~al.(2007)Fei-Fei, Fergus, and Perona]{caltech101}
L.~Fei-Fei, R.~Fergus, and P.~Perona.
\newblock Learning generative visual models from few training examples:
  \uppercase{A}n incremental \uppercase{B}ayesian approach tested on 101 object
  categories.
\newblock \emph{Computer Vision and Image Understanding}, 106\penalty0
  (1):\penalty0 59--70, 2007.

\bibitem[Georghiades et~al.(2001)Georghiades, Belhumeur, and
  Kriegman]{yale_face}
A.~Georghiades, P.~Belhumeur, and D.~Kriegman.
\newblock From few to many: Illumination cone models for face recognition under
  variable lighting and pose.
\newblock \emph{IEEE Transactions on PAMI}, 23\penalty0 (6):\penalty0 643--660,
  2001.

\bibitem[Goes et~al.(2014)Goes, Zhang, Arora, and Lerman]{rspca}
J.~Goes, T.~Zhang, R.~Arora, and G.~Lerman.
\newblock Robust stochastic principal component analysis.
\newblock In \emph{Proceedings of the 17th International Conference on Articial
  Intelligence and Statistics}, pages 266--274, 2014.

\bibitem[He et~al.(2012)He, Balzano, and Szlam]{grasta}
J.~He, L.~Balzano, and A.~Szlam.
\newblock Incremental gradient on the {G}rassmannian for online foreground and
  background separation in subsampled video.
\newblock \emph{IEEE Computer Vision and Pattern Recognition}, pages
  1937--1944, 2012.

\bibitem[Jain and Netrapalli(2015)]{JN15}
P.~Jain and P.~Netrapalli.
\newblock Fast exact matrix completion with finite samples.
\newblock In \emph{Proceedings of The 28th Conference on Learning Theory
  (COLT)}, pages 1007--1034, 2015.

\bibitem[Jain et~al.(2013)Jain, Netrapalli, and Sanghavi]{Jain}
P.~Jain, P.~Netrapalli, and S.~Sanghavi.
\newblock Low-rank matrix completion using alternating minimization.
\newblock In \emph{Proceedings of the Forty-fifth Annual ACM Symposium on
  Theory of Computing}, pages 665--674, 2013.

\bibitem[Keshavan et~al.(2010)Keshavan, Montanari, and Oh]{KeshavanMC}
R.~Keshavan, A.~Montanari, and S.~Oh.
\newblock Matrix completion from a few entries.
\newblock \emph{IEEE Transactions on Information Theory}, 56\penalty0
  (6):\penalty0 2980--2998, 2010.

\bibitem[Lee(2003)]{lee2003smooth}
J.~M. Lee.
\newblock \emph{Smooth manifolds}.
\newblock Springer, 2003.

\bibitem[Lewis(2003)]{LewisPartlySmooth}
A.~S. Lewis.
\newblock Active sets, non-smoothness, and sensitivity.
\newblock \emph{SIAM Journal on Optimization}, 13\penalty0 (3):\penalty0
  702--725, 2003.

\bibitem[Lewis and Zhang(2013)]{LewisPartlyTiltHessian}
A.~S. Lewis and S.~Zhang.
\newblock Partial smoothness, tilt stability, and generalized {H}essians.
\newblock \emph{SIAM Journal on Optimization}, 23\penalty0 (1):\penalty0
  74--94, 2013.

\bibitem[Liang(2016)]{liang2016convergence}
J.~Liang.
\newblock \emph{Convergence Rates of First-Order Operator Splitting Methods}.
\newblock PhD thesis, Normandie Universit{\'e}; GREYC CNRS UMR 6072, 2016.

\bibitem[Liang et~al.(2014)Liang, Fadili, and Peyr{\'e}]{liang2014local}
J.~Liang, J.~Fadili, and G.~Peyr{\'e}.
\newblock Local linear convergence of {F}orward--{B}ackward under partial
  smoothness.
\newblock In \emph{Advances in Neural Information Processing Systems}, pages
  1970--1978, 2014.

\bibitem[Liang et~al.(2016)Liang, Fadili, and Peyr{\'e}]{liang2016multi}
J.~Liang, J.~Fadili, and G.~Peyr{\'e}.
\newblock A multi-step inertial {F}orward--{B}ackward splitting method for
  non-convex optimization.
\newblock In \emph{Advances in Neural Information Processing Systems}, 2016.

\bibitem[Lin et~al.(2009)Lin, Ganesh, Wright, Wu, Chen, and Ma]{dual_rpca}
Z.~Lin, A.~Ganesh, J.~Wright, L.~Wu, M.~Chen, and Yi~Ma.
\newblock Fast convex optimization algorithms for exact recovery of a corrupted
  low-rank matrix.
\newblock \emph{UIUC Technical Report UILU-ENG-09-2214}, 2009.

\bibitem[Lin et~al.(2010)Lin, Chen, and Ma]{LinChenMa}
Z.~Lin, M.~Chen, and Y.~Ma.
\newblock The augmented {L}agrange multiplier method for exact recovery of
  corrupted low-rank matrices, 2010.
\newblock arXiv1009.5055.

\bibitem[Lions and Mercier(1979)]{lions1979splitting}
P.~L. Lions and B.~Mercier.
\newblock Splitting algorithms for the sum of two nonlinear operators.
\newblock \emph{SIAM Journal on Numerical Analysis}, 16\penalty0 (6):\penalty0
  964--979, 1979.

\bibitem[Mare\v{c}ek et~al.(2017)Mare\v{c}ek, Richt\'{a}rik, and
  Tak\'{a}\v{c}]{Marecek_MC}
J.~Mare\v{c}ek, P.~Richt\'{a}rik, and M.~Tak\'{a}\v{c}.
\newblock Matrix completion under interval uncertainty.
\newblock \emph{European Journal of Operational Research}, 256\penalty0
  (1):\penalty0 35 -- 43, 2017.

\bibitem[Netrapalli et~al.(2014)Netrapalli, Niranjan, Sanghavi, Anandkumar, and
  Jain]{NIPS2014_5430}
P.~Netrapalli, U.~N. Niranjan, S.~Sanghavi, A.~Anandkumar, and P.~Jain.
\newblock Non-convex robust {P}{C}{A}.
\newblock In \emph{Advances in Neural Information Processing Systems 27}, pages
  1107--1115. 2014.

\bibitem[Recht et~al.(2010)Recht, Fazel, and Parrilo]{RechtFazelParrilo2007}
B.~Recht, M.~Fazel, and P.~A. Parrilo.
\newblock Guaranteed minimum-rank solutions of linear matrix equations via
  nuclear norm minimization.
\newblock \emph{SIAM review}, 52\penalty0 (3):\penalty0 471--501, 2010.

\bibitem[Rockafellar and Wets(1998)]{rockafellar1998variational}
R.~T. Rockafellar and R.~Wets.
\newblock \emph{Variational analysis}, volume 317.
\newblock Springer Verlag, 1998.

\bibitem[Tao and Yang(2011)]{rmc_taoyuan}
M.~Tao and J.~Yang.
\newblock Recovering low-rank and sparse components of matrices from incomplete
  and noisy observations.
\newblock \emph{SIAM Journal on Optimization}, 21\penalty0 (1):\penalty0
  57--81, 2011.

\bibitem[Toyama et~al.(1999)Toyama, Krumm, Brumitt, and Meyers]{wallflower}
K~Toyama, J.~Krumm, B.~Brumitt, and B.~Meyers.
\newblock Wallflower: Principles and practice of background maintainance.
\newblock \emph{Seventh International Conference on Computer Vision}, pages
  255--261, 1999.

\bibitem[Wang et~al.(2004)Wang, Bovik, Sheikh, and Simoncelli]{mssim}
Z.~Wang, A.~C. Bovik, H.~R. Sheikh, and E.~P. Simoncelli.
\newblock Image quality assessment: from error visibility to structural
  similarity.
\newblock \emph{IEEE Transaction on Image Processing}, 13\penalty0
  (4):\penalty0 600--612, 2004.

\bibitem[Waters et~al.(2011)Waters, Sankaranarayanan, and Baraniuk]{SpaRCS}
A.~E. Waters, A.~C. Sankaranarayanan, and R.~Baraniuk.
\newblock Spa{R}{C}{S}: Recovering low-rank and sparse matrices from
  compressive measurements.
\newblock \emph{Proceedings of 24nd Advances in Neural Information Processing
  systems}, pages 1089--1097, 2011.

\bibitem[Wen et~al.(2012)Wen, Yin, and Zhang]{LMAFIT}
Z.~Wen, W.~Yin, and Y.~Zhang.
\newblock Solving a low-rank factorization model for matrix completion by a
  nonlinear successive over-relaxation algorithm.
\newblock \emph{Mathematical Programming Computation}, 4\penalty0 (4):\penalty0
  333--361, 2012.

\bibitem[Wright et~al.(2009)Wright, Peng, Ma, Ganseh, and Rao]{APG}
J.~Wright, Y.~Peng, Y.~Ma, A.~Ganseh, and S.~Rao.
\newblock Robust principal component analysis: exact recovery of corrputed
  low-rank matrices by convex optimization.
\newblock \emph{Proceedings of 22nd Advances in Neural Information Processing
  systems}, pages 2080--2088, 2009.

\bibitem[Yi et~al.(2016)Yi, Park, Chen, and Caramanis]{RPCAgd}
X.~Yi, D.~Park, Y.~Chen, and C.~Caramanis.
\newblock Fast algorithms for robust \uppercase{PCA} via gradient descent.
\newblock \emph{Advances in Neural Information Processing systems}, pages
  361--369, 2016.

\bibitem[Yuan and Yang(2013)]{adm_rpca}
X.~Yuan and J.~Yang.
\newblock Sparse and low-rank matrix decomposition via alternating direction
  methods.
\newblock \emph{Pacific Journal of Optimization}, 9\penalty0 (1):\penalty0
  167--180, 2013.

\bibitem[Zhang and Yang(2018)]{zhangpca}
T.~Zhang and Y.~Yang.
\newblock Robust \uppercase{PCA} by manifold optimization.
\newblock \emph{Journal of Machine Learning Research}, 19:\penalty0 1--39,
  2018.

\bibitem[Zhou and Tao(2011)]{godec}
T.~Zhou and D.~Tao.
\newblock Godec: Randomized low-rank and sparse matrix decomposition in noisy
  case.
\newblock In \emph{Proceedings of the 28th International Conference on Machine
  Learning (ICML)}, pages 33--40, 2011.

\end{thebibliography}

\clearpage

\appendix

\part*{Supplementary Material}

The organization of this supplementary material is: extra supporting numerical experiments are reported in Section \ref{sec:extra-exp}; Proofs for the global convergence result of Algorithm \ref{alg:apg} is provided in Section \ref{sec:proof-global}; The proof of local linear convergence and a numerical example are provided in Section \ref{sec:proof-local}. Lastly, we provide a comprehensive table to list all baselines we compare to in Section~\ref{sec:table}.


\section{Extra Experiments}\label{sec:extra-exp}

In this section, we empirically study convergence properties of Algorithm~\ref{alg:apg} on synthetic, well-understood data. In particular, we examine its sensitivity to user-specified parameters $\gamma, a_k, b_k$, target sparsity level $\alpha$, target rank $r$ and lastly the sensitivity to initialization. Moreover, we provide extra phase transition diagrams and both quantitative and qualitative results on the inlier detection problem.

\subsection{Sensitivity to the choice of $\gamma, a_k, b_k$}

In this experiment, we compare different choices of algorithm parameters $\gamma, a_k, b_k$ on instances of~\eqref{eq:X-Y} with various target sparsity level $\alpha$ and target rank $r$. In each experiment, we make sure that the solution exists; we generate random matrices $\tilde{L},\tilde{S}$ (with independent entries ${\cal N}(0,1)$), project them onto low rank and sparse constraint set respectively to obtain $\hat{L}, \hat{S}$ and set $A = \hat{L}+ \hat{S}$. For simplicity we consider only $a_k = b_k=a$ and $m=n=100$. Figure~\ref{fig:sens} shows the result. We see that parameter choice $\gamma = 1.1, a_k = b_k = \frac12$ is the most reliable.

\begin{figure}[h]
\centering
\begin{minipage}{0.25\textwidth}
  \centering
\includegraphics[width =  \textwidth ]{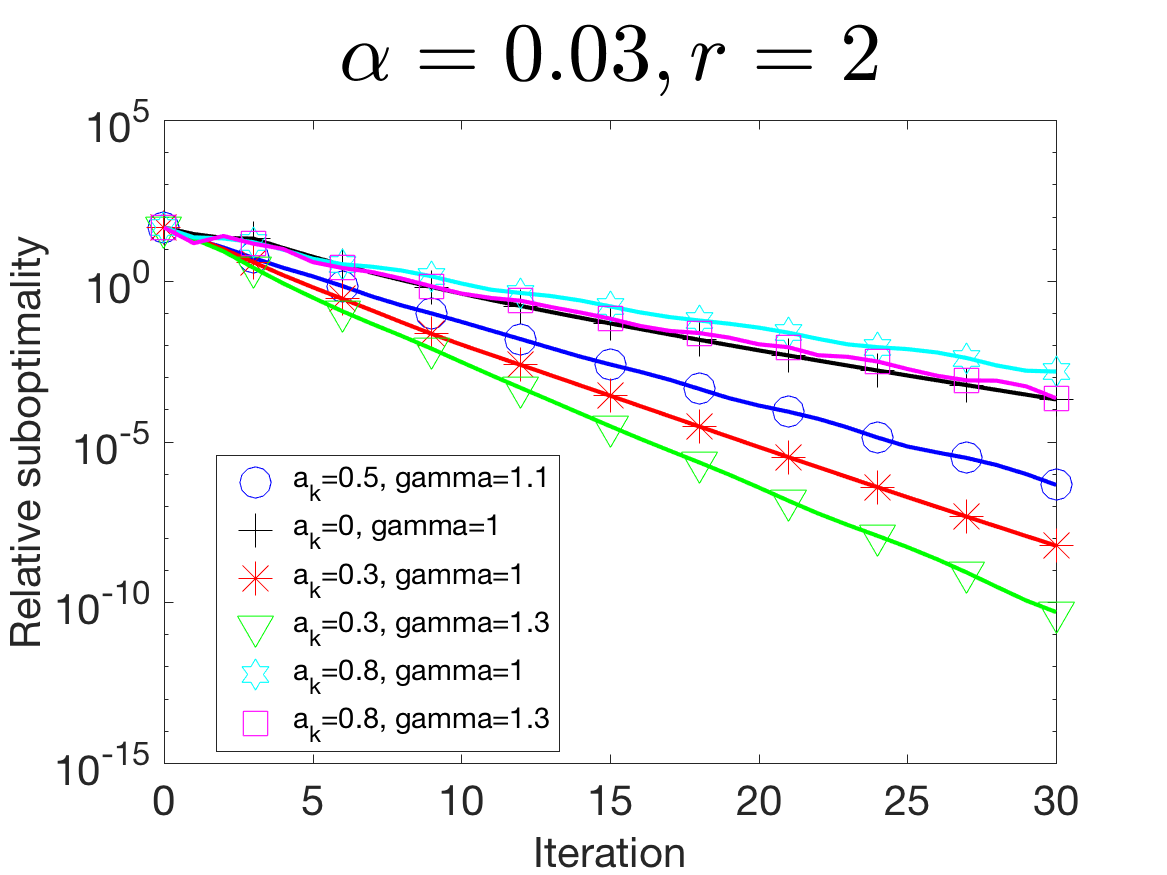}
\end{minipage}%
\begin{minipage}{0.25\textwidth}
  \centering
\includegraphics[width =  \textwidth ]{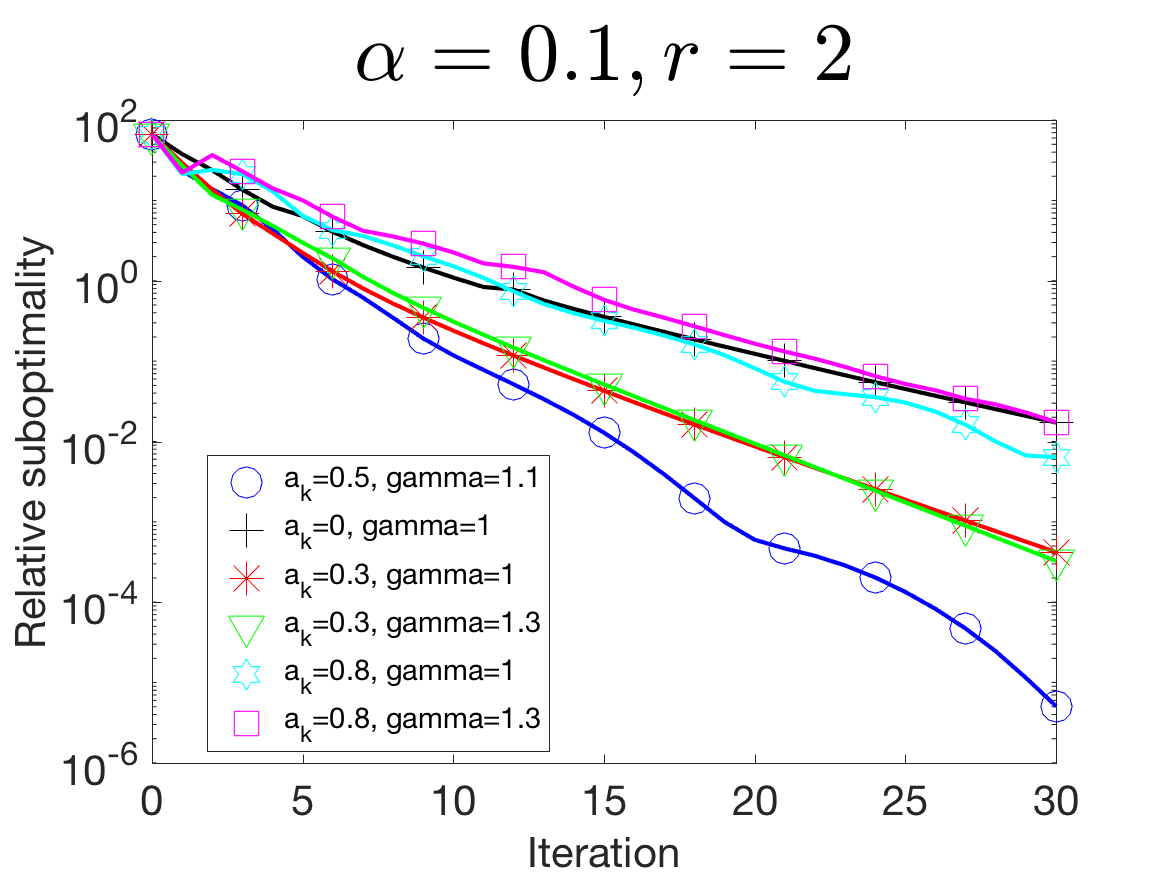}
\end{minipage}%
\begin{minipage}{0.25\textwidth}
  \centering
\includegraphics[width =  \textwidth ]{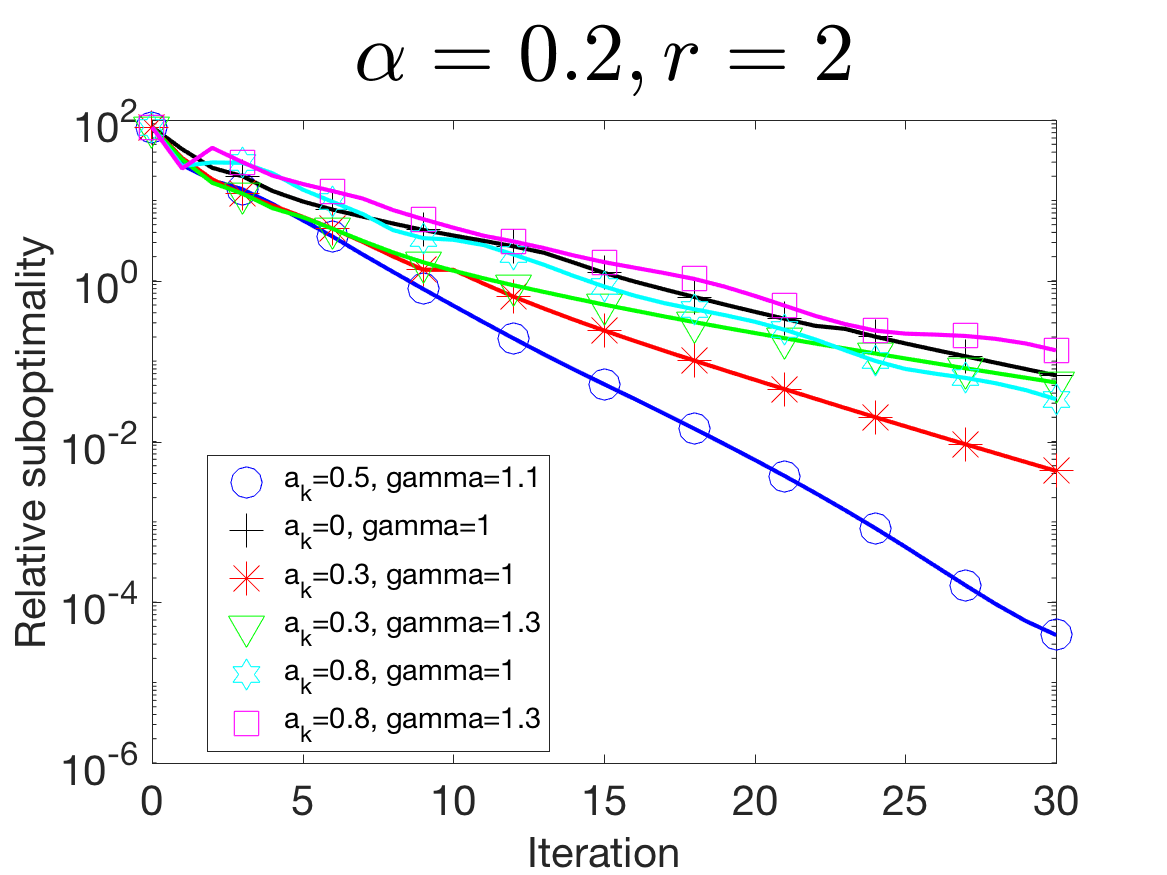}
\end{minipage}%
\begin{minipage}{0.25\textwidth}
  \centering
\includegraphics[width =  \textwidth ]{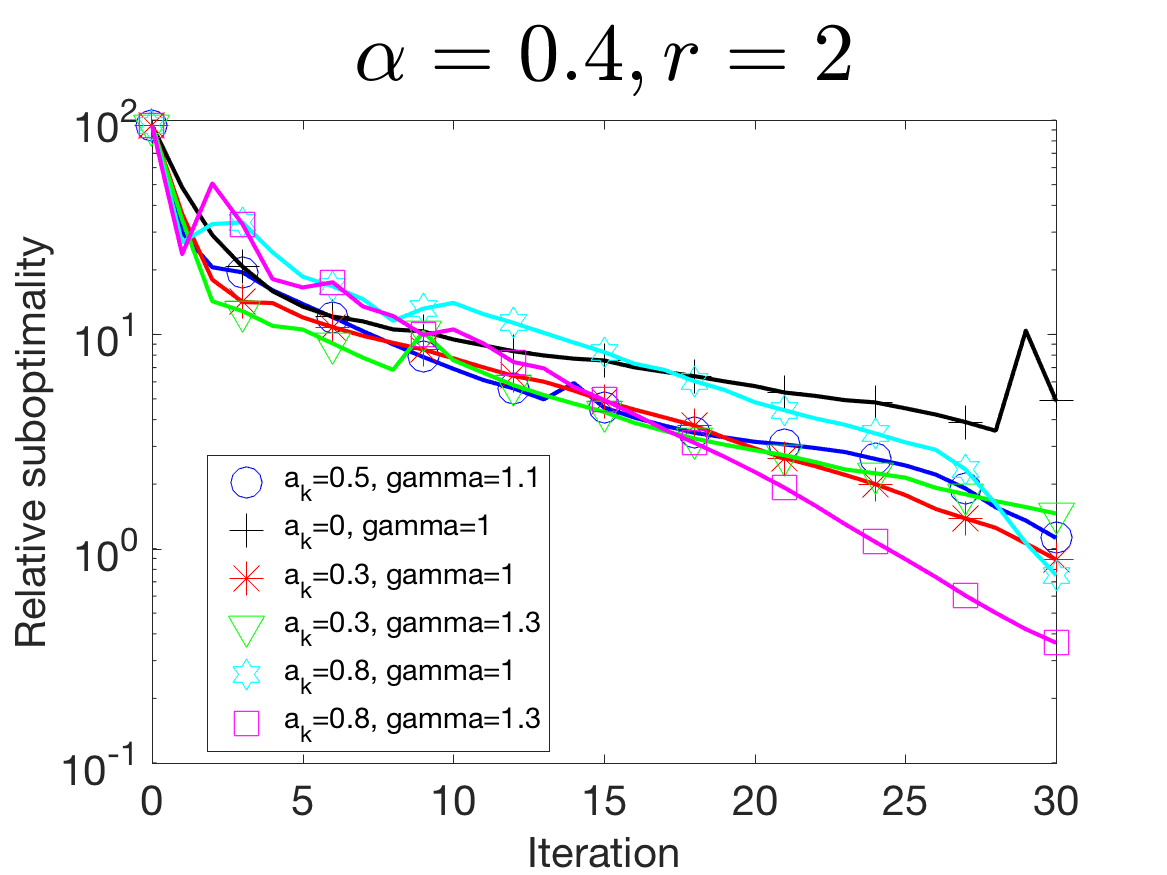}
\end{minipage}%
\\
\begin{minipage}{0.25\textwidth}
  \centering
\includegraphics[width =  \textwidth ]{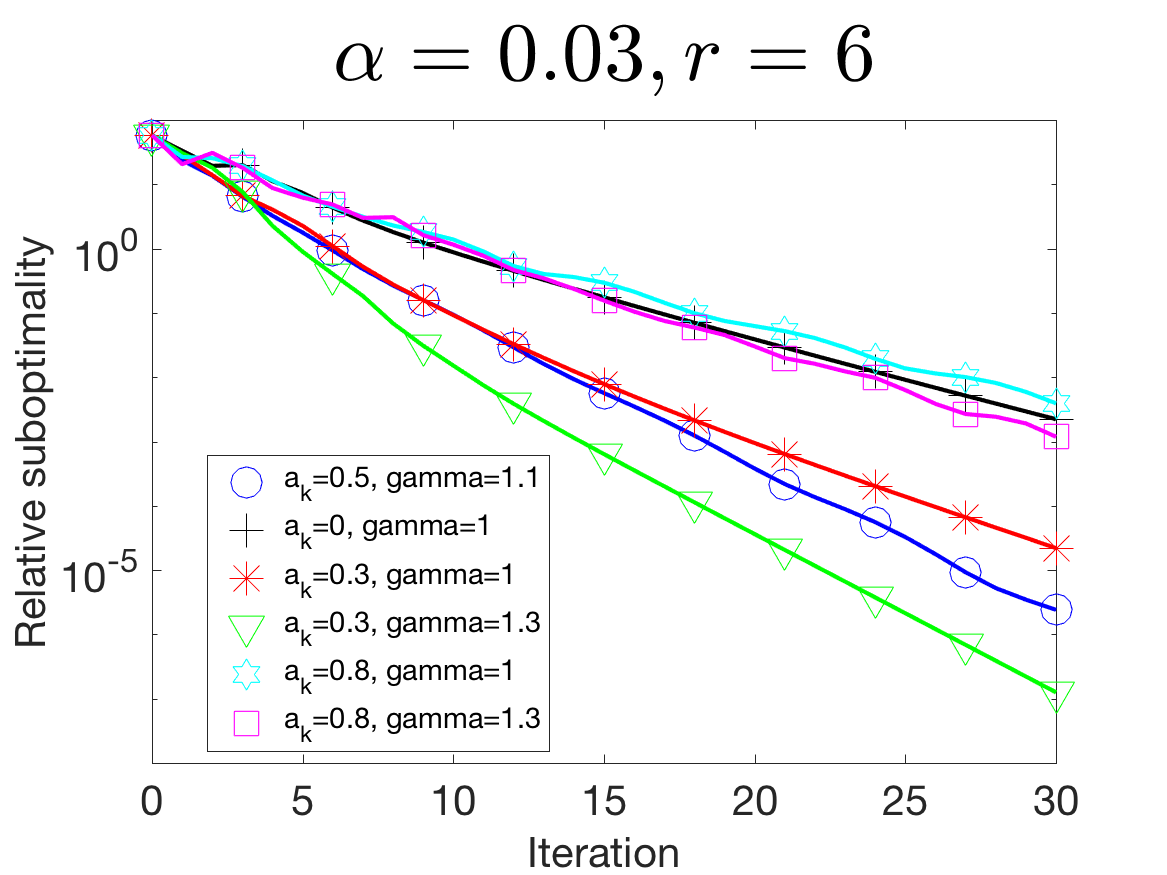}
\end{minipage}%
\begin{minipage}{0.25\textwidth}
  \centering
\includegraphics[width =  \textwidth]{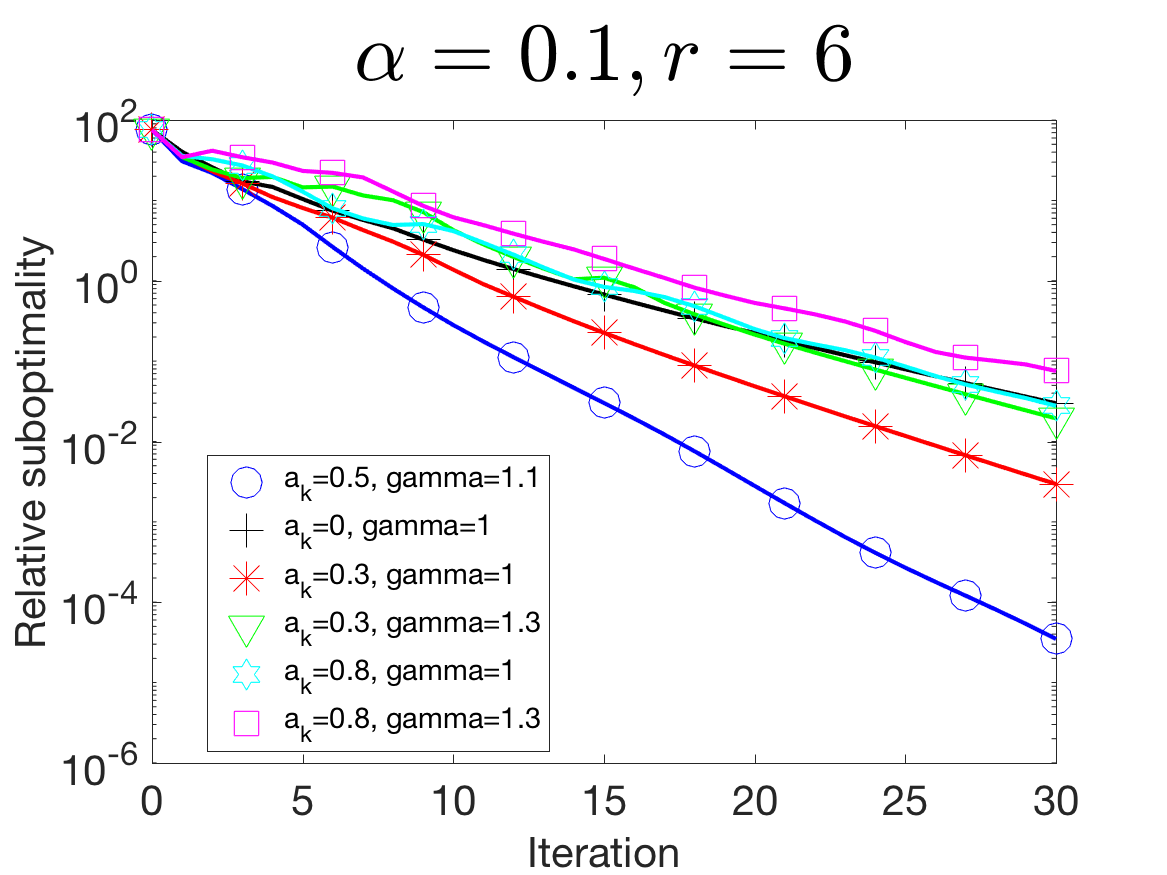}
\end{minipage}%
\begin{minipage}{0.25\textwidth}
  \centering
\includegraphics[width =  \textwidth ]{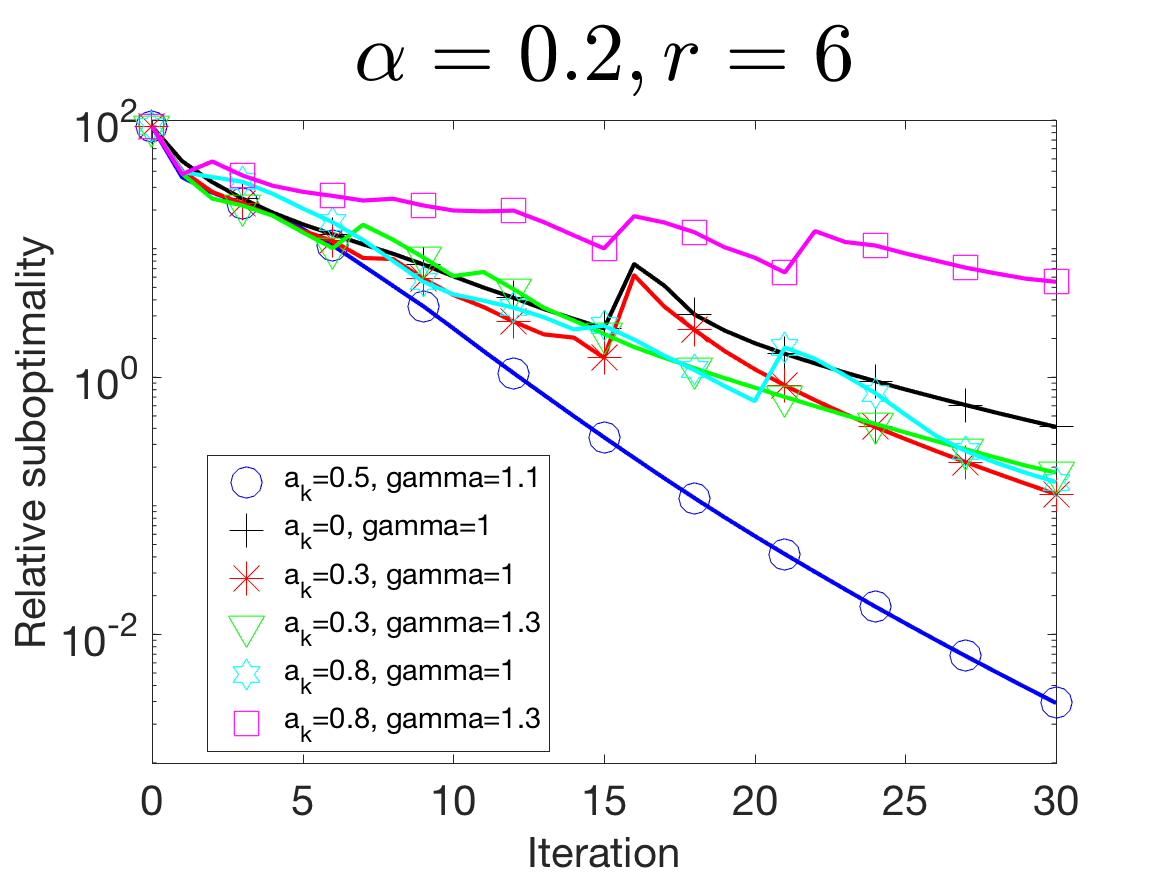}
\end{minipage}%
\begin{minipage}{0.25\textwidth}
  \centering
\includegraphics[width =  \textwidth ]{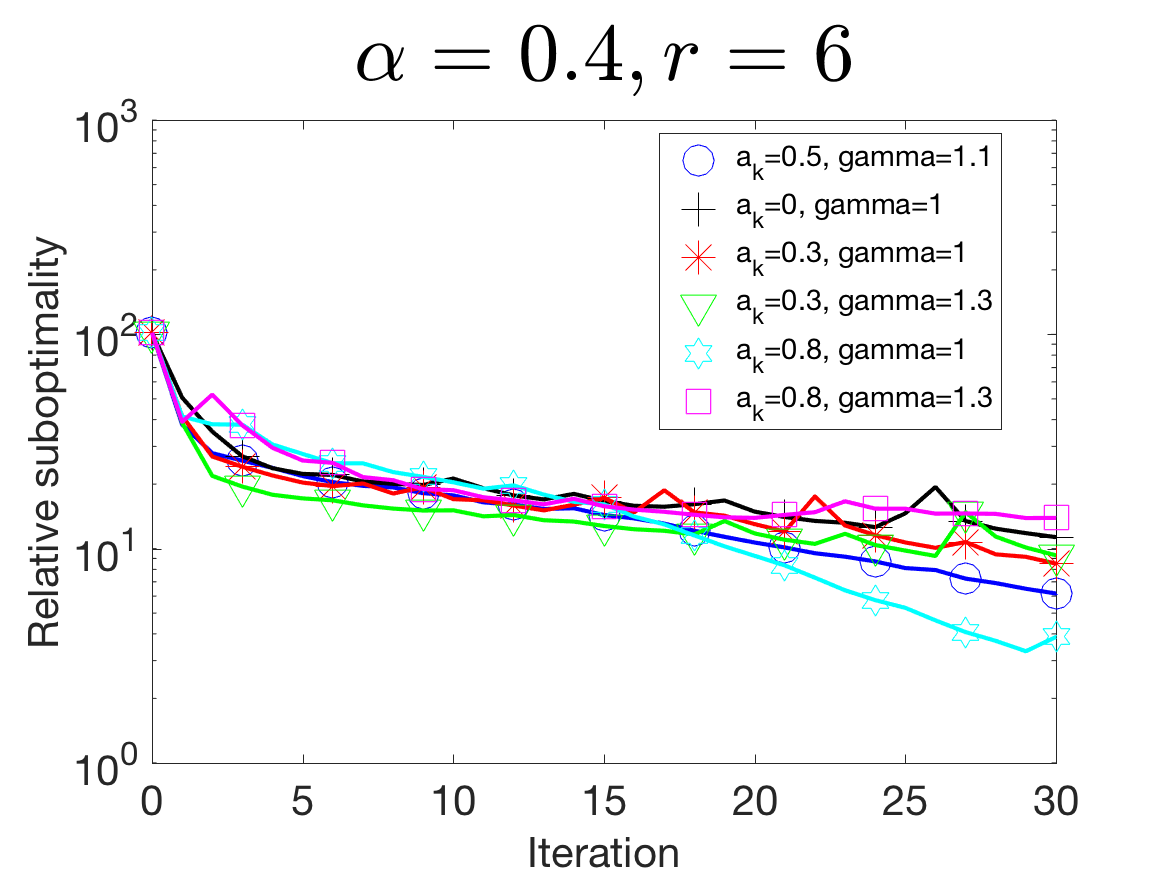}
\end{minipage}%
\\
\begin{minipage}{0.25\textwidth}
  \centering
\includegraphics[width =  \textwidth ]{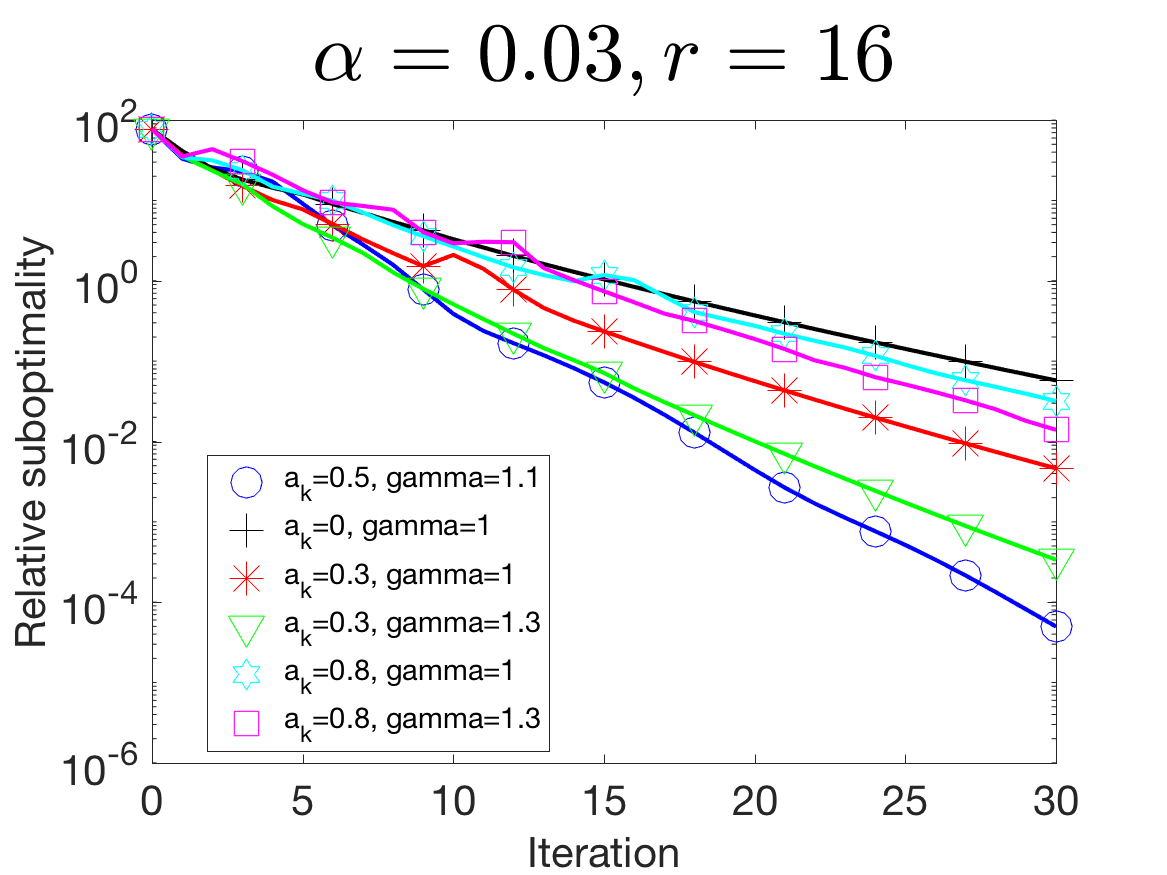}
\end{minipage}%
\begin{minipage}{0.25\textwidth}
  \centering
\includegraphics[width =  \textwidth ]{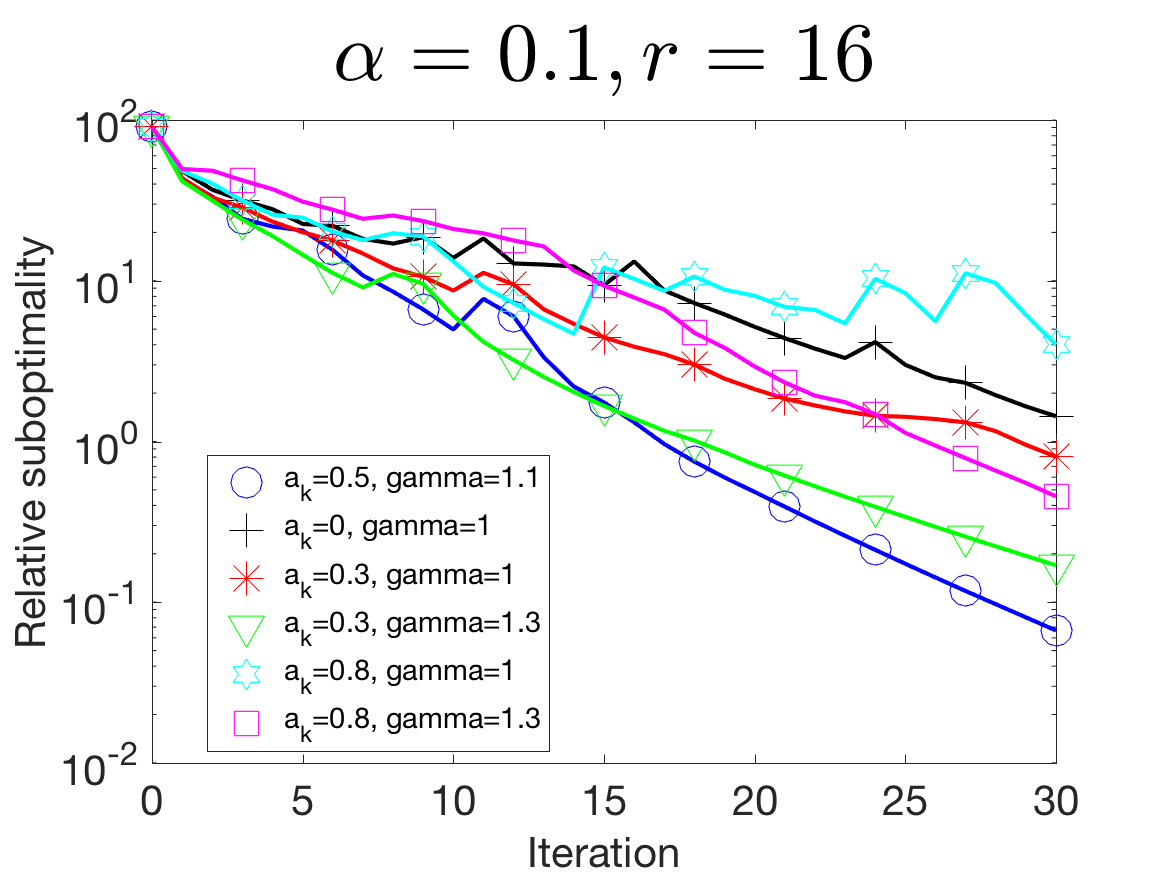}
\end{minipage}%
\begin{minipage}{0.25\textwidth}
  \centering
\includegraphics[width =  \textwidth ]{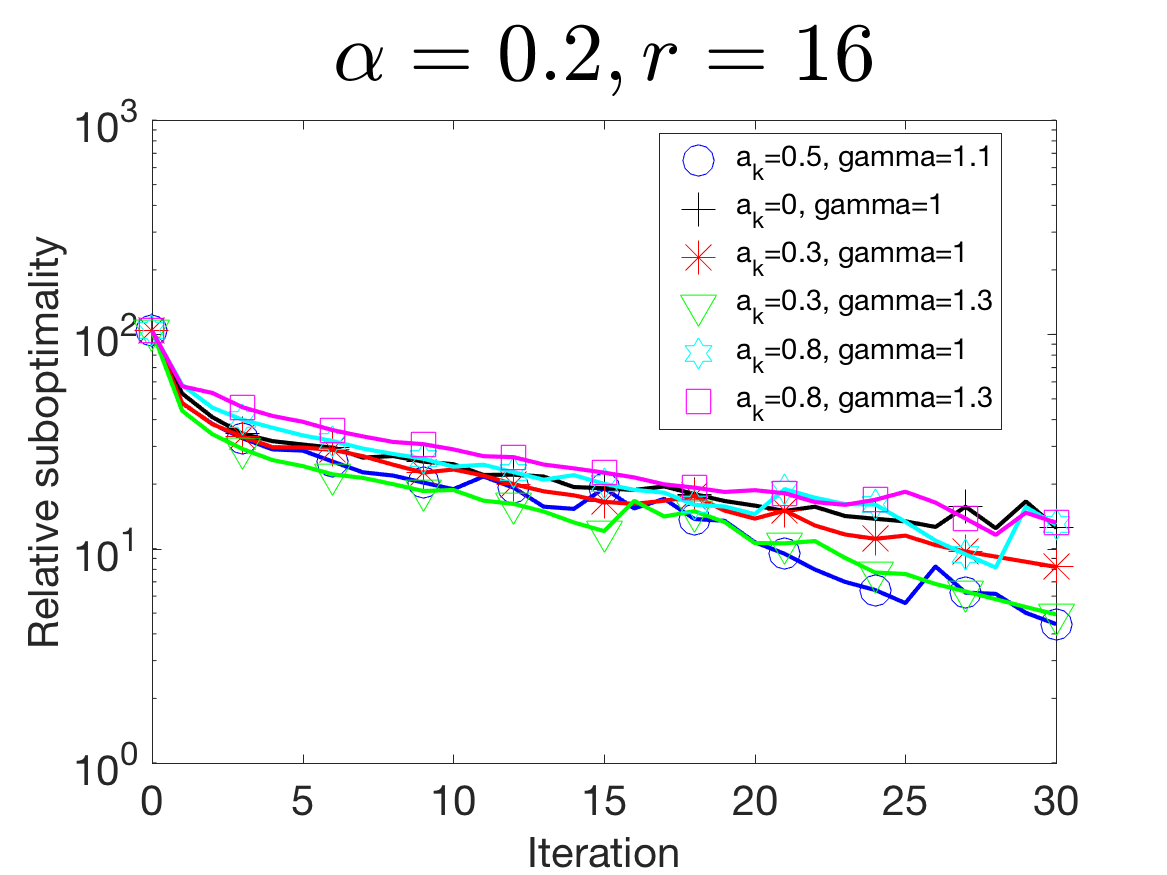}
\end{minipage}%
\begin{minipage}{0.25\textwidth}
  \centering
\includegraphics[width =  \textwidth ]{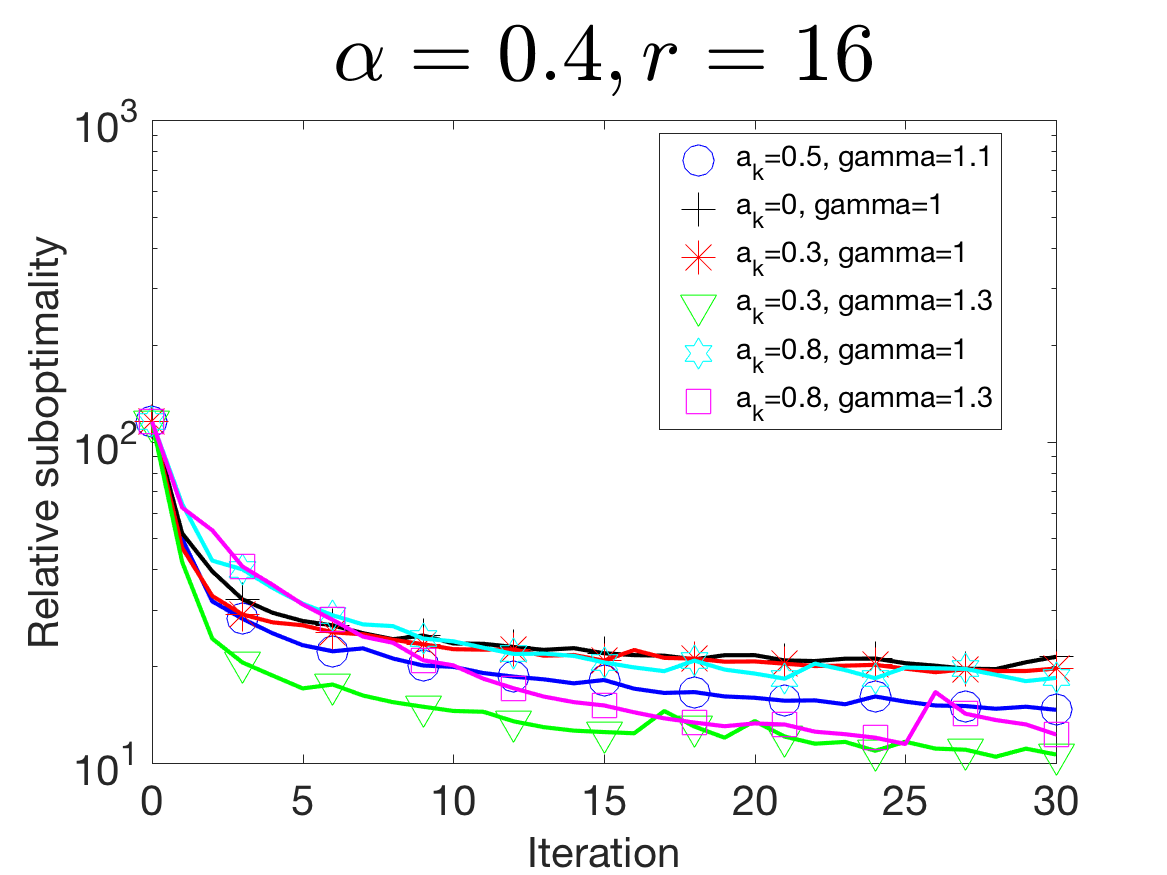}
\end{minipage}%
\\
\begin{minipage}{0.25\textwidth}
  \centering
\includegraphics[width =  \textwidth ]{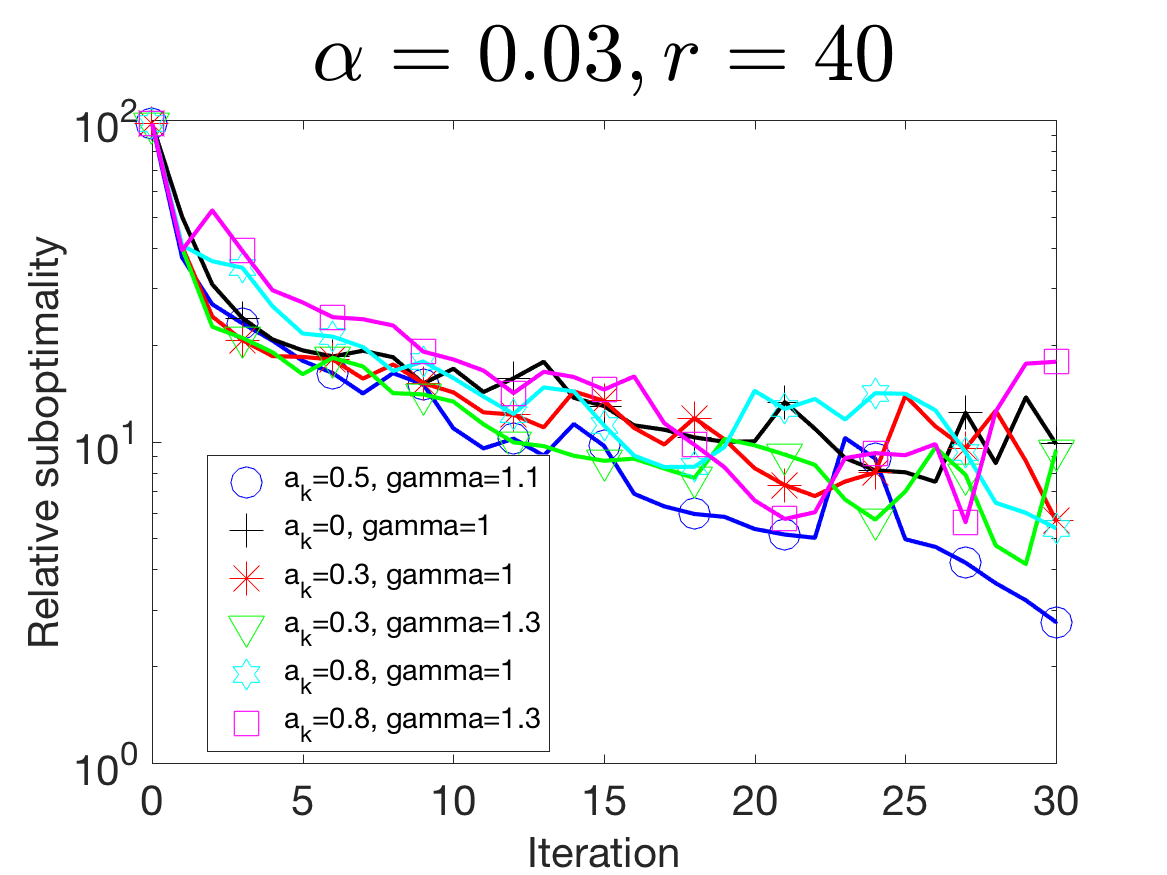}
\end{minipage}%
\begin{minipage}{0.25\textwidth}
  \centering
\includegraphics[width =  \textwidth ]{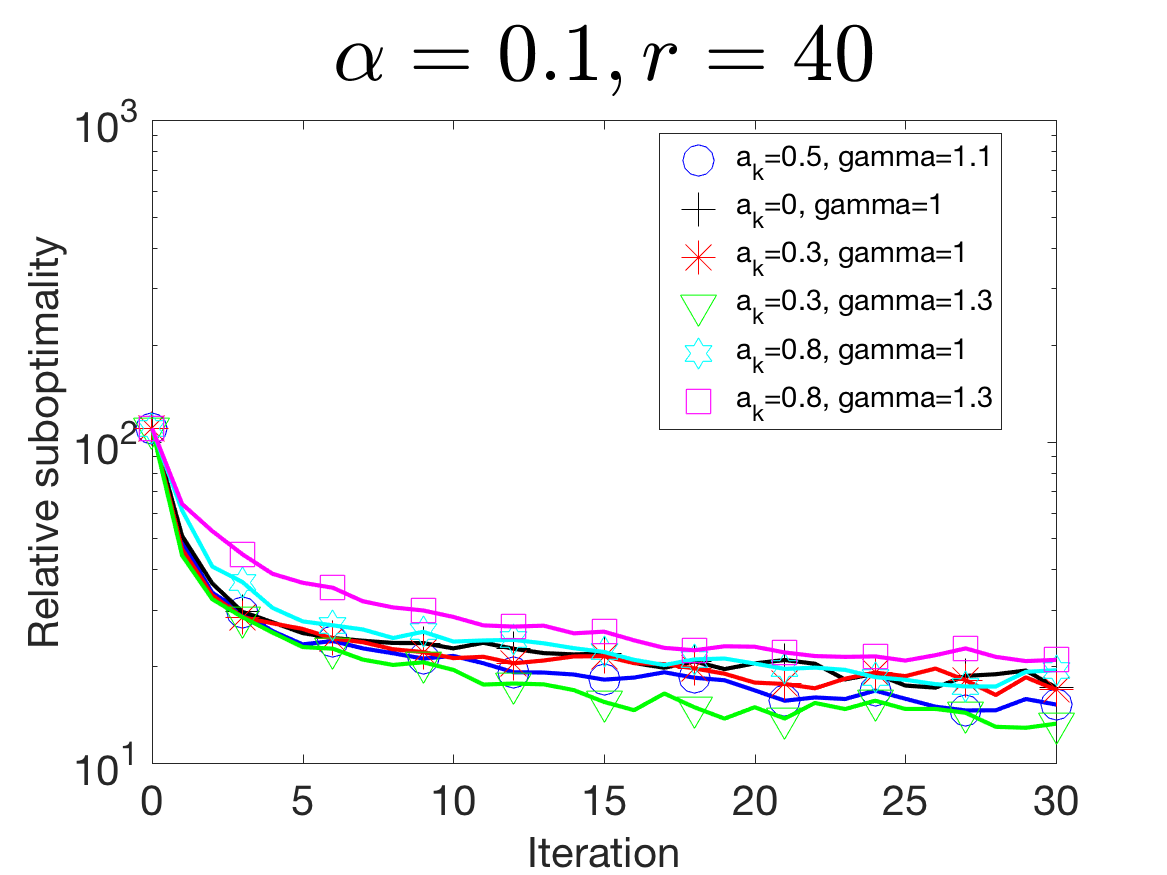}
\end{minipage}%
\begin{minipage}{0.25\textwidth}
  \centering
\includegraphics[width =  \textwidth ]{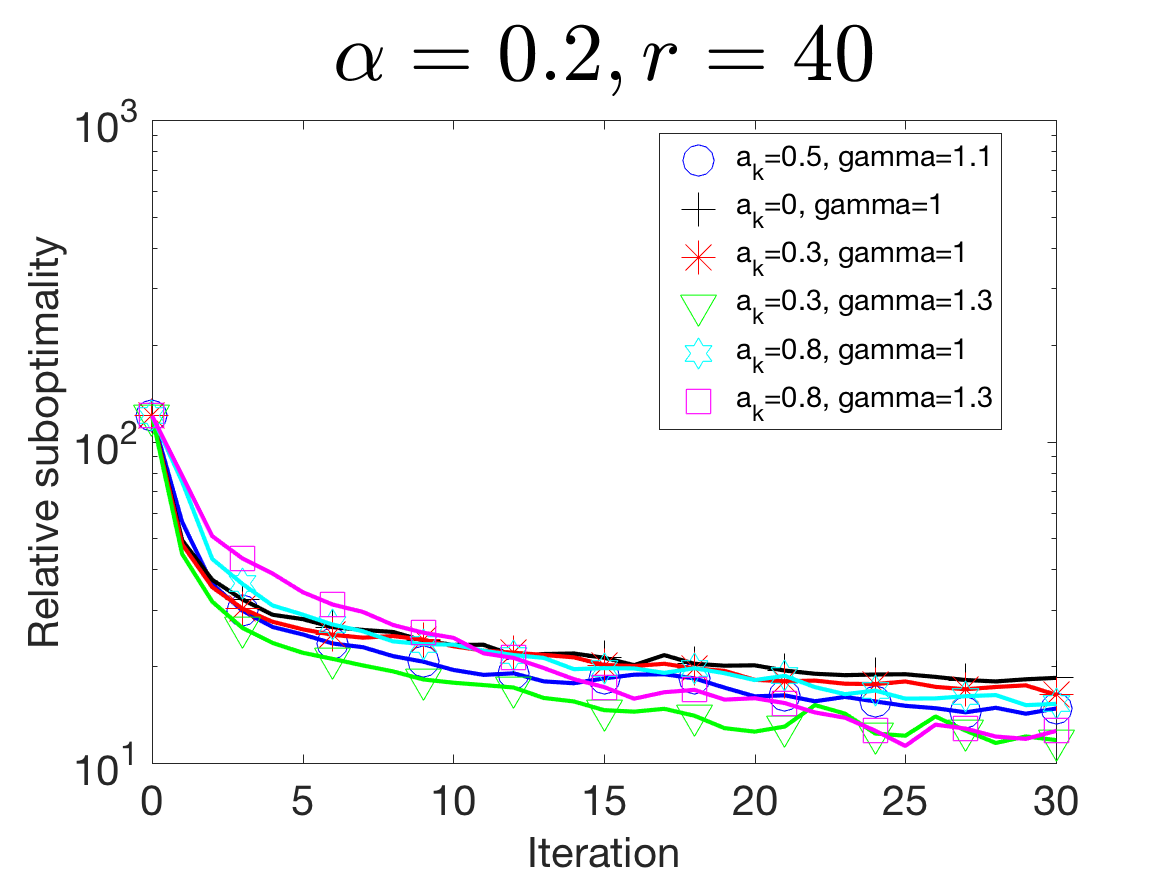}
\end{minipage}%
\begin{minipage}{0.25\textwidth}
  \centering
\includegraphics[width =  \textwidth ]{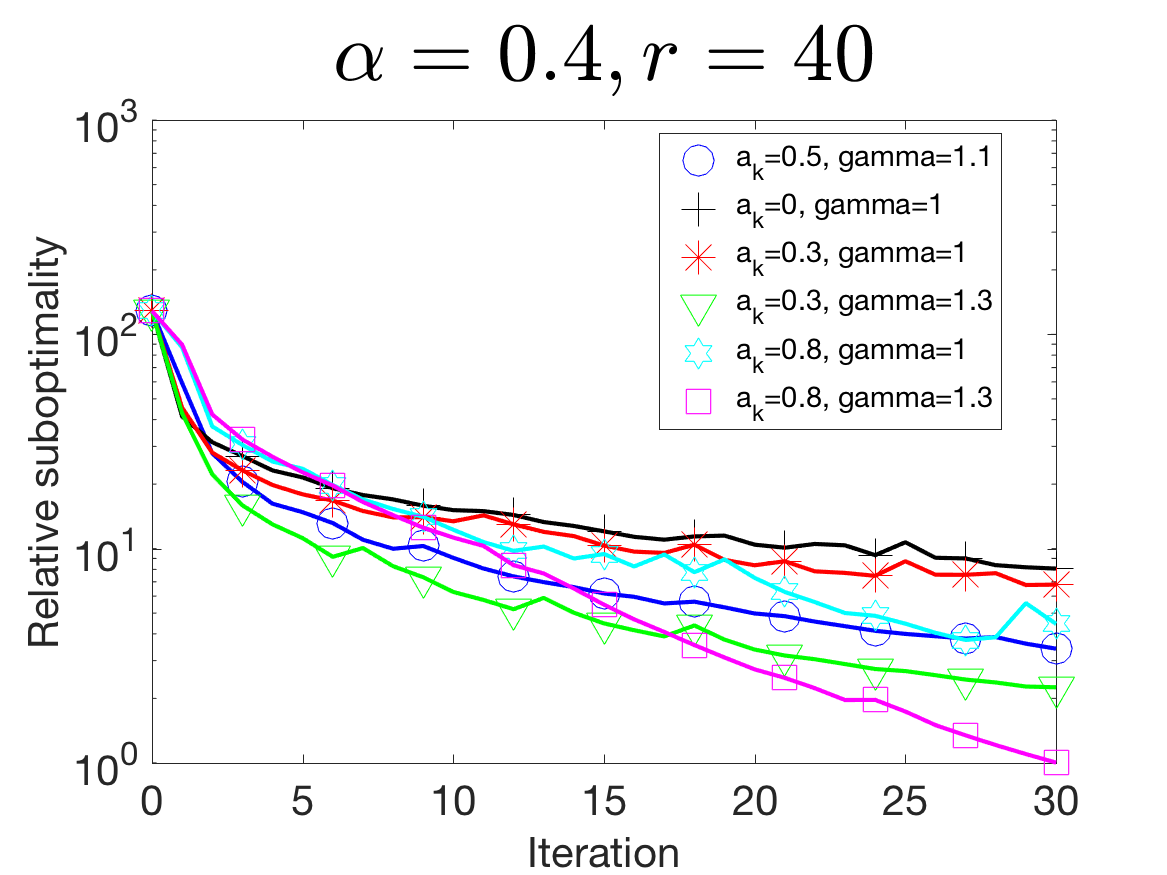}
\end{minipage}%
\caption{ Sensitivity of Algrithm~\ref{alg:apg} with respect to choice of $\gamma, a_k=b_k$.}\label{fig:sens}
\end{figure}

\subsection{Sensitivity to the choice of $r, \alpha$}
In this experiment, we examine how sensitive is Algorithm~\ref{alg:apg} on the correct choice of the target sparsity level $\alpha$ and the target rank $r$. 

In each experiment, we generate random matrices $\tilde{L},\tilde{S}$ (with independent entries ${\cal N}(0,1)$), project them onto $\hat{r}$-low rank and $\hat{\alpha}$-sparse constraint set respectively to obtain $\hat{L}, \hat{S}$ and set $A = \hat{L}+ \hat{S}$. Then, we run Algorithm~\ref{alg:apg} with various choices of $r,\alpha$ and report the results. For simplicity we consider only $\gamma = 1.1, a_k = b_k = \frac12$ (from the previous experiment) and $m=n=100$. Figure~\ref{fig:sens2} shows the result. We can see that if sparsity level is underestimated, the method converges very slowly. Moreover, the method is more sensitive to the correct choices of target sparsity than target rank. The last take-away from this experiment is that over-estimation of target parameters usually leads to slightly slower convergence.

\begin{figure}[h]
\centering
\begin{minipage}{0.25\textwidth}
  \centering
\includegraphics[width =  \textwidth ]{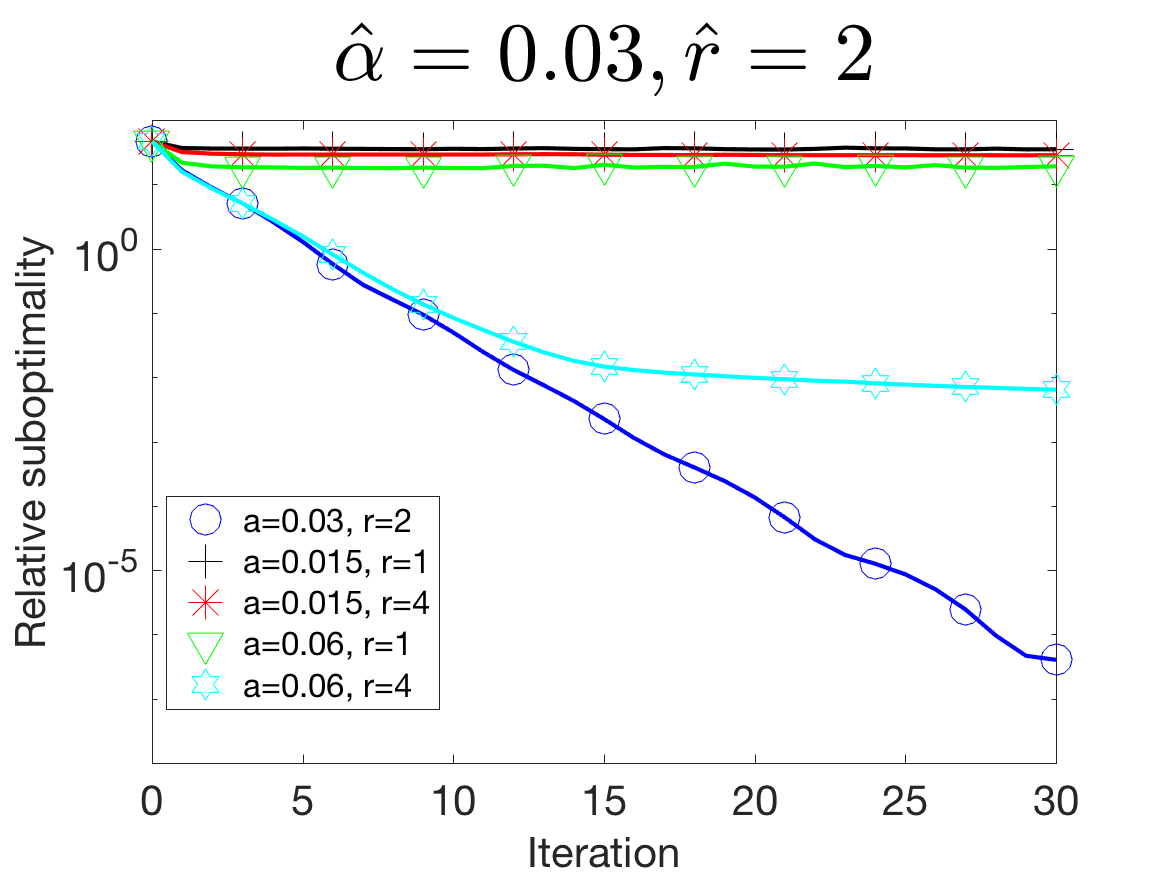}
\end{minipage}%
\begin{minipage}{0.25\textwidth}
  \centering
\includegraphics[width =  \textwidth ]{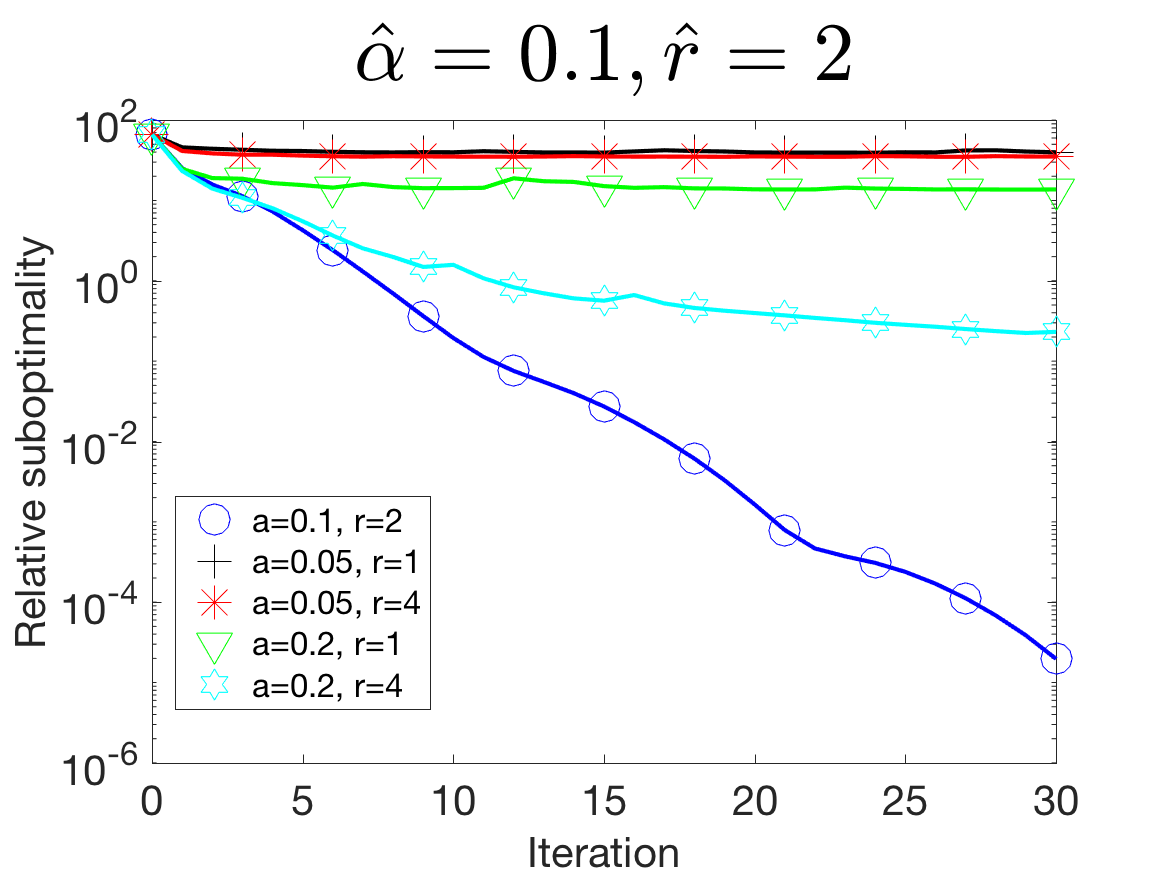}
\end{minipage}%
\begin{minipage}{0.25\textwidth}
  \centering
\includegraphics[width =  \textwidth ]{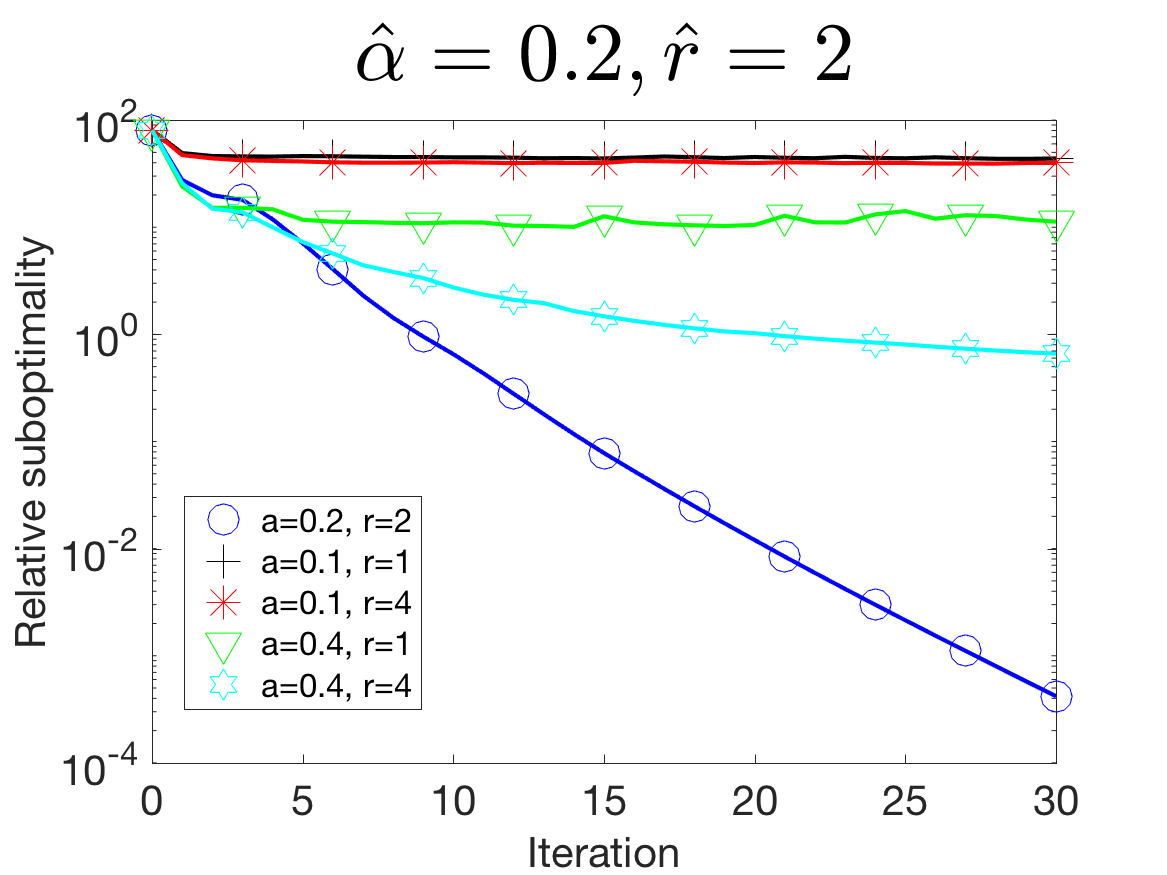}
\end{minipage}%
\begin{minipage}{0.25\textwidth}
  \centering
\includegraphics[width =  \textwidth ]{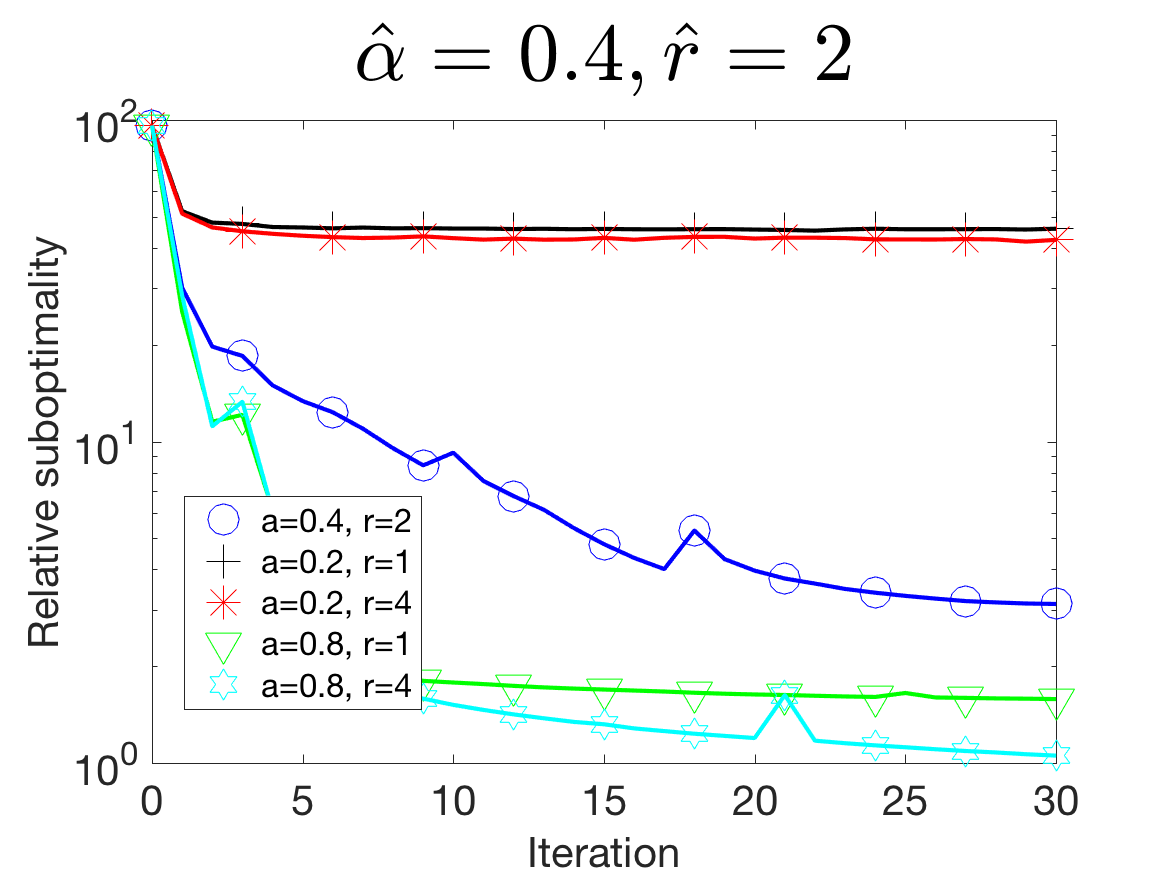}
\end{minipage}%
\\
\begin{minipage}{0.25\textwidth}
  \centering
\includegraphics[width =  \textwidth ]{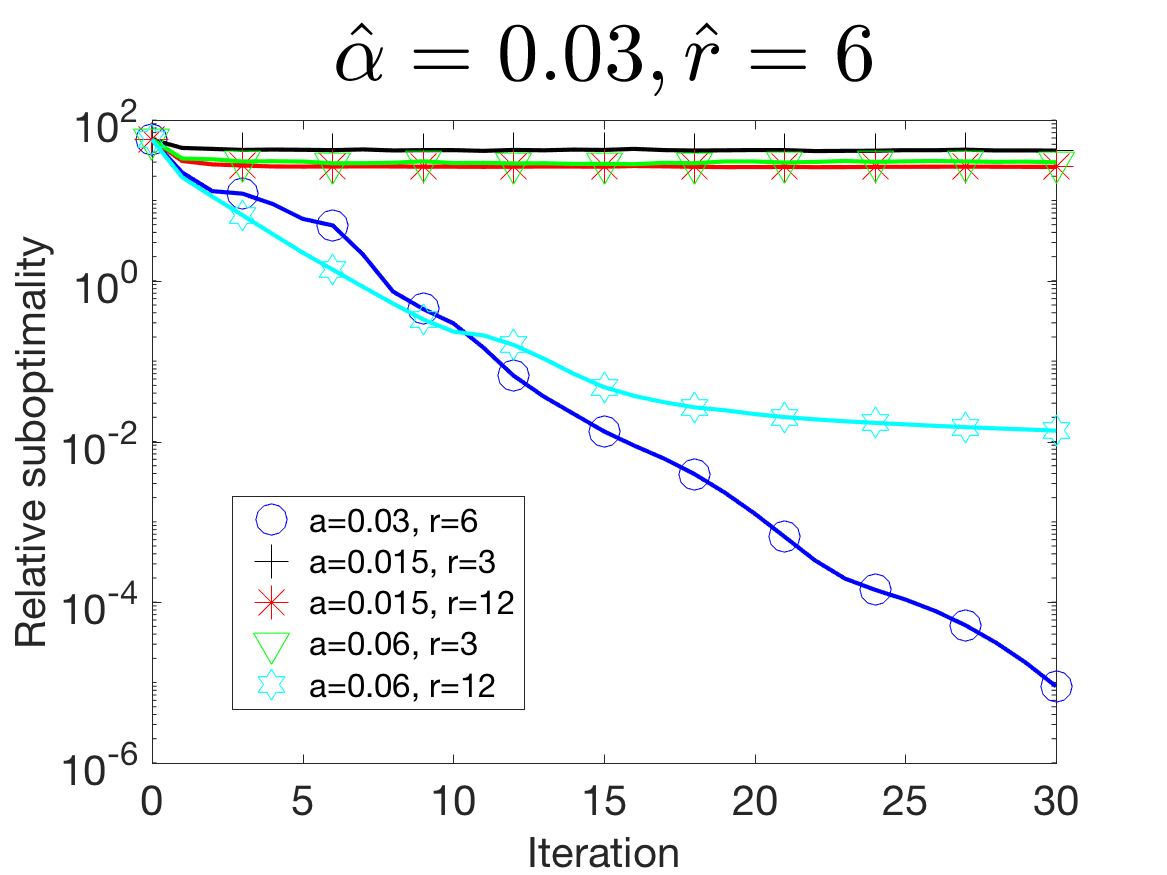}
\end{minipage}%
\begin{minipage}{0.25\textwidth}
  \centering
\includegraphics[width =  \textwidth ]{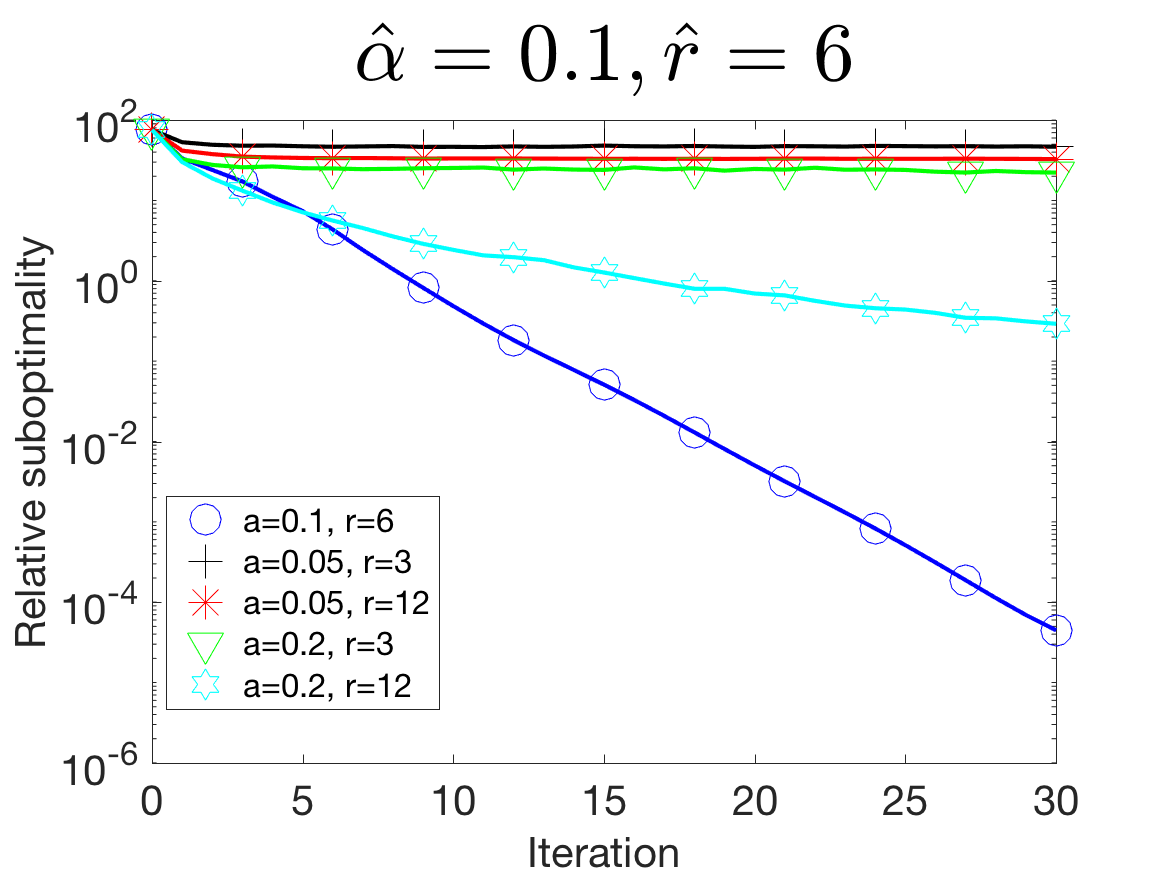}
\end{minipage}%
\begin{minipage}{0.25\textwidth}
  \centering
\includegraphics[width =  \textwidth ]{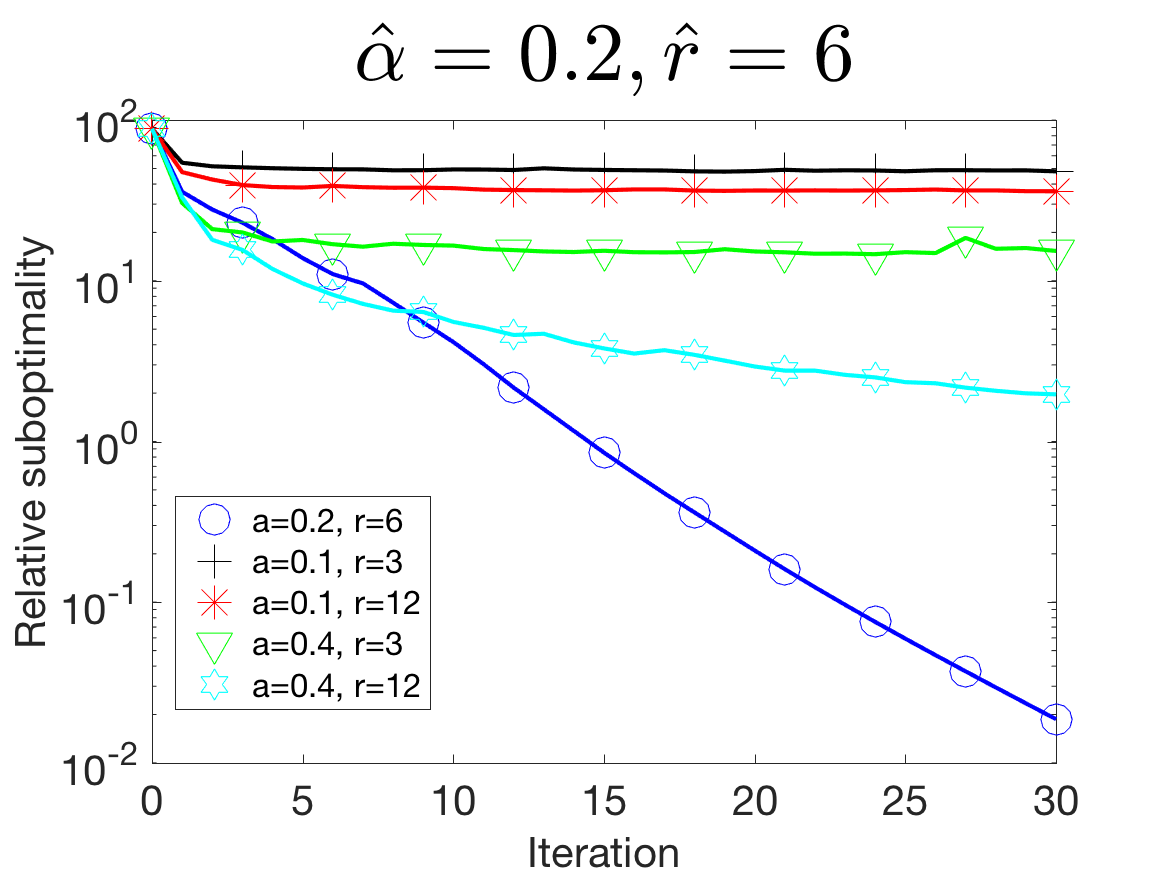}
\end{minipage}%
\begin{minipage}{0.25\textwidth}
  \centering
\includegraphics[width =  \textwidth ]{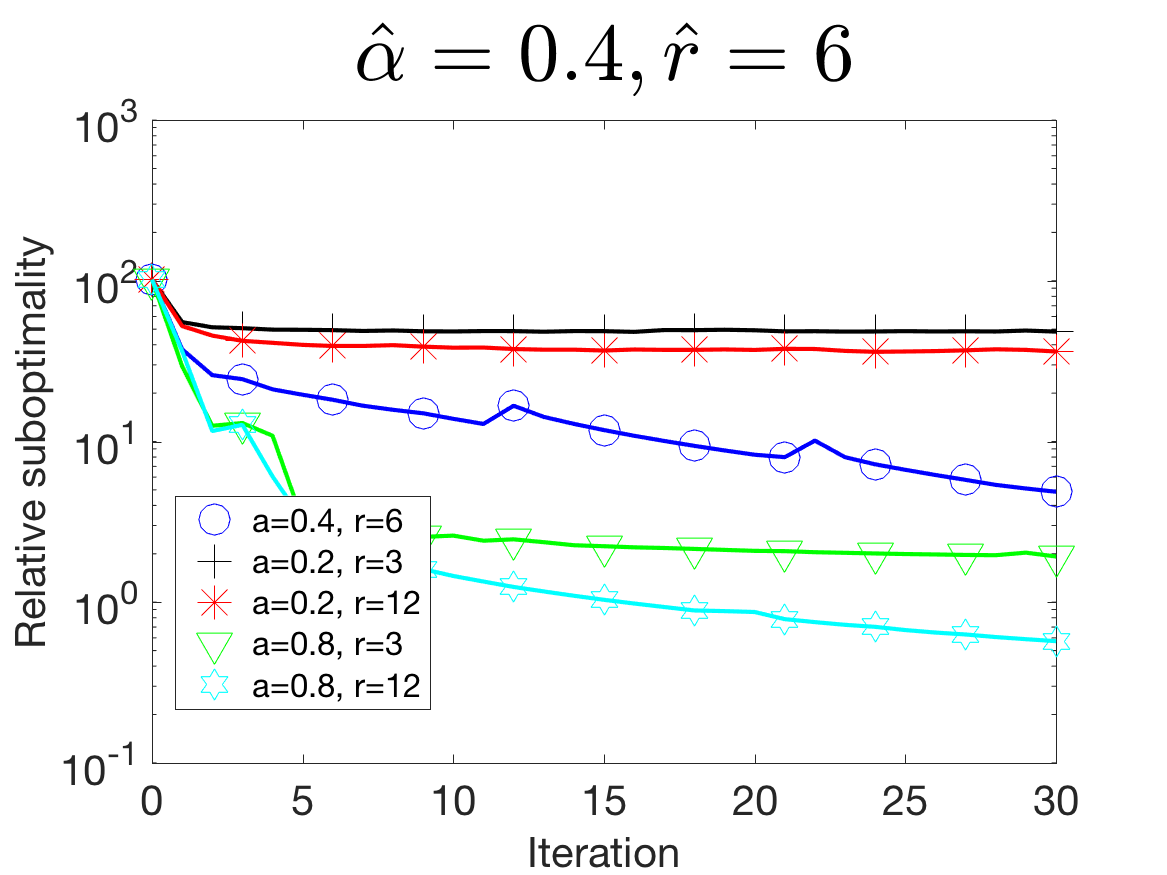}
\end{minipage}%
\\
\begin{minipage}{0.25\textwidth}
  \centering
\includegraphics[width =  \textwidth ]{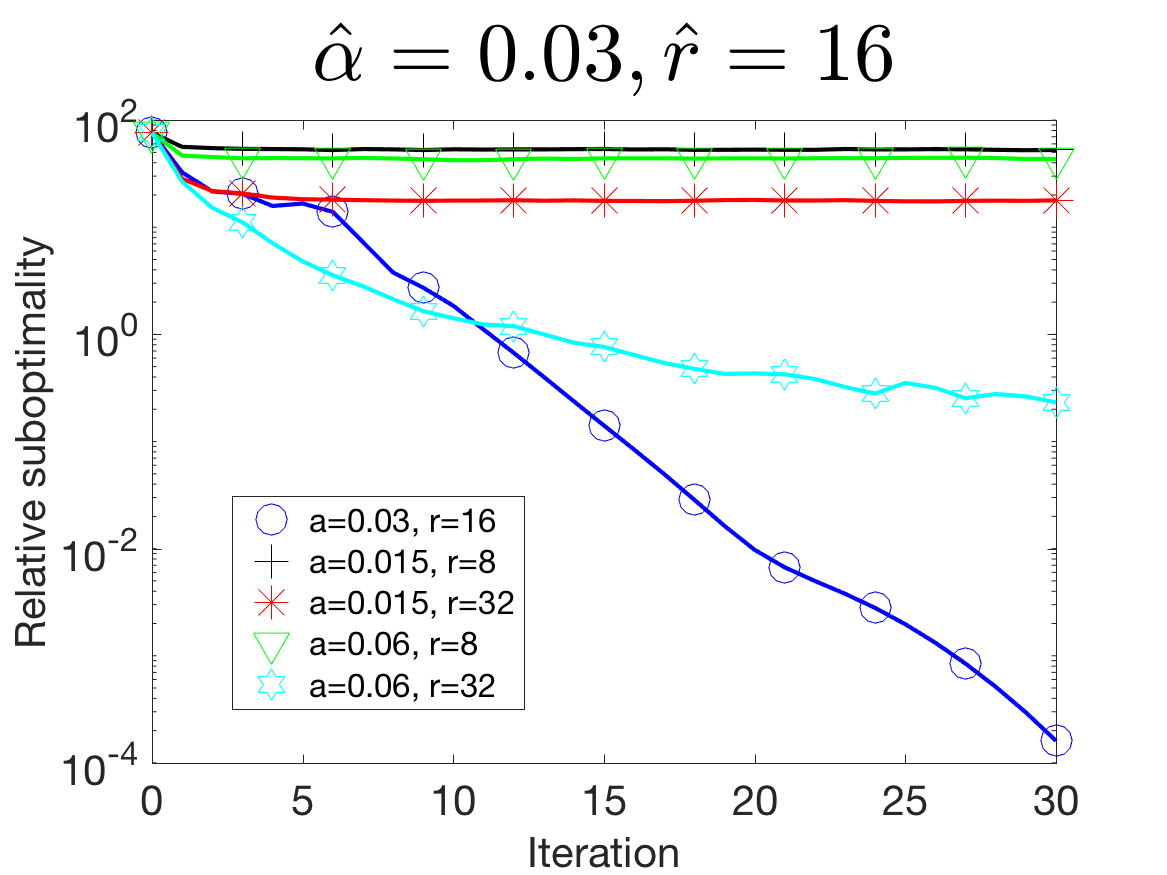}
\end{minipage}%
\begin{minipage}{0.25\textwidth}
  \centering
\includegraphics[width =  \textwidth ]{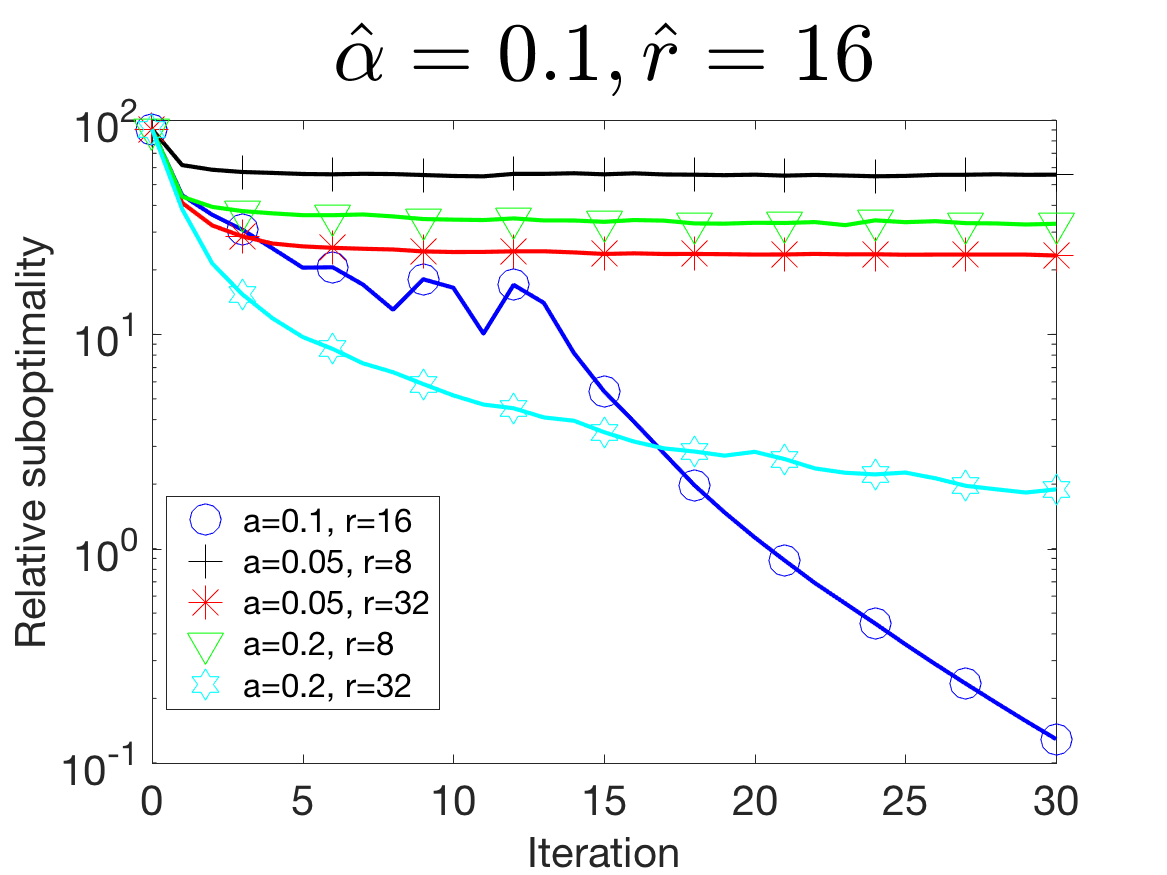}
\end{minipage}%
\begin{minipage}{0.25\textwidth}
  \centering
\includegraphics[width =  \textwidth ]{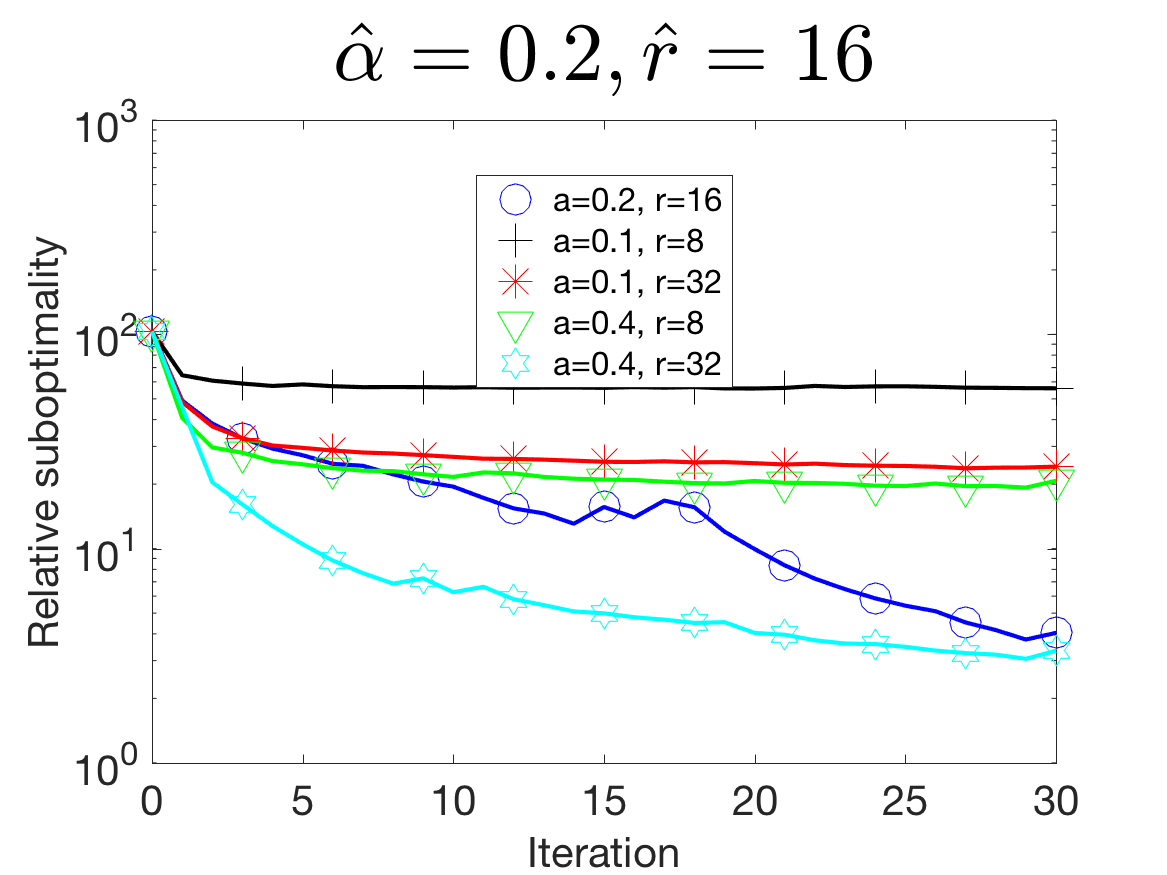}
\end{minipage}%
\begin{minipage}{0.25\textwidth}
  \centering
\includegraphics[width =  \textwidth ]{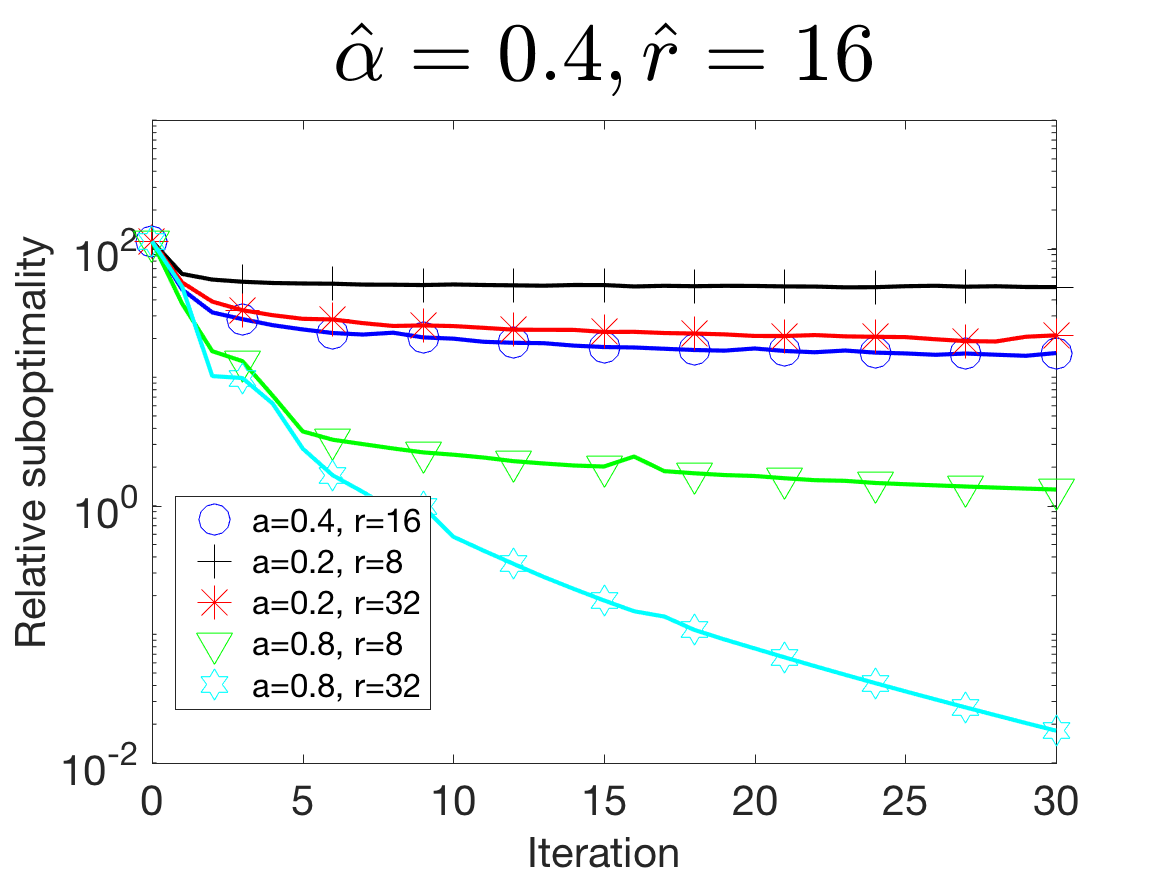}
\end{minipage}%
\\
\begin{minipage}{0.25\textwidth}
  \centering
\includegraphics[width =  \textwidth ]{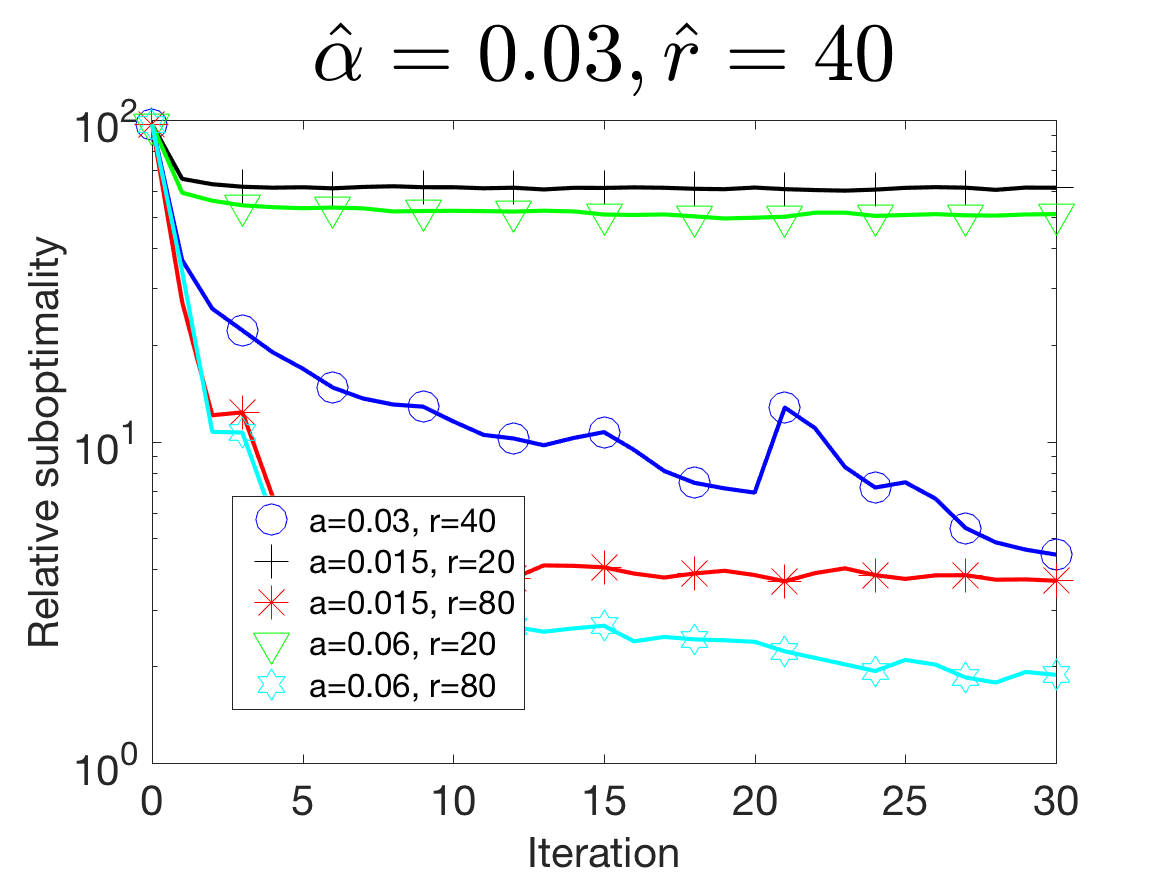}
\end{minipage}%
\begin{minipage}{0.25\textwidth}
  \centering
\includegraphics[width =  \textwidth ]{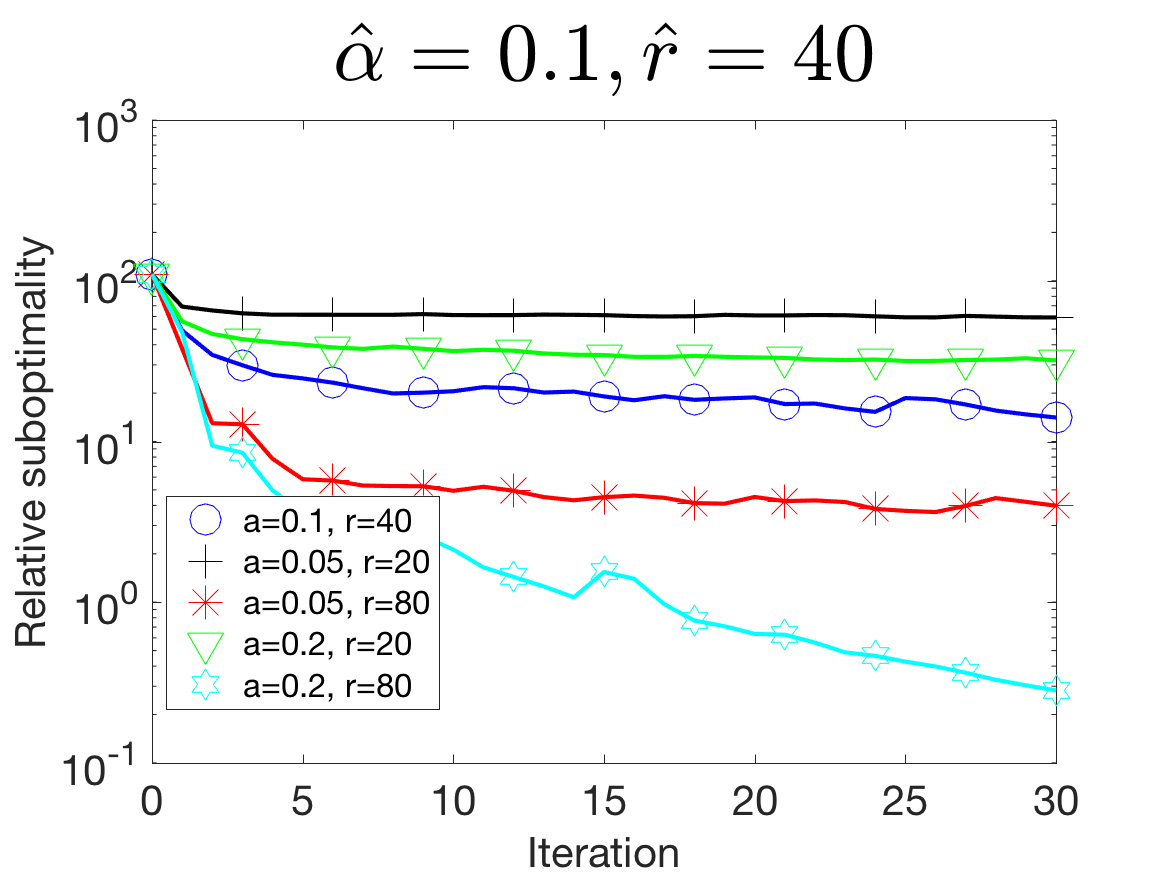}
\end{minipage}%
\begin{minipage}{0.25\textwidth}
  \centering
\includegraphics[width =  \textwidth ]{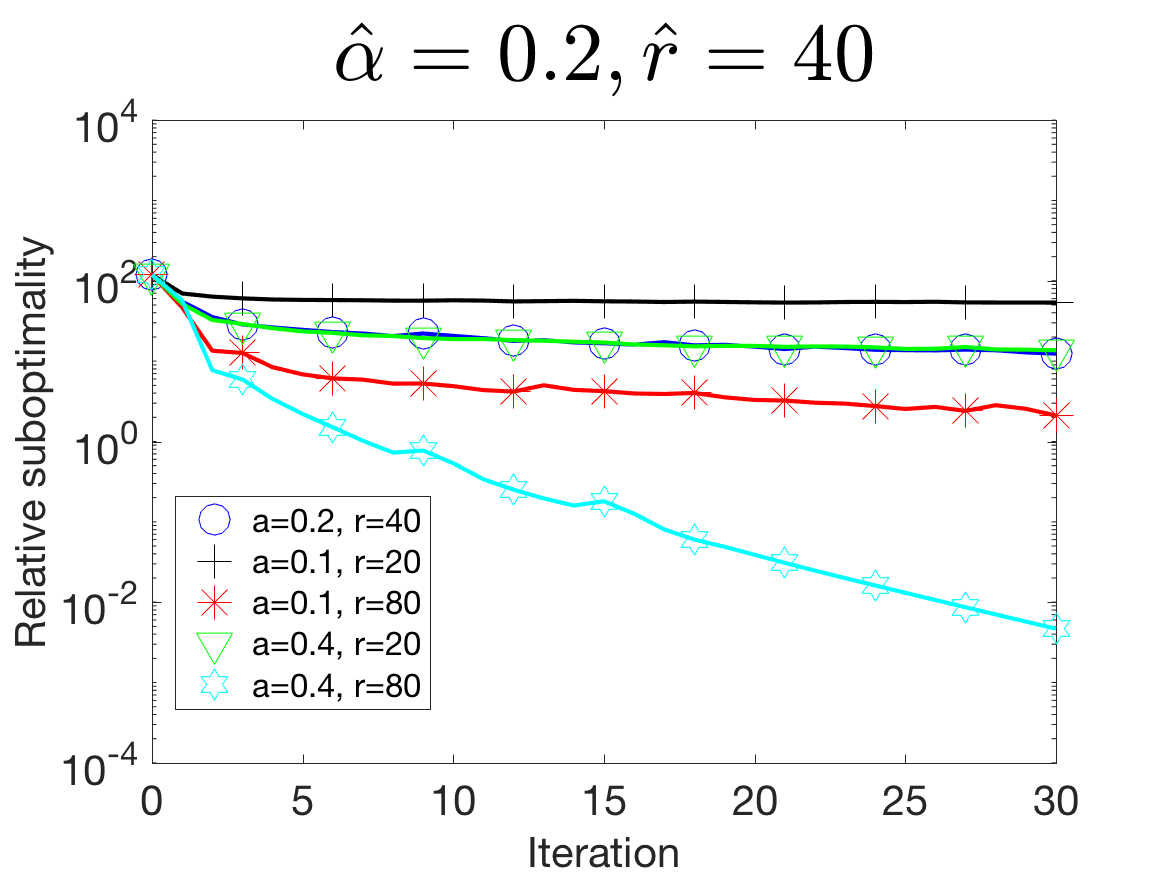}
\end{minipage}%
\begin{minipage}{0.25\textwidth}
  \centering
\includegraphics[width =  \textwidth ]{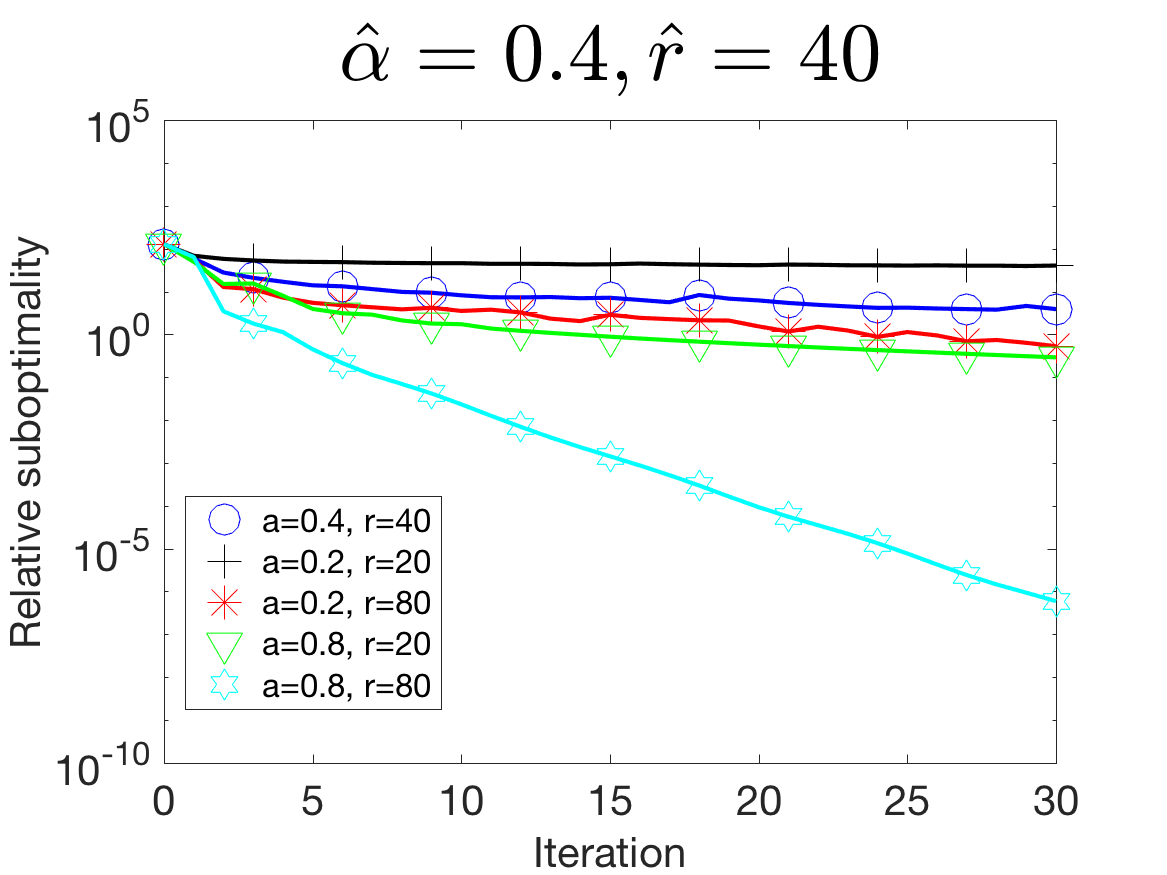}
\end{minipage}%
\caption{Sensitivity of Algorithm~\ref{alg:apg} with respect to the correct choice of target rank and target sparsity.}\label{fig:sens2}
\end{figure}

\subsection{Sensitivity to the choice of the starting point}

In the last experiment, we examine how the starting point influences the convergence rate. For each problem instance, we perform 50 independent runs of Algorithm~\ref{alg:apg} and report the best, worst and median performance. 

For simplicity, we consider only problems with known target rank and sparsity -- we generate random matrices $\tilde{L},\tilde{S}$ (with independent entries ${\cal N}(0,1)$), project them onto low rank and sparse constraint set respectively to obtain $\hat{L}, \hat{S}$ and set $A = \hat{L}+ \hat{S}$. Further, we set $a_k = b_k=0.5$, $\gamma =1.1$ and $m=n=100$. Figure~\ref{fig:sens3} shows the result. We can see that the convergence speed of Algorithm~\ref{alg:apg} is, in most cases, not influenced significantly by the starting point. Thus, the non-convex nature of the problem is surprisingly not causing any issues. Lastly, the convergence rate of  Algorithm~\ref{alg:apg} is faster for small values of $\alpha,r$, which is often the most interesting case in terms of the practical application.

\begin{figure}[h]
\centering
\begin{minipage}{0.25\textwidth}
  \centering
\includegraphics[width =  \textwidth ]{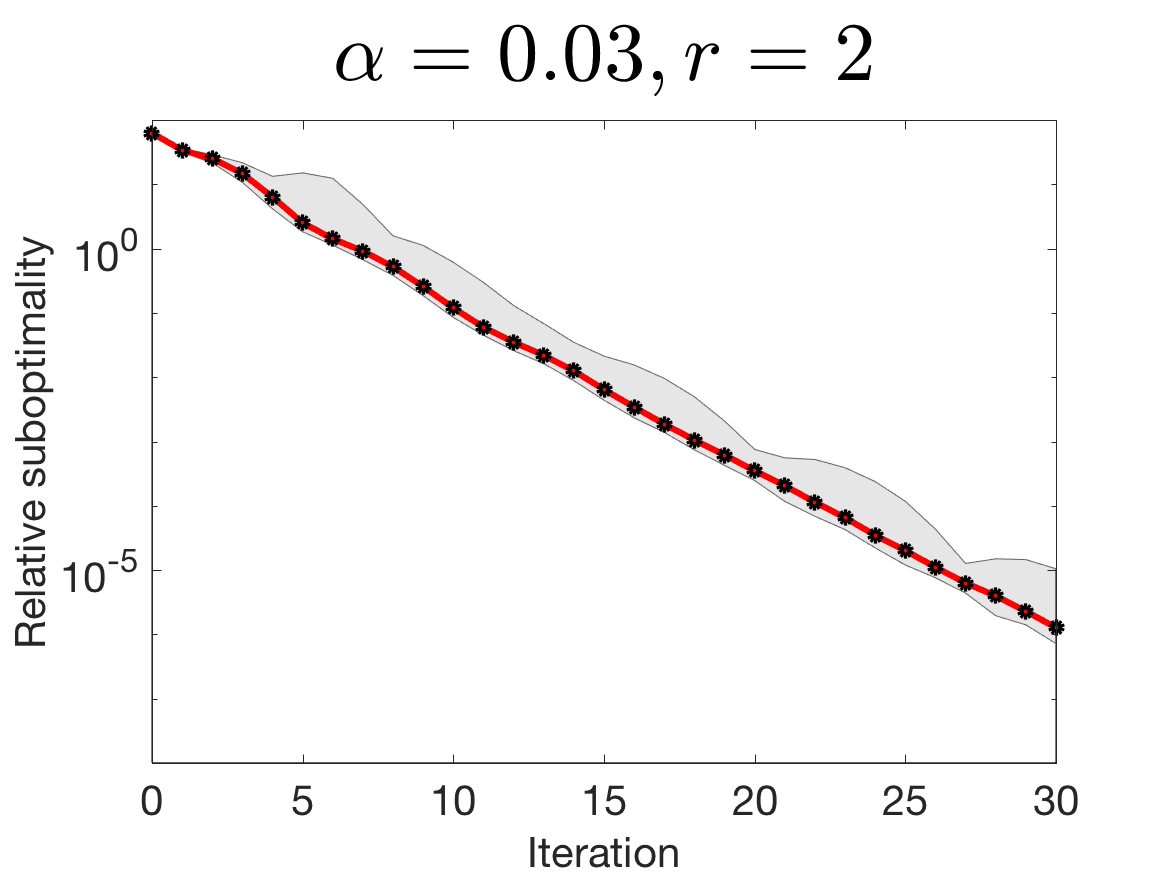}
\end{minipage}%
\begin{minipage}{0.25\textwidth}
  \centering
\includegraphics[width =  \textwidth ]{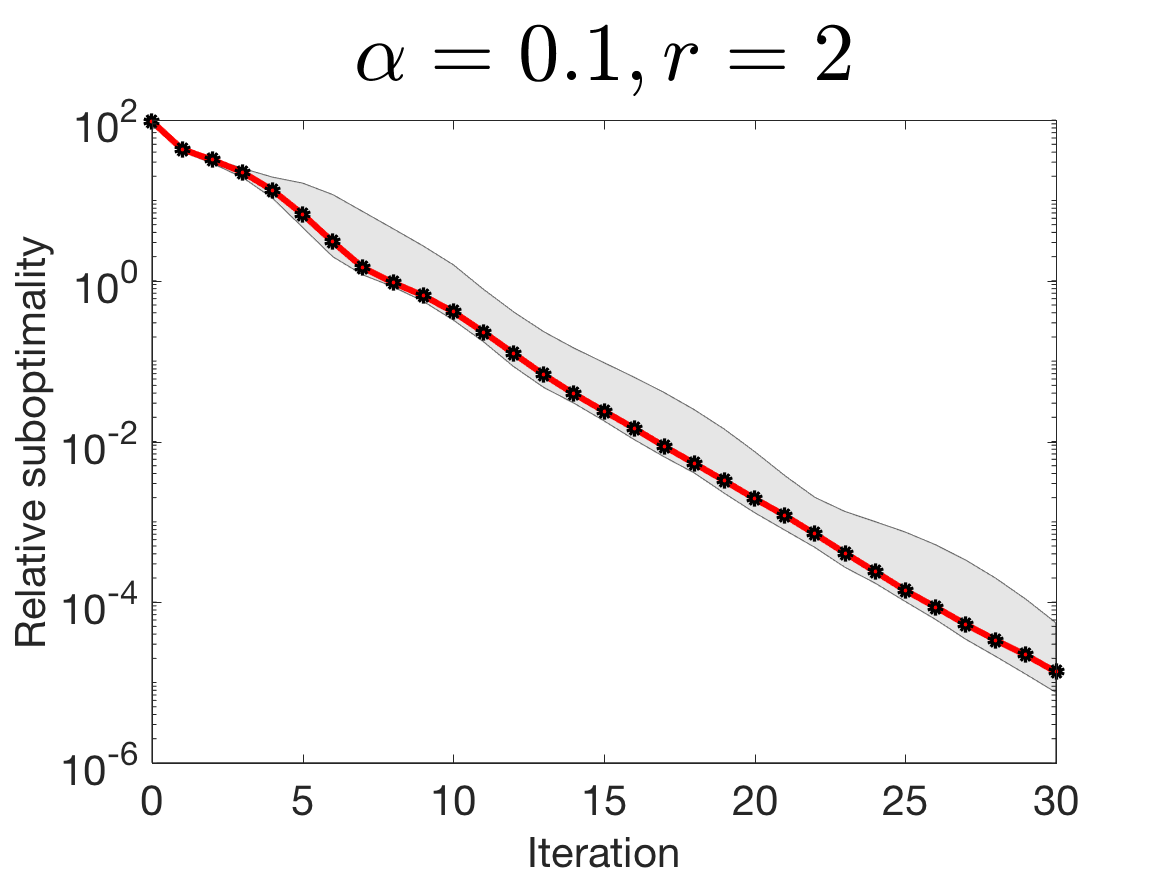}
\end{minipage}%
\begin{minipage}{0.25\textwidth}
  \centering
\includegraphics[width =  \textwidth ]{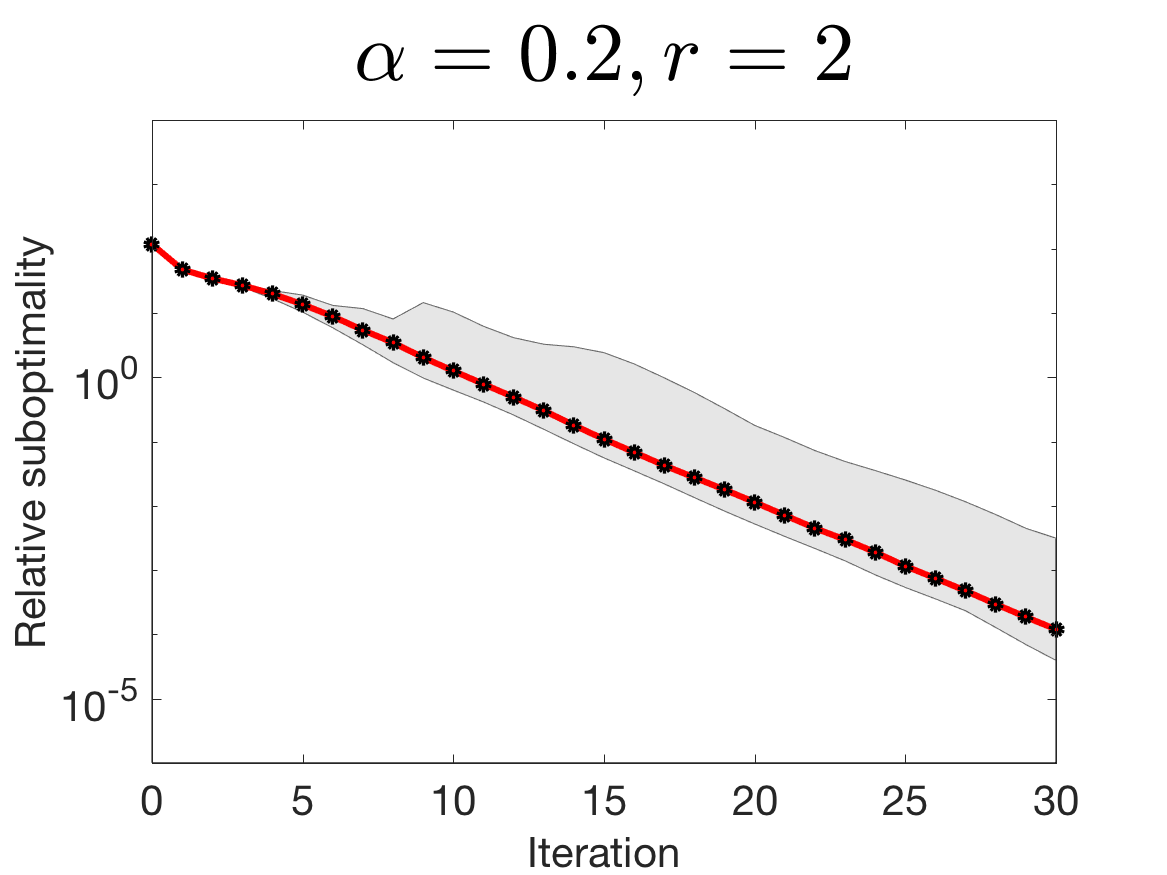}
\end{minipage}%
\begin{minipage}{0.25\textwidth}
  \centering
\includegraphics[width =  \textwidth ]{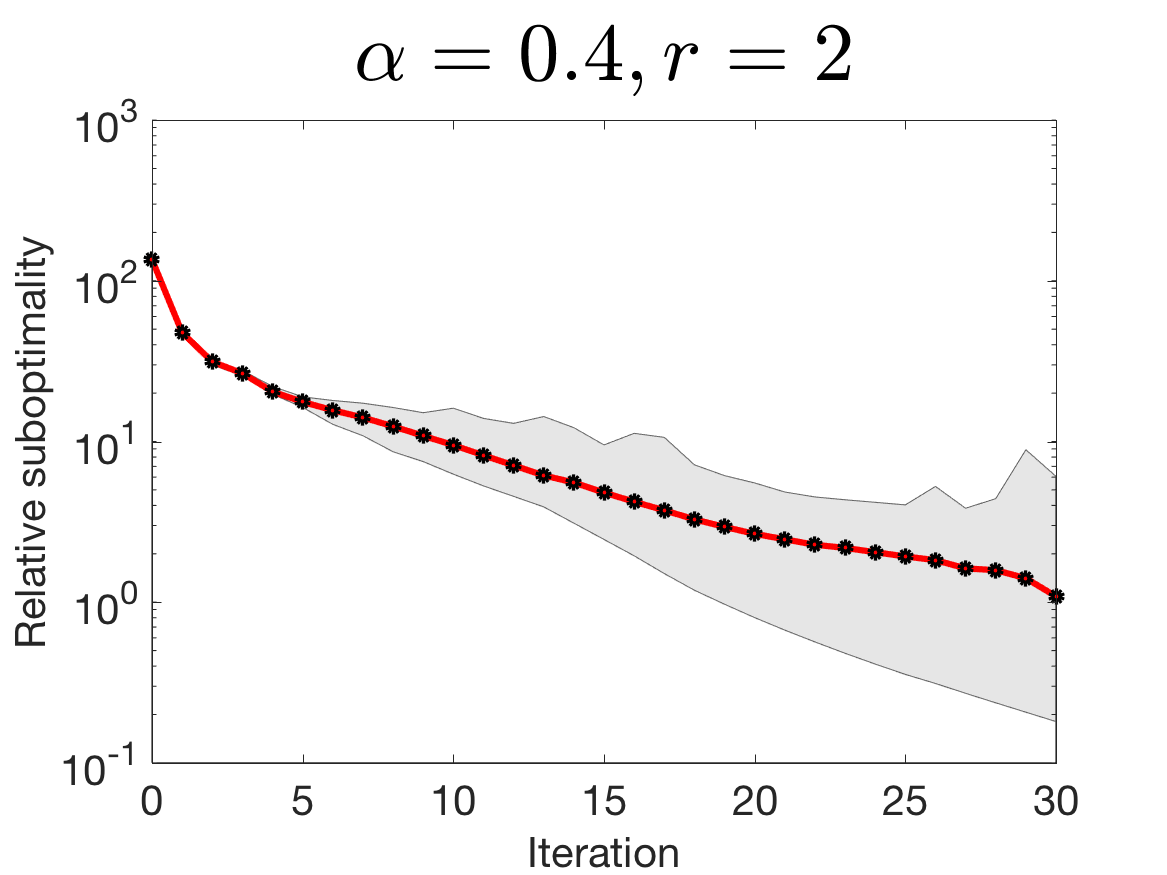}
\end{minipage}%
\\
\begin{minipage}{0.25\textwidth}
  \centering
\includegraphics[width =  \textwidth ]{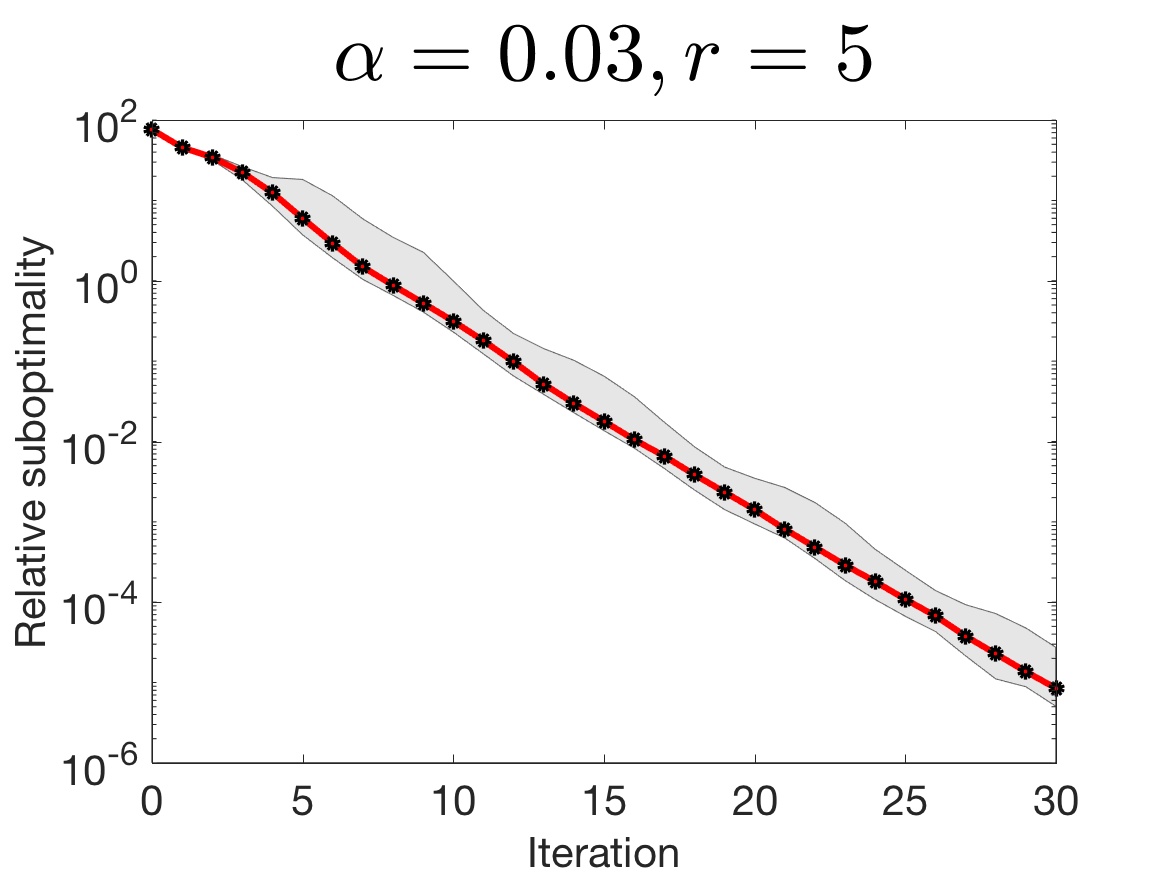}
\end{minipage}%
\begin{minipage}{0.25\textwidth}
  \centering
\includegraphics[width =  \textwidth ]{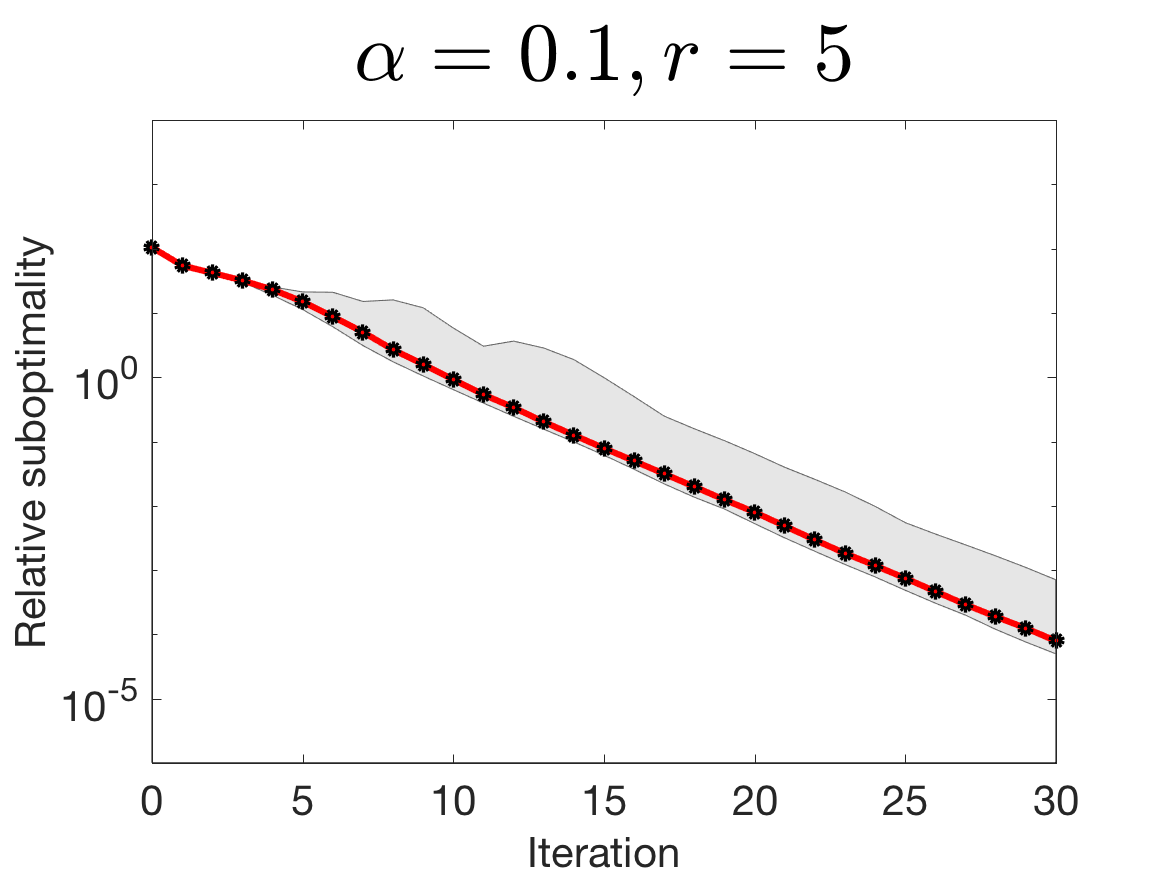}
\end{minipage}%
\begin{minipage}{0.25\textwidth}
  \centering
\includegraphics[width =  \textwidth ]{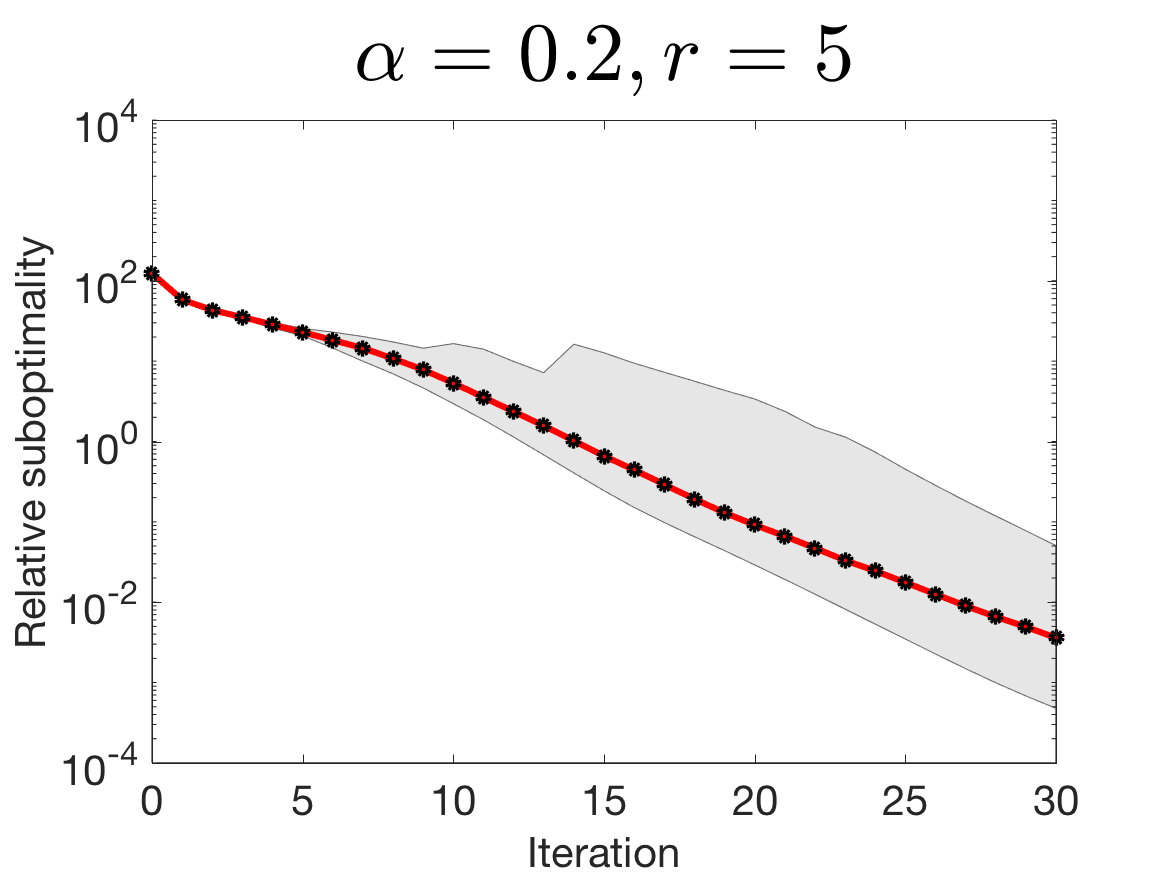}
\end{minipage}%
\begin{minipage}{0.25\textwidth}
  \centering
\includegraphics[width =  \textwidth ]{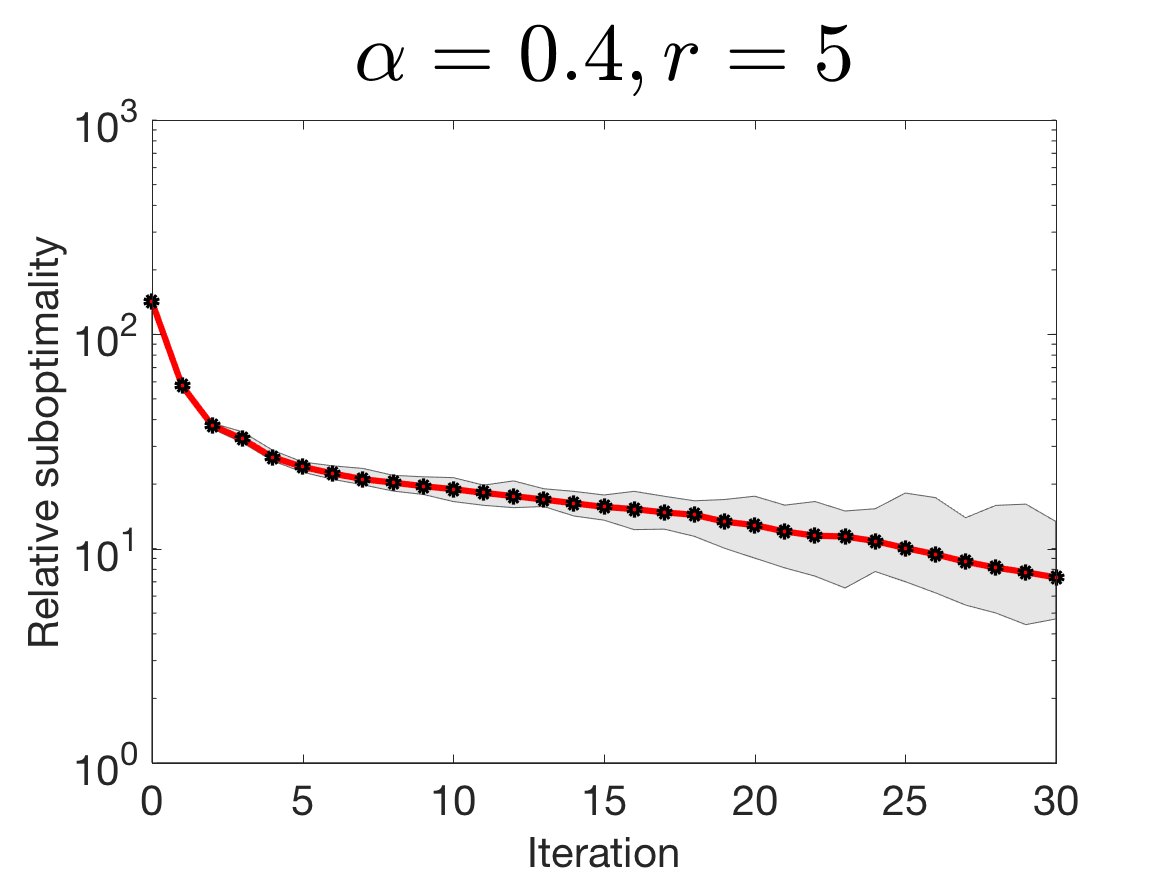}
\end{minipage}%
\\
\begin{minipage}{0.25\textwidth}
  \centering
\includegraphics[width =  \textwidth ]{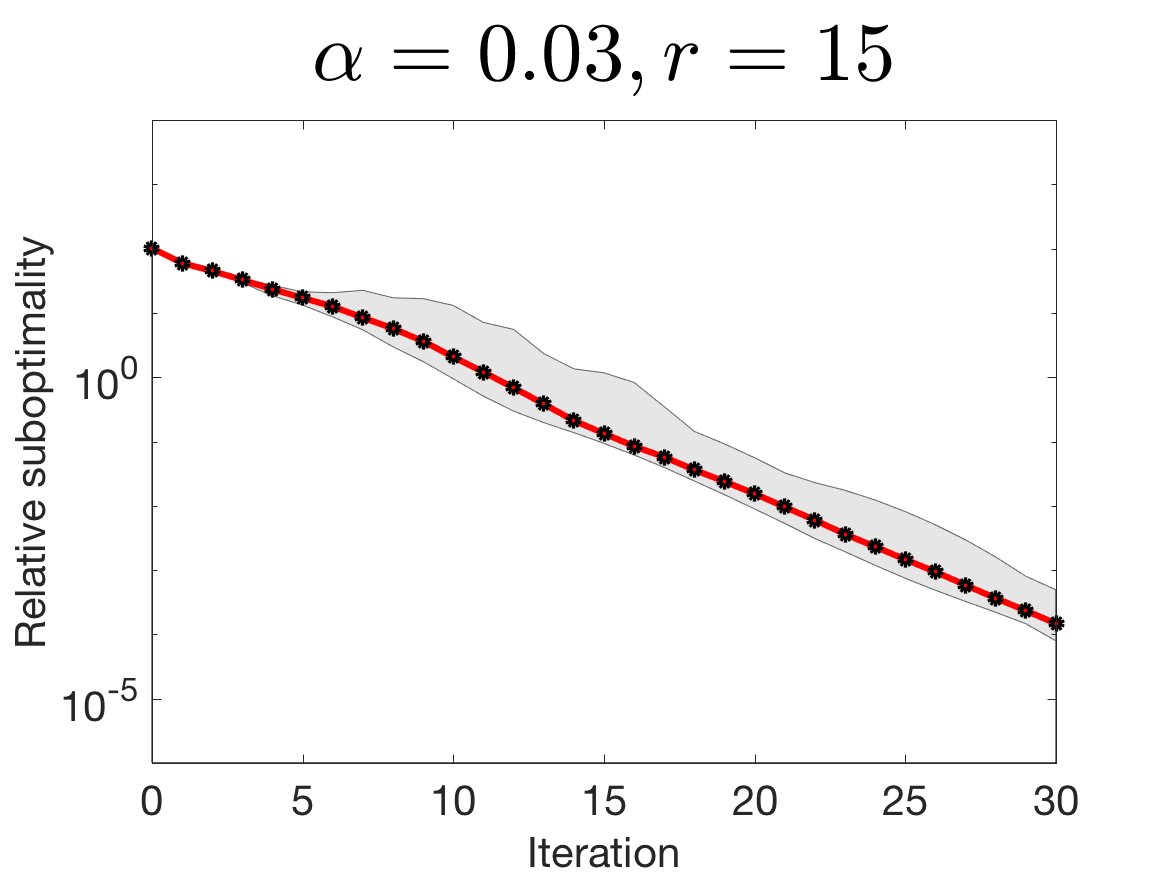}
\end{minipage}%
\begin{minipage}{0.25\textwidth}
  \centering
\includegraphics[width =  \textwidth ]{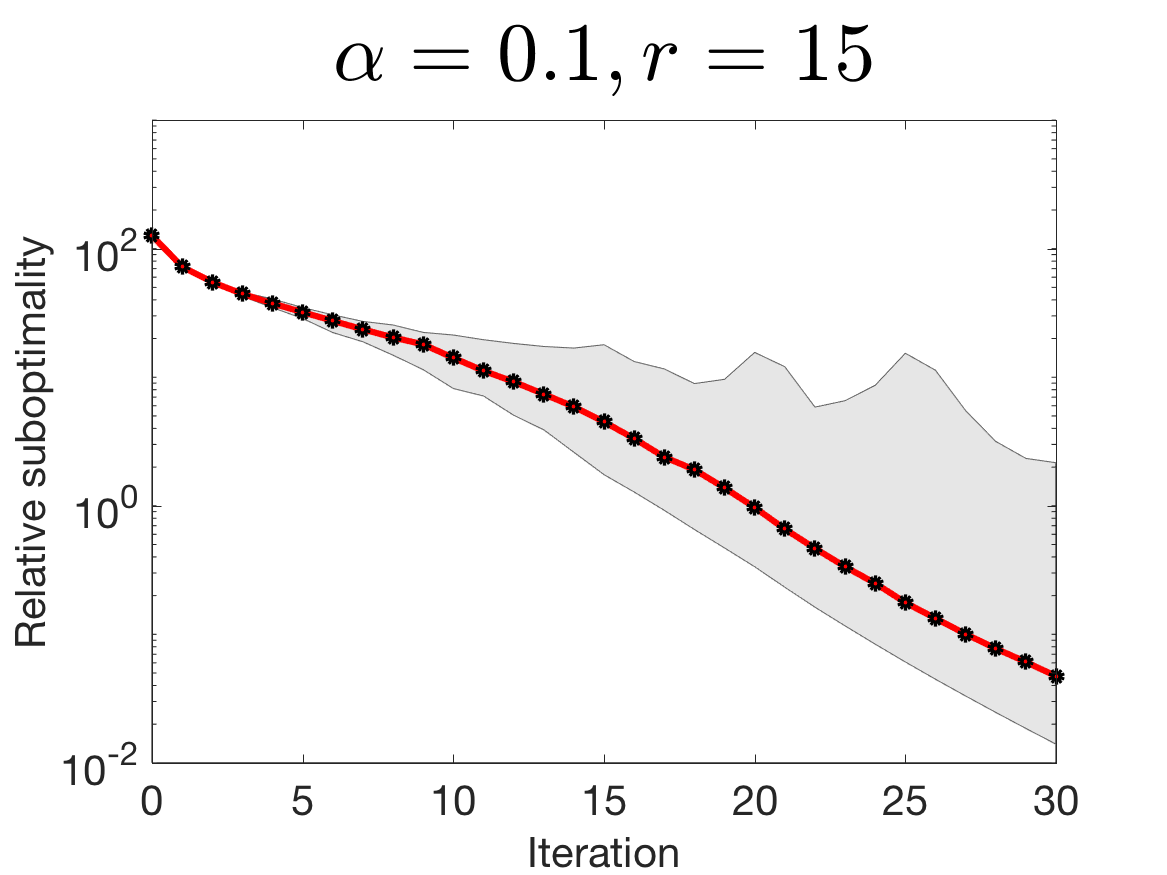}
\end{minipage}%
\begin{minipage}{0.25\textwidth}
  \centering
\includegraphics[width =  \textwidth ]{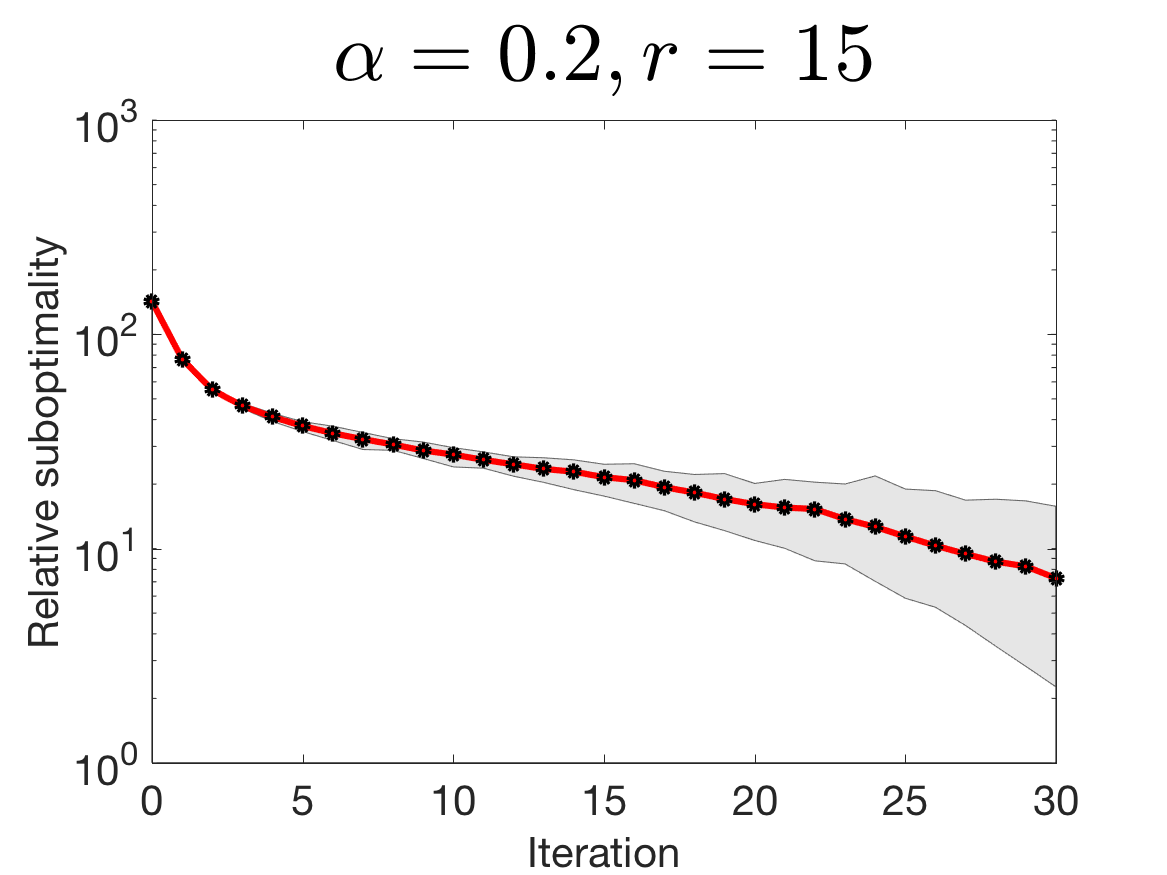}
\end{minipage}%
\begin{minipage}{0.25\textwidth}
  \centering
\includegraphics[width =  \textwidth ]{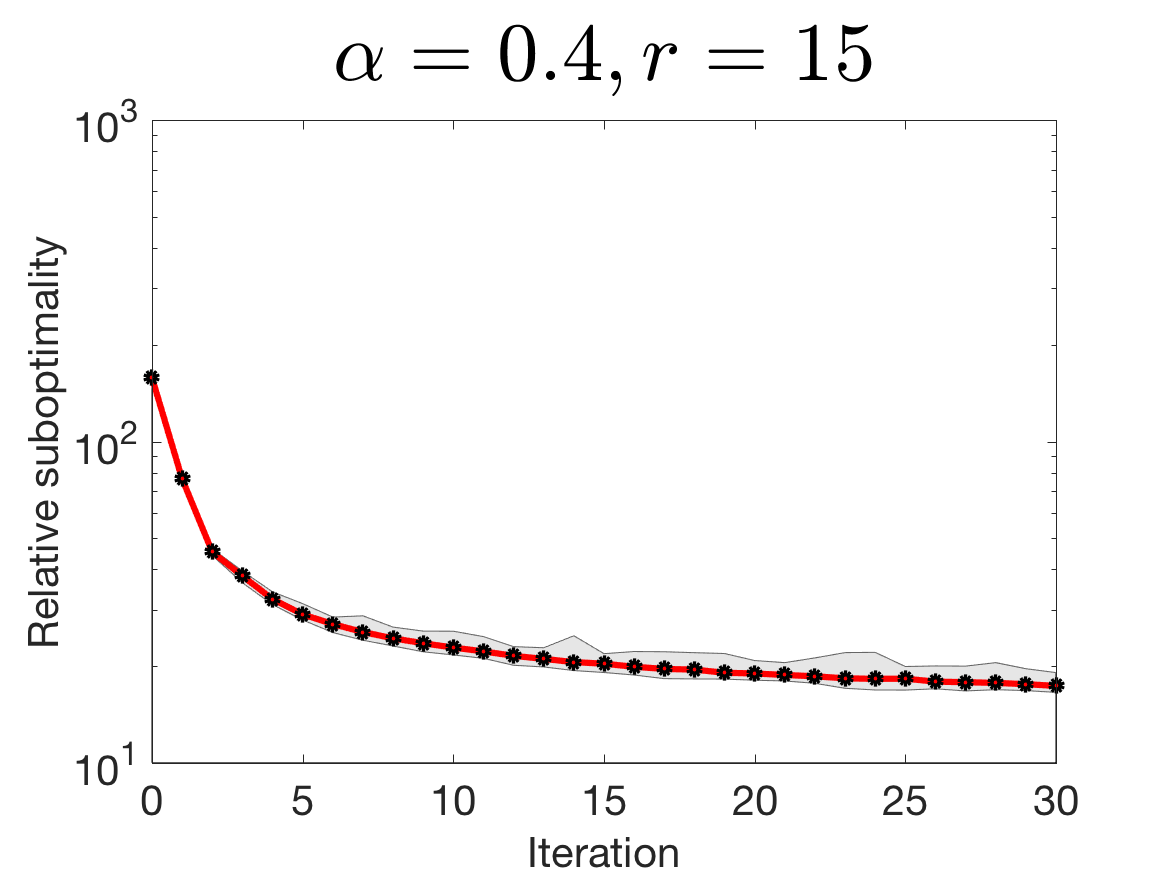}
\end{minipage}%
\\
\begin{minipage}{0.25\textwidth}
  \centering
\includegraphics[width =  \textwidth ]{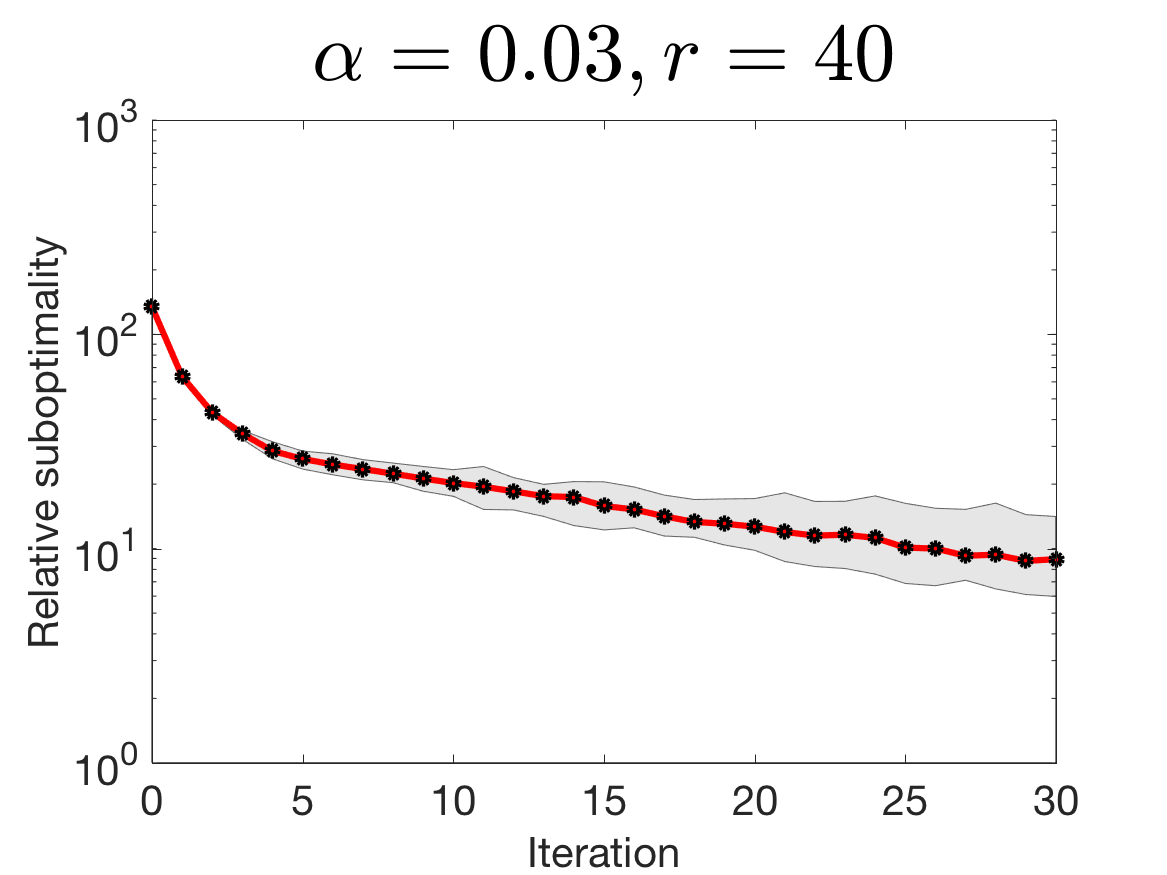}
\end{minipage}%
\begin{minipage}{0.25\textwidth}
  \centering
\includegraphics[width =  \textwidth ]{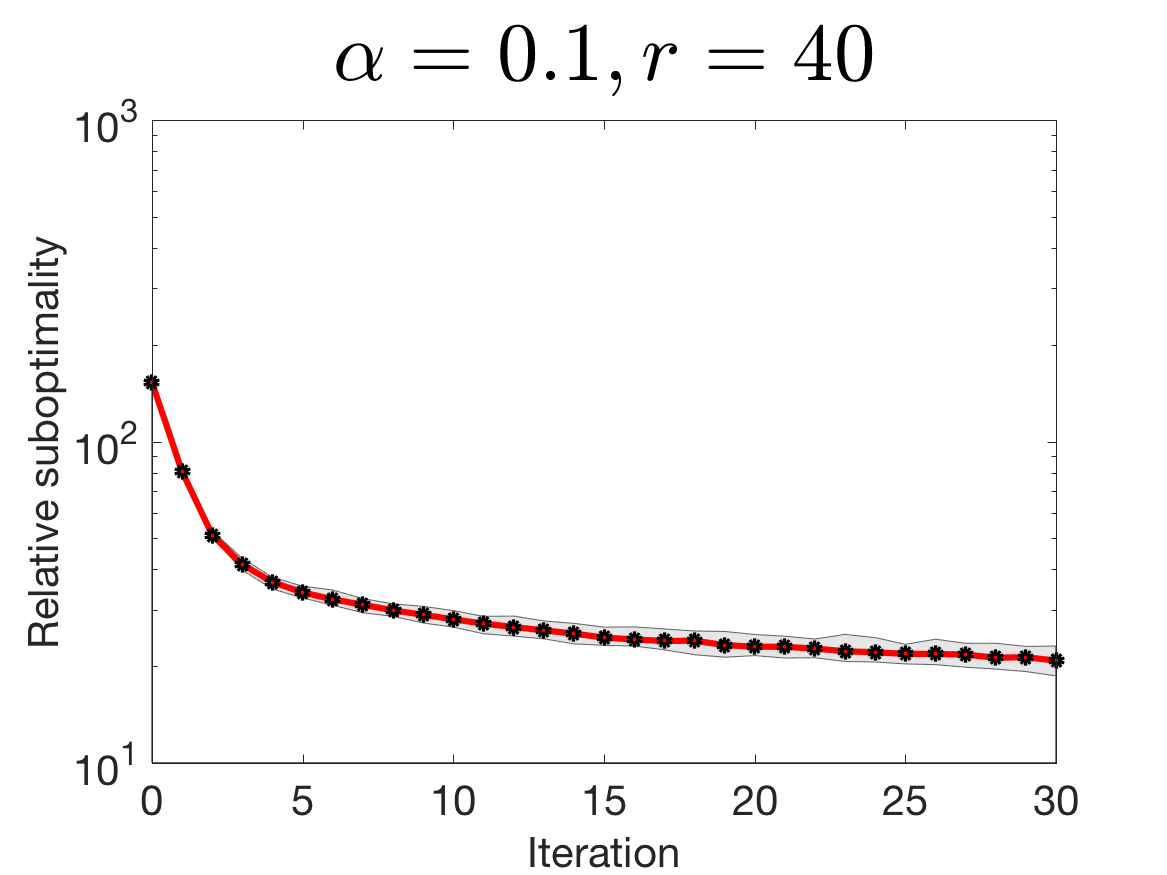}
\end{minipage}%
\begin{minipage}{0.25\textwidth}
  \centering
\includegraphics[width =  \textwidth ]{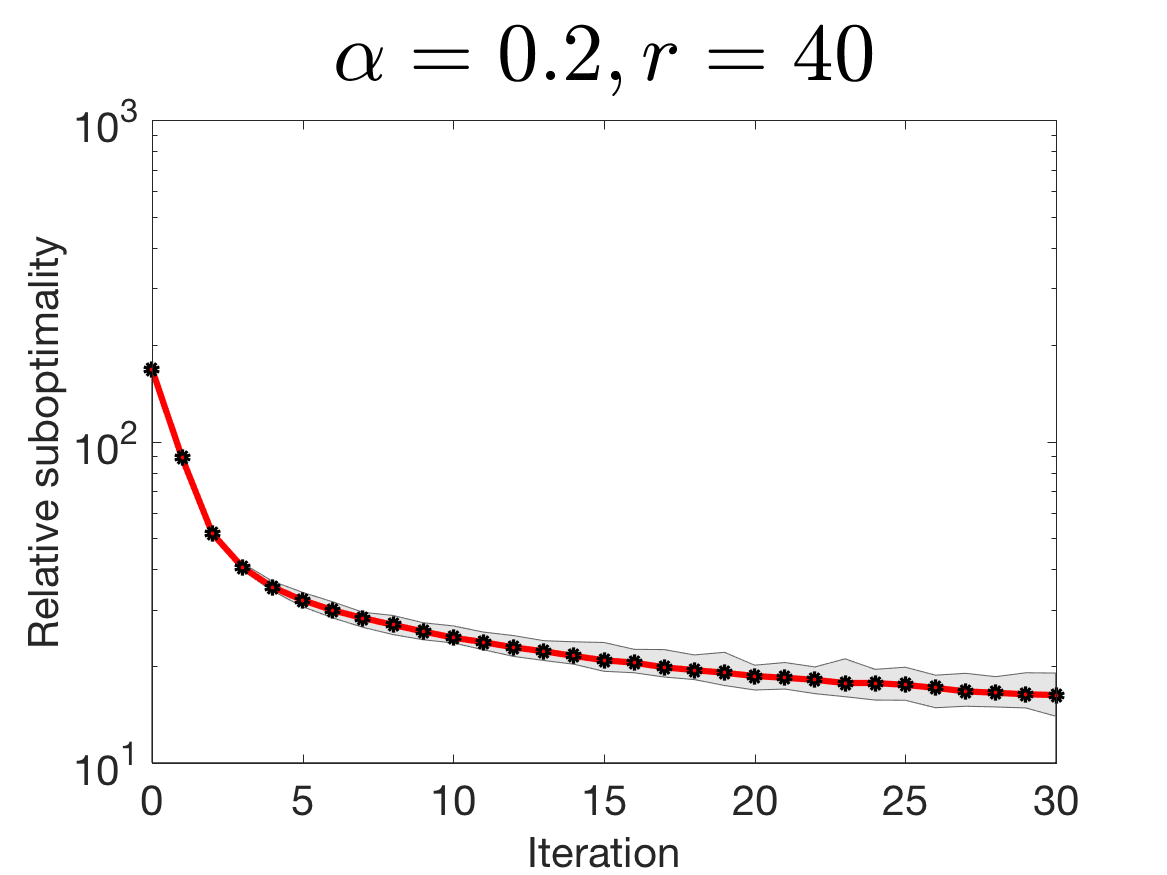}
\end{minipage}%
\begin{minipage}{0.25\textwidth}
  \centering
\includegraphics[width =  \textwidth ]{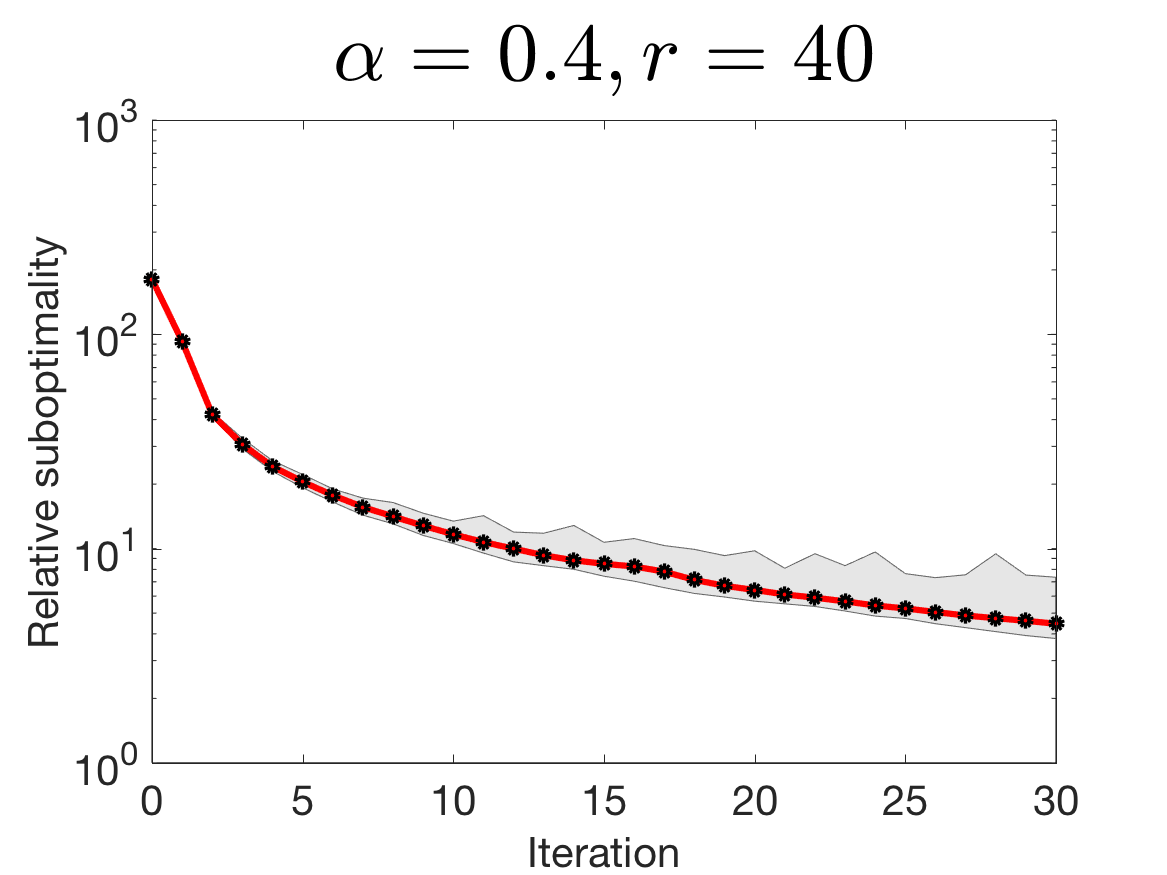}
\end{minipage}%
\caption{ Sensitivity of Algorithm~\ref{alg:apg} with respect to the starting point.}\label{fig:sens3}
\end{figure}

\begin{figure}[h]
    \centering
    \begin{minipage}{0.32\textwidth}
    \includegraphics[width=\textwidth]{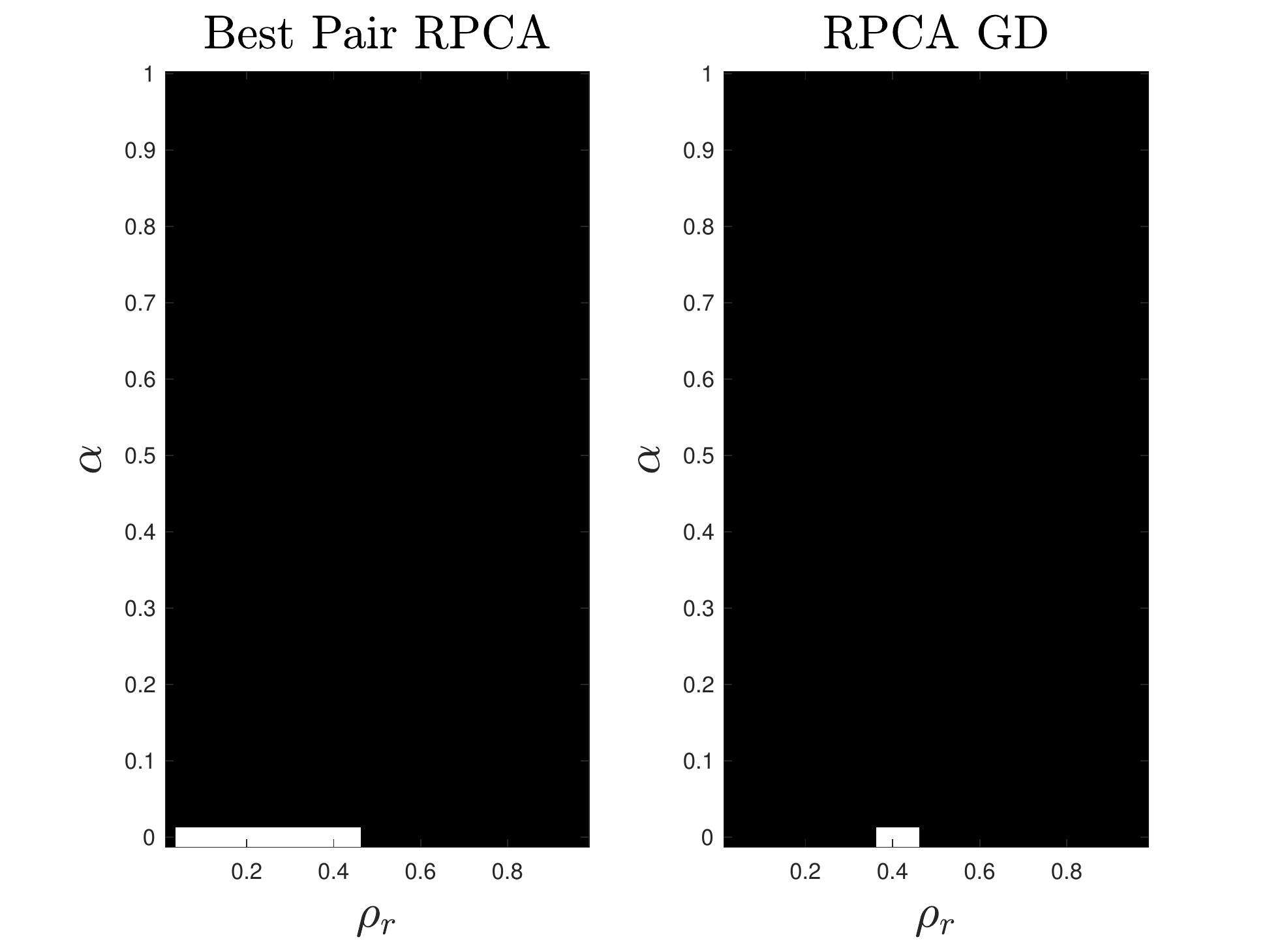}
      \end{minipage}
    \begin{minipage}{0.32\textwidth}
    \includegraphics[width = \textwidth]{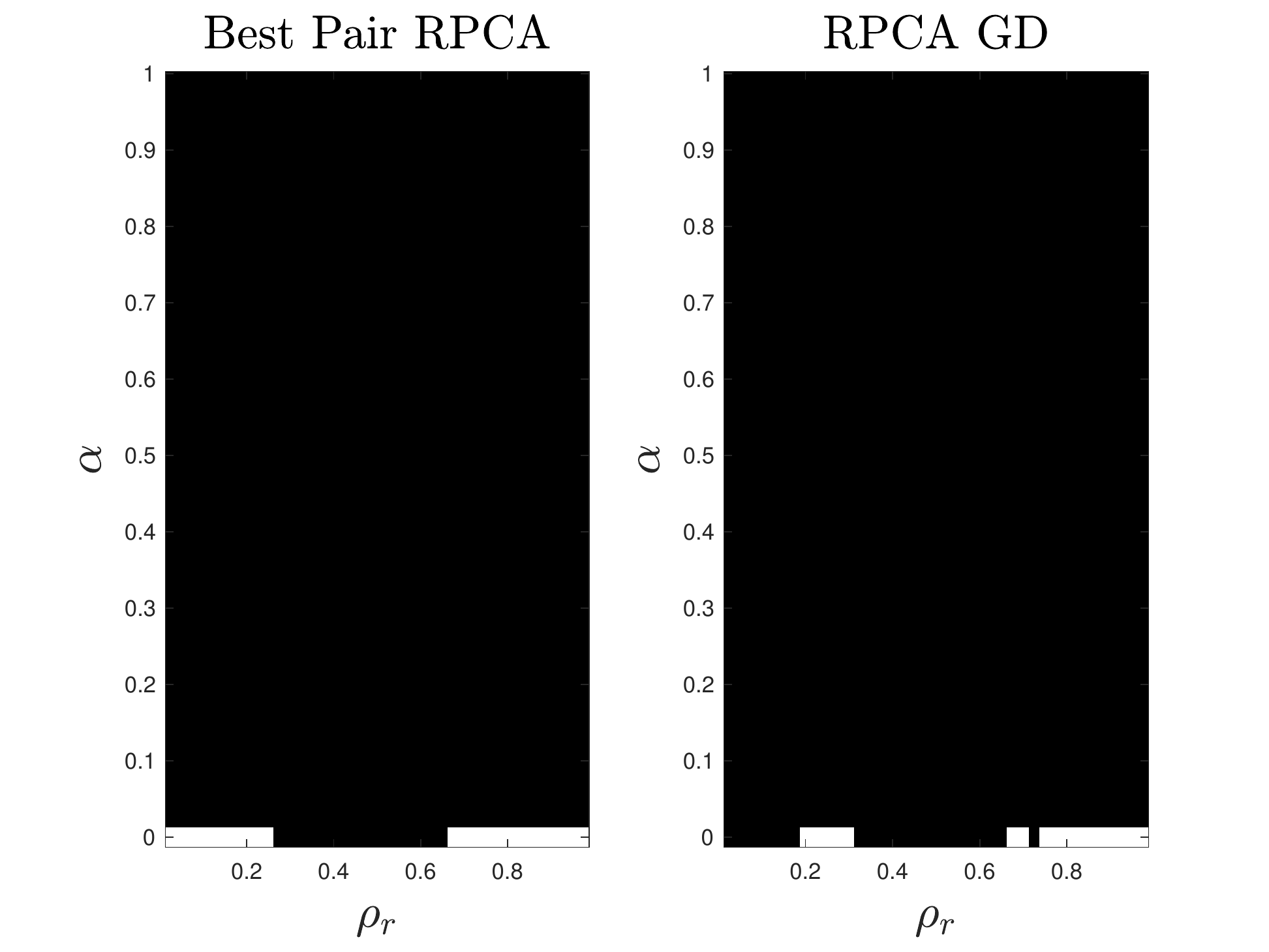}
  \end{minipage} 
   \begin{minipage}{0.32\textwidth}
    \includegraphics[width = \textwidth]{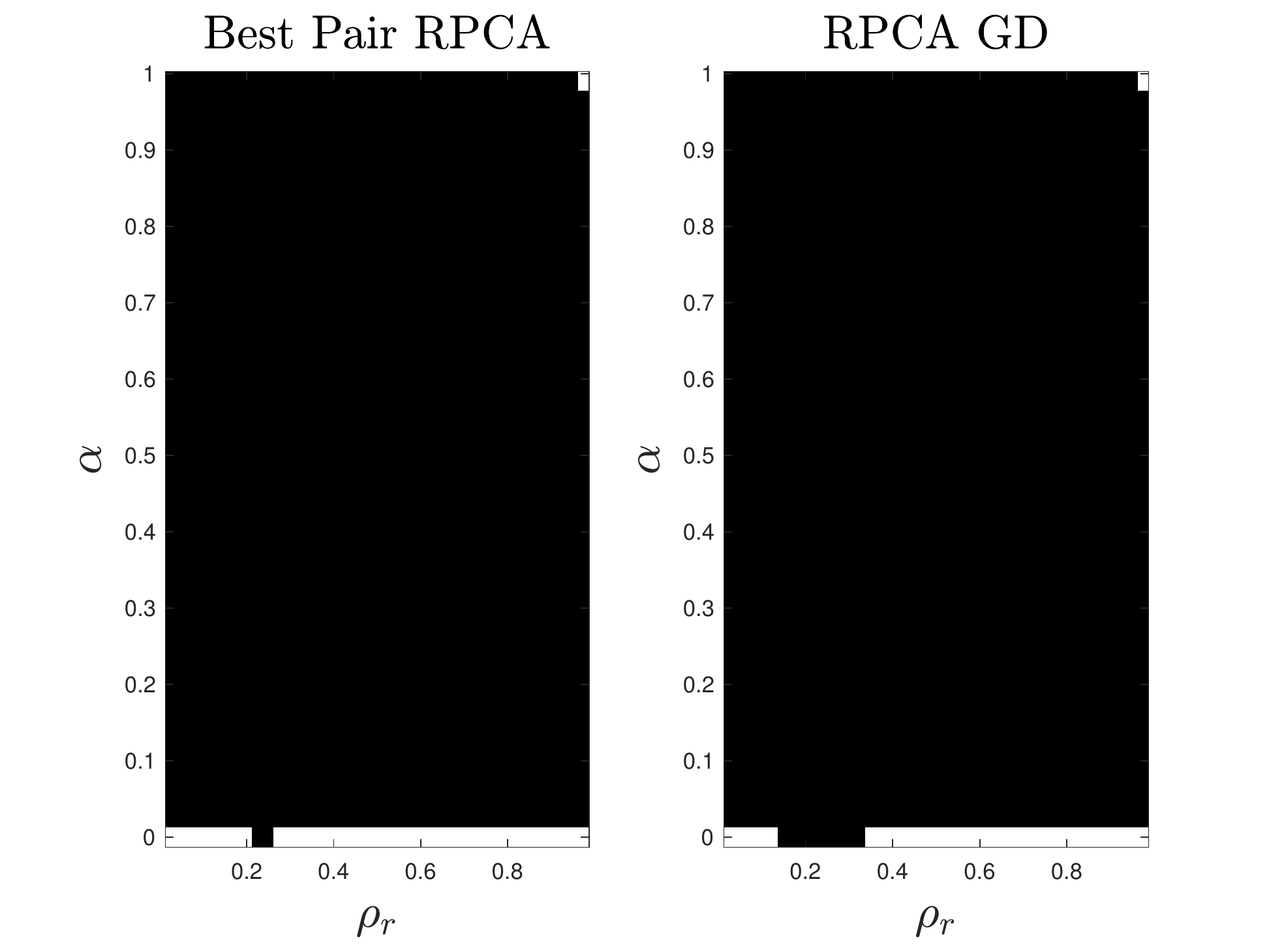}
  \end{minipage} 
\caption{\small{Phase transition diagram to compare between RPCA GD and best-pair RPCA with respect to rank and error sparsity. We set $\rho_r={\rm rank}(L)/m$ and $\alpha$ is the sparsity measure. We have $(\rho_r,\alpha)\in (0.025,1]\times(0,1)$ with $r=5:5:200$ and $\alpha = {\tt linspace}(0,0.99,40)$.}}
    \label{syntheticdata3}
\end{figure}

\begin{figure}[h]
    \centering
    \begin{minipage}{0.49\textwidth}
    \includegraphics[width=\textwidth]{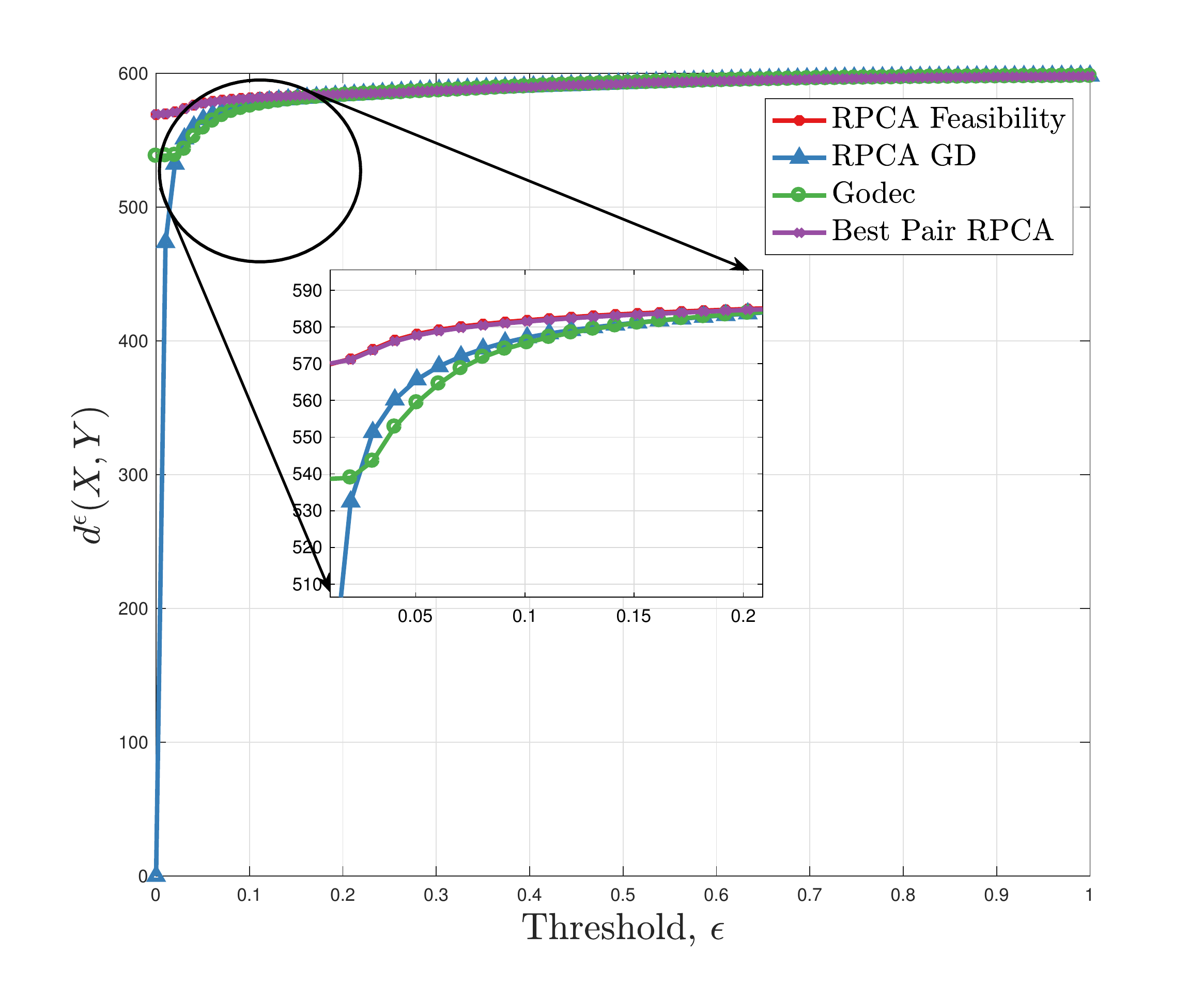}
      \end{minipage}
    \begin{minipage}{0.5\textwidth}
    \includegraphics[width = \textwidth, height = 2.2in]{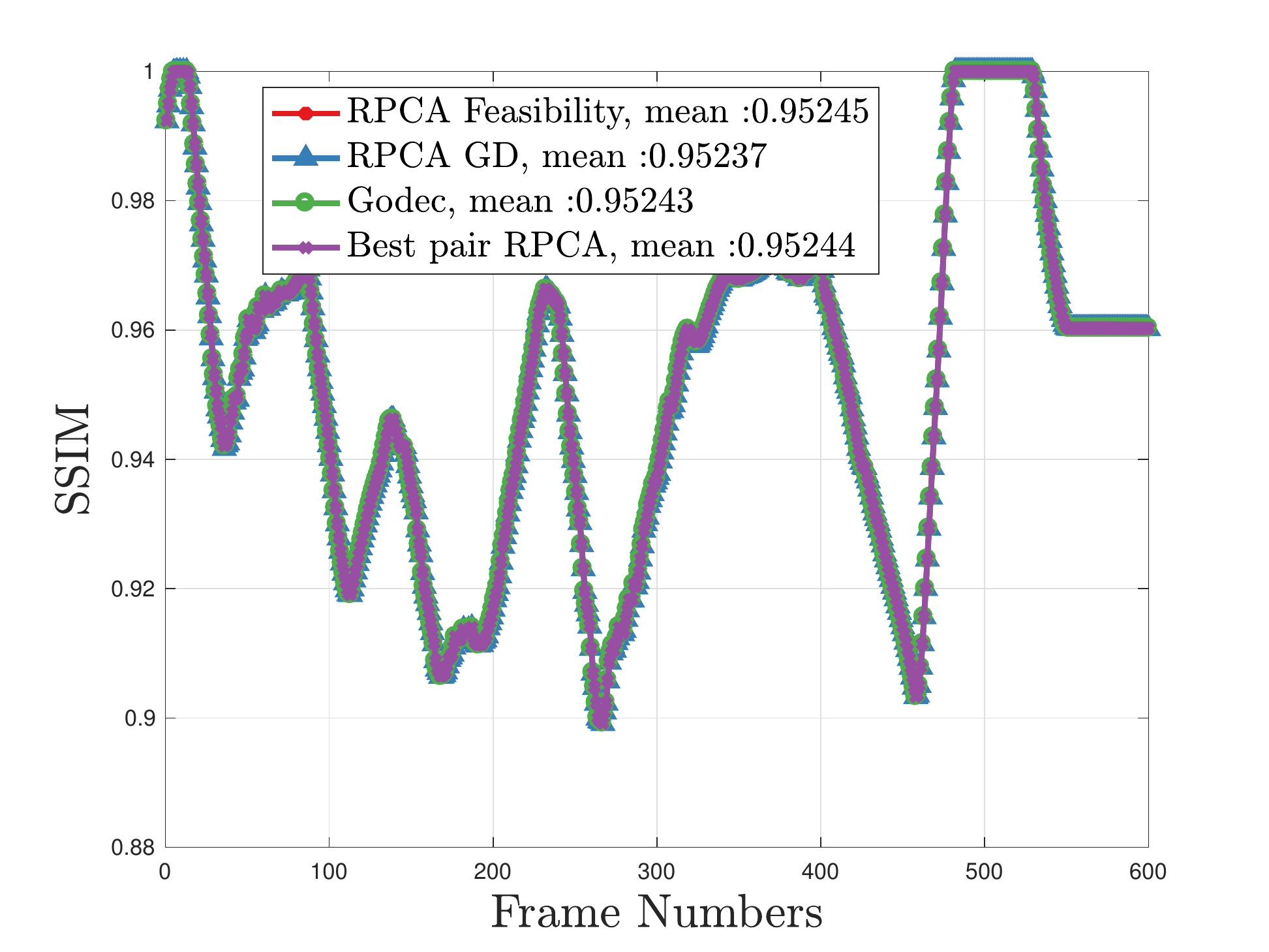}
  \end{minipage} 
\caption{{Quantitative comparison between different algorithms on Stuttgart {\tt Basic} sequence. We compare the recovered foreground by different methods with respect to the foreground GT available for each frame on two different metrics: $\epsilon$-proximity metric--$d_{\epsilon}(X,Y)$ as in \citep{duttahanzely} and structural similarity index measure by \citep{mssim}. In recovering the foreground objects, our best pair RPCA is as robust as the other baseline methods with respect to both metrics.}}
    \label{qunatitative_BG}
\end{figure}

\begin{figure}[h]
    \centering
    \begin{minipage}{0.32\textwidth}
    \includegraphics[width=\textwidth]{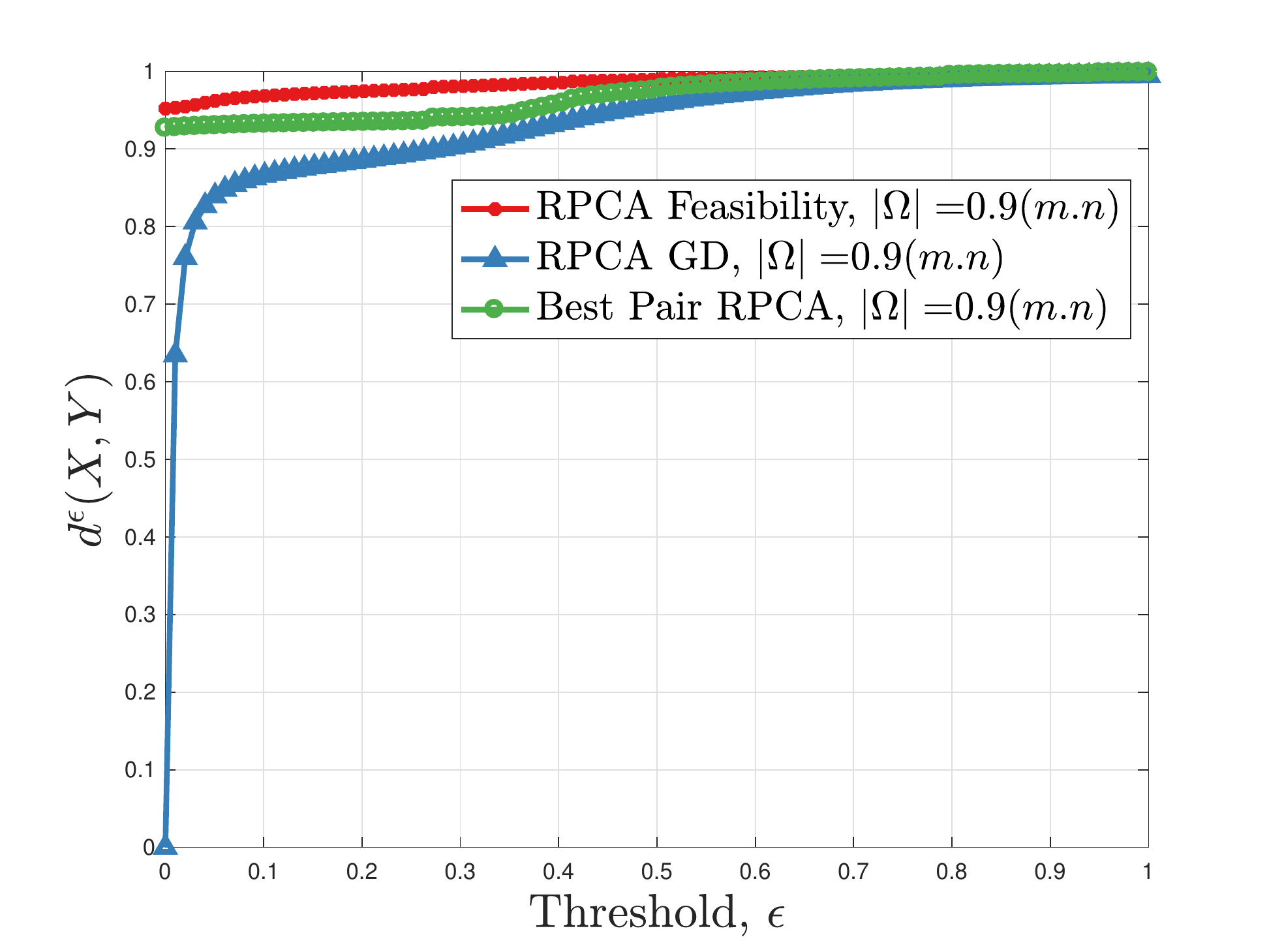}
      \end{minipage}
    \begin{minipage}{0.32\textwidth}
    \includegraphics[width = \textwidth]{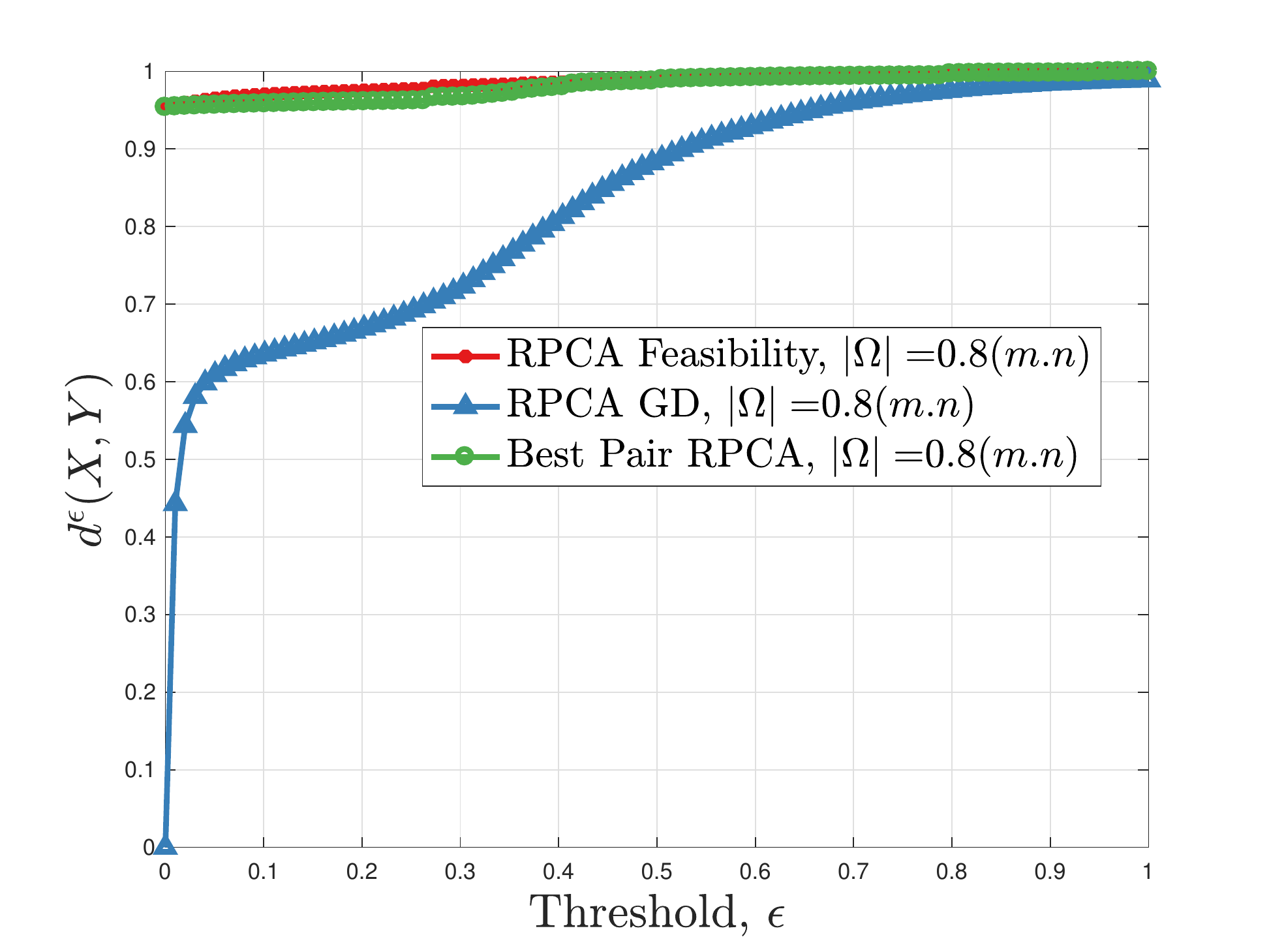}
  \end{minipage} 
   \begin{minipage}{0.32\textwidth}
    \includegraphics[width = \textwidth]{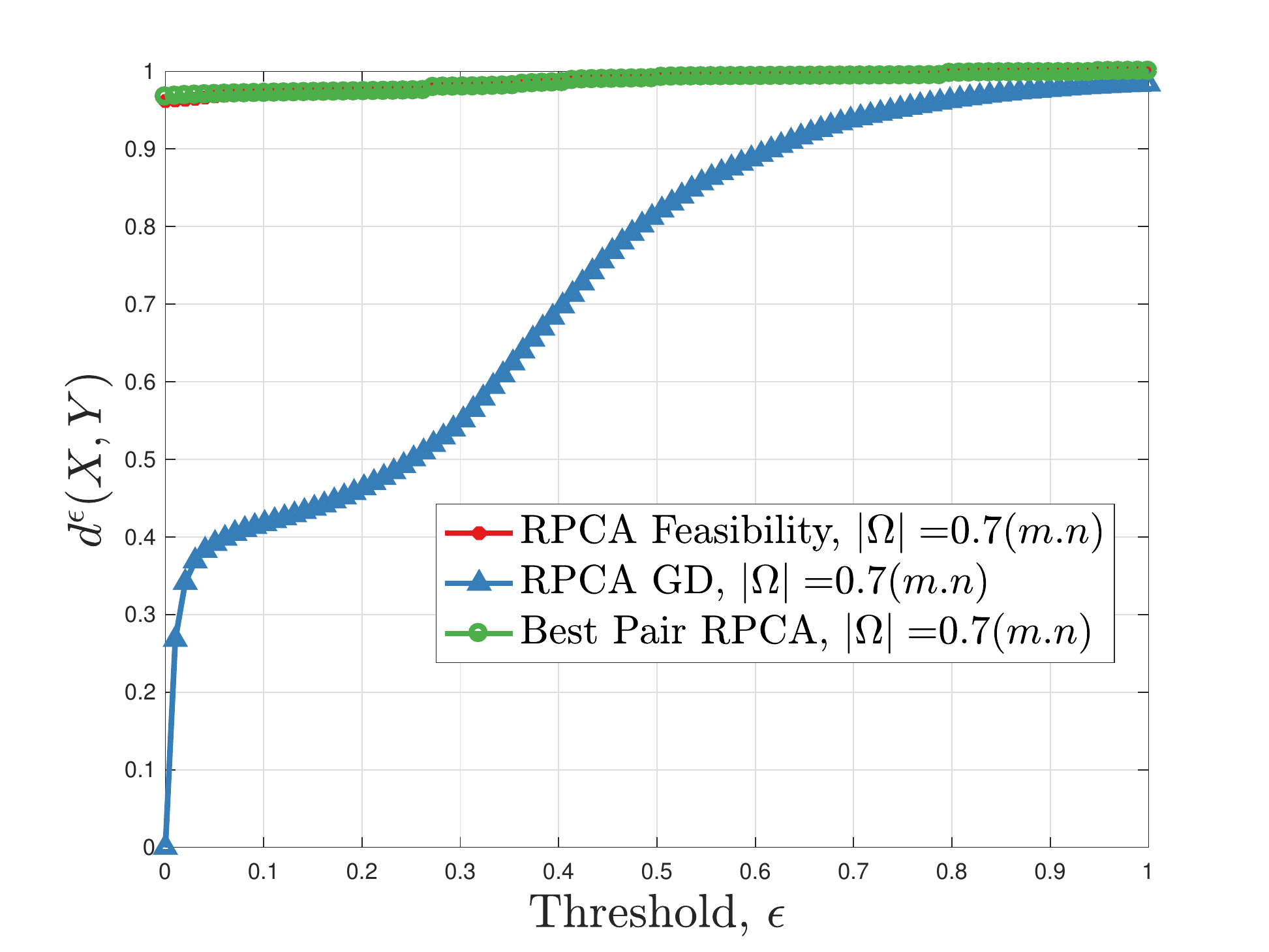}
  \end{minipage} 
  \\
  \begin{minipage}{0.32\textwidth}
    \includegraphics[width=\textwidth]{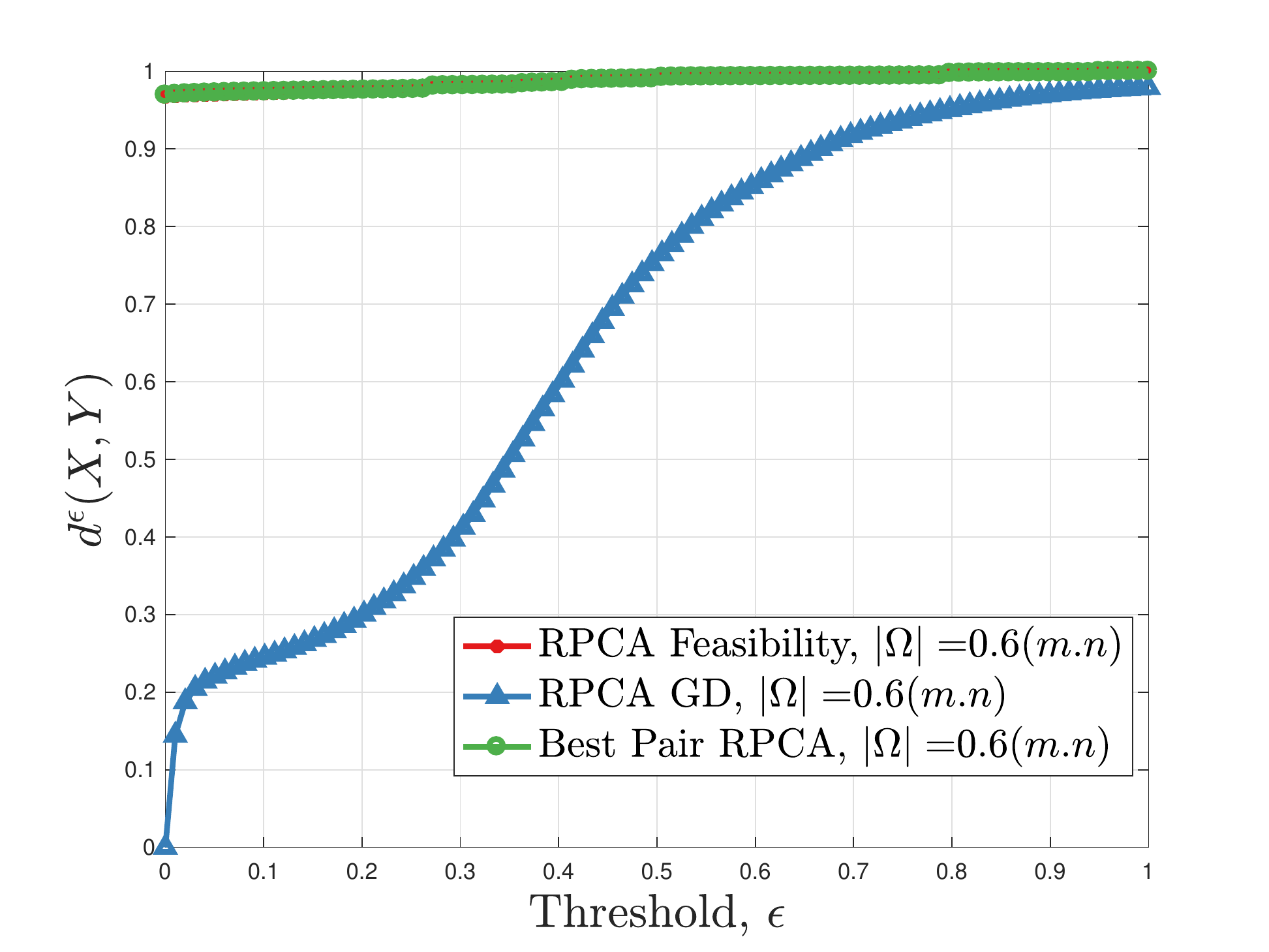}
      \end{minipage}
    \begin{minipage}{0.32\textwidth}
    \includegraphics[width = \textwidth]{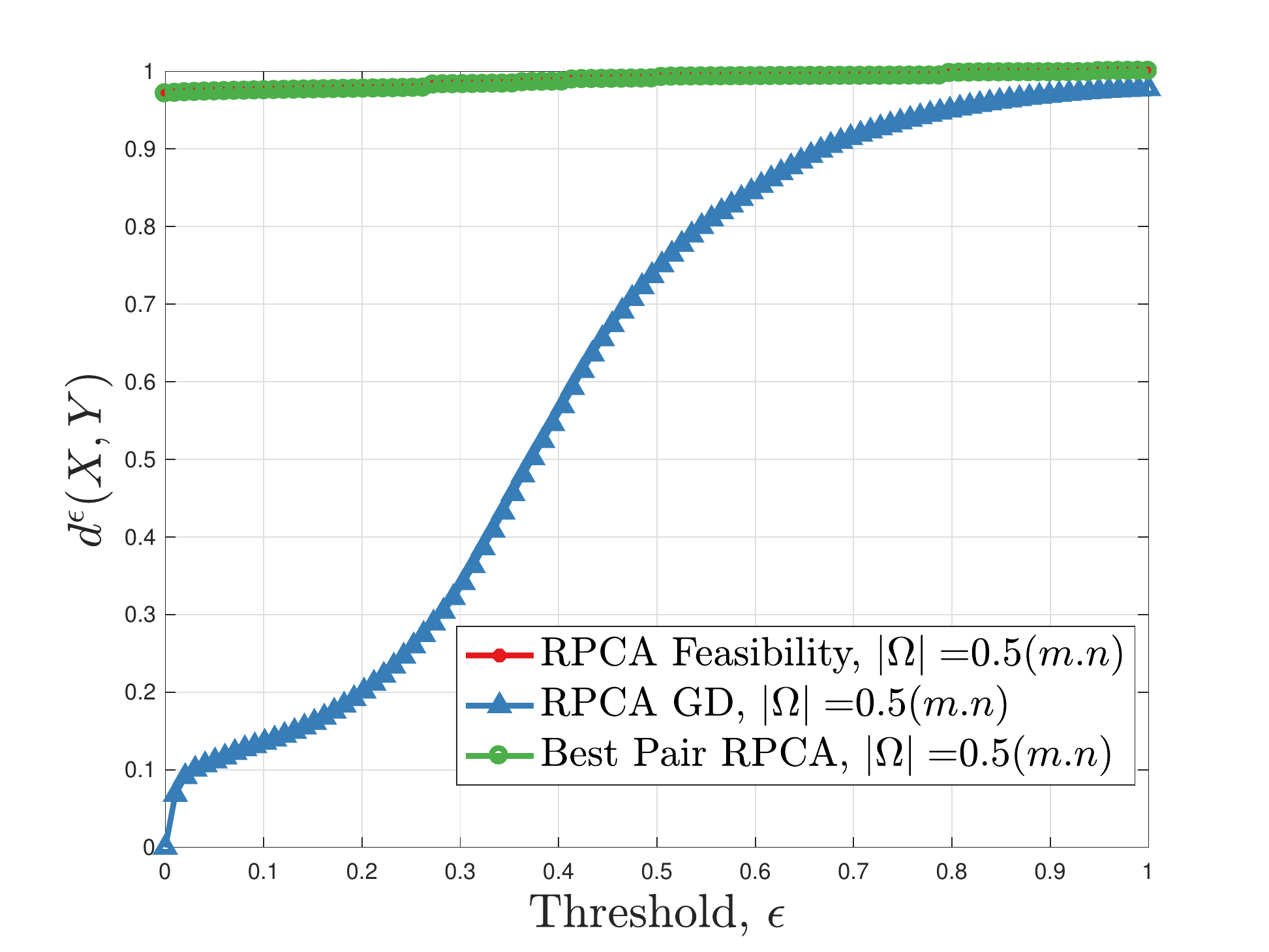}
  \end{minipage} 
   \begin{minipage}{0.32\textwidth}
    \includegraphics[width = \textwidth]{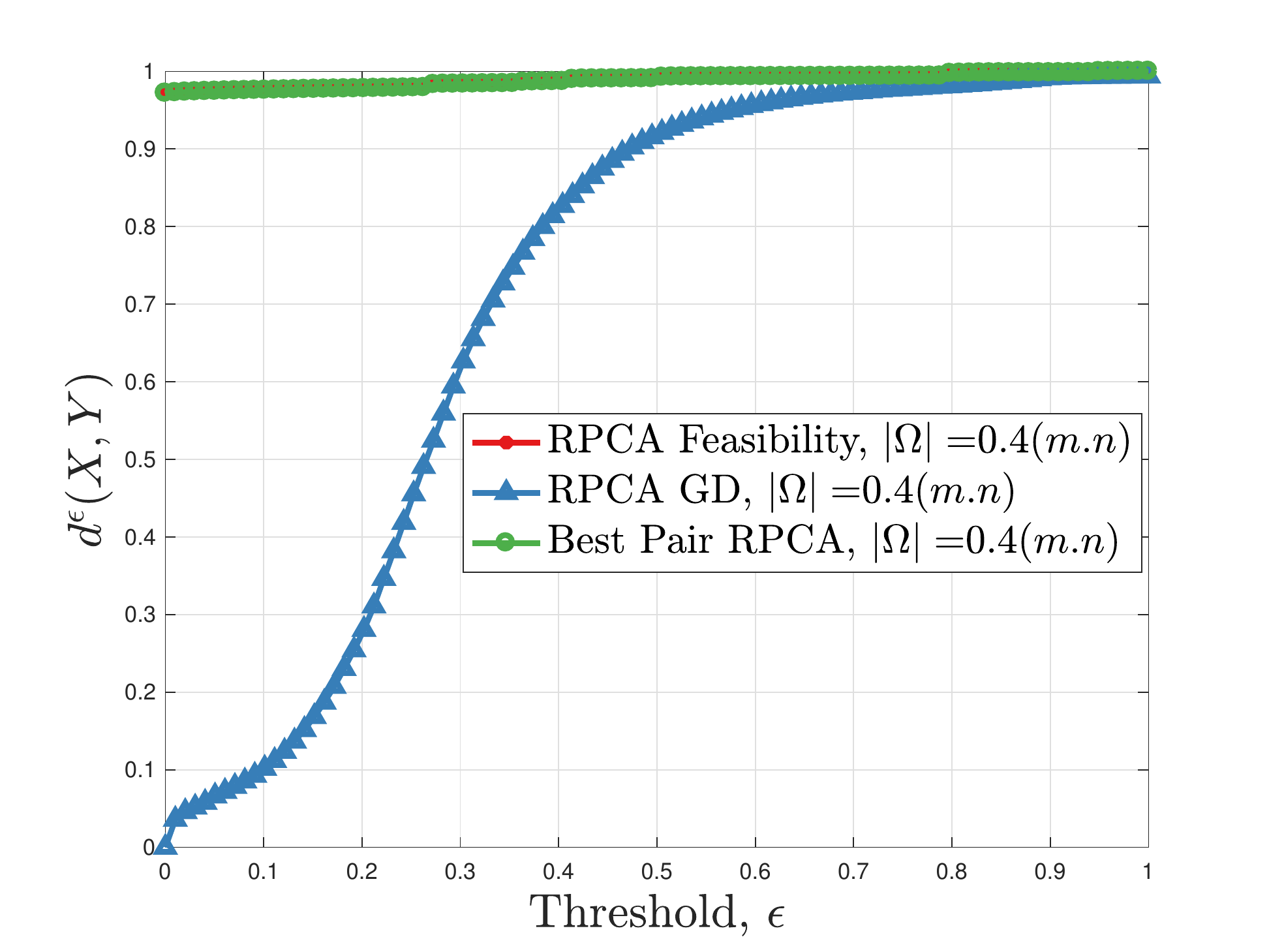}
  \end{minipage} 
\caption{\small{Quantitative comparison of foreground recovered by best pair RPCA, RPCA GD, and RPCA NCFF Stuttgart {\tt Basic} sequence, frame size $144\times176$ with observable entries: (a) $|\Omega|=0.9(m.n)$, (b) $|\Omega|=0.8(m.n),$ (c) $|\Omega|=0.7(m.n)$, (d) $|\Omega|=0.6(m.n)$, (e) $|\Omega|=0.5(m.n)$, and (f) $|\Omega|=0.4(m.n)$. The performance of RPCA GD drops significantly as $|\Omega|$ decreases. In contrast, the performance of our best pair RPCA and RPCA NCF stay stable irrespective of the size of $|\Omega|$.}}
    \label{qunatitative_BG_1}
\end{figure}

\subsubsection{Inlier detection}\label{sec:inlier detection}

Historically, PCA and RPCA are used in detecting the inliers and the outliers from a composite dataset. We infused  400 random, grayscale, downsampled~($20\times20$ pixels) natural images from the BACKGROUND/Google folder of the Caltech101 database \citep{caltech101} with the {\tt Yale Extended Face Database} to construct the data set. The inliers are the grayscale images of faces~(of the same resolution) under different illuminations while the 400 random natural images serve as outliers. The goal is to consider a low-dimensional model and to project the inliers to a 9-dimensional linear subspace where the images of the same face lie.~Goes et al.\  in \citep{rspca} designed seven algorithms to explicitly find a low-rank subspace. To this end, Goes et al.\ used the classical SGD, an incremental approach, and mirror descent algorithms to find the 9-dimensional subspace. However, we split the dataset, $A$, into a 9-dimensional low-rank subspace $L$ and expect the outliers to be in the sparse set, $S$. Once we find $L$, we find the basis of $L$ via orthogonalization and project the faces on it.  In Figure \ref{inlier_detect}, we show the qualitative results of our experiments\footnote{The codes and datasets for experiments in Section \ref{sec:inlier detection} are obtained from {https://github.com/jwgoes/RSPCA}}.

As proposed in \citep{rspca}, we use the normalized error term $\|P_L-P_{L^*}\|_F/{3\sqrt{2}}$, where $L$ is subspace fitted by the PCA to the set of inliers and $L^*$ be the subspace fitted by different algorithms. Note that, the metric is expected to lie between 0 and 1 where the smaller is the better. 
We refer to Table \ref{inlier_data} for our quantitative results.

\begin{figure}[h]
    \centering
    \includegraphics[width = \textwidth]{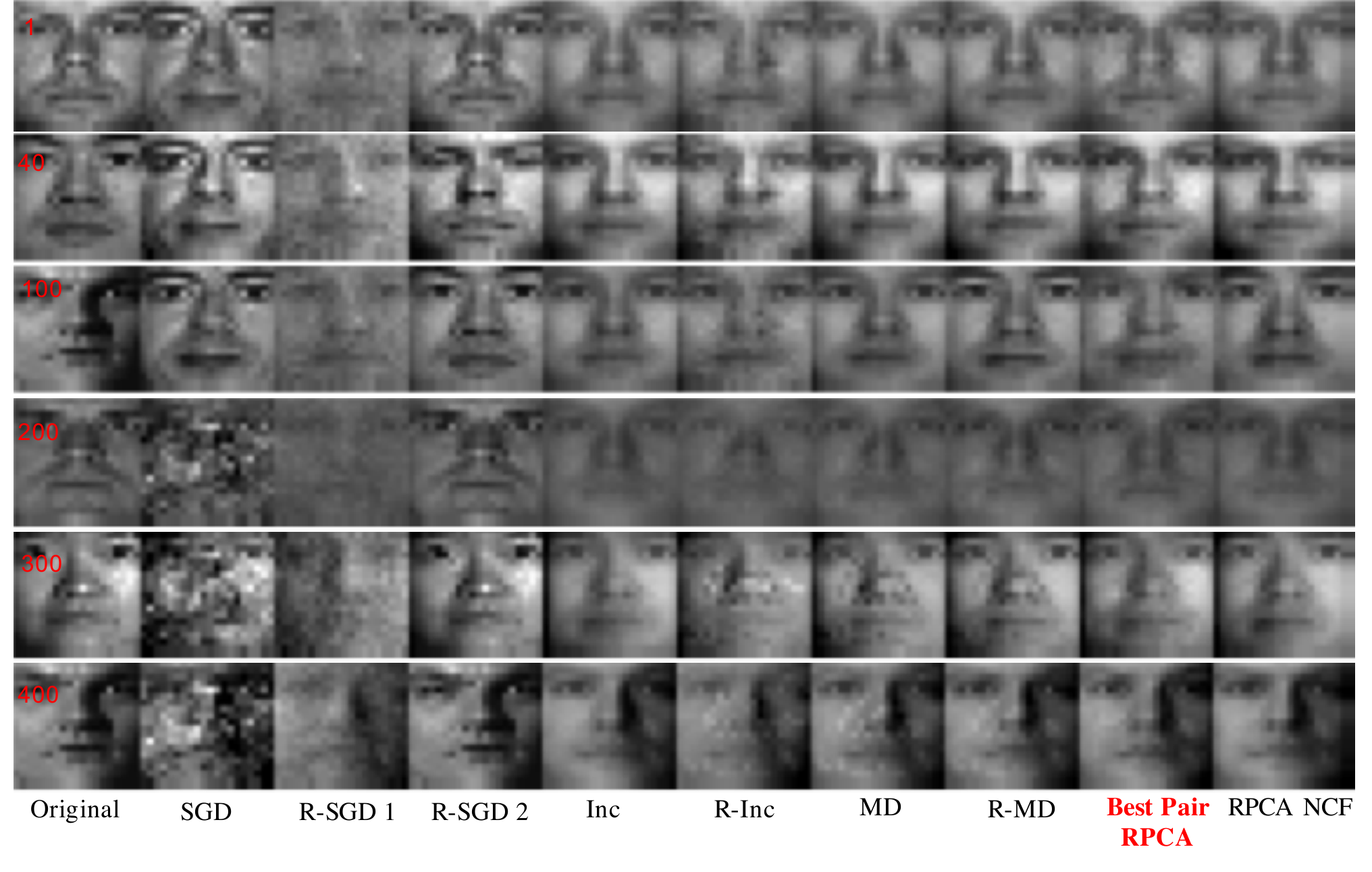}
    \caption{\small{Inliers and outliers detection. We project the face images (inliers) to 9 dimensional subspaces found by different methods.}}
    \label{inlier_detect}
\end{figure}

\begin{table}[h]
\begin{center}
\begin{tabular}{|l|c|c|c|c|c|c|c|c|c|c|c|}
\hline
\;\;Metric & SGD & R-SGD1 & R-SGD2 & Inc & R-Inc &MD & R-MD & RPCA-F& Best pair & SVT\\
\hline
$\frac{\|P_L-P_{L^*}\|_F}{3\sqrt{2}}$&0.7   & 0.86   & 4.66   & 0.77  &  0.72 &   0.67  &  0.67 &   0.78 & 0.76 & 0.79\\
\hline
\end{tabular}
\end{center}
\caption{\small{Quantitative performance of different algorithms in inlier detection experiment. Except R-SGD2 all methods are competitive.}}\label{inlier_data}
\end{table}

\section{Proof of the global convergence}
\label{sec:proof-global}

For convenience, define $\Dt{k} \eqdef \norm{\Yk - \Ykm}$.

\begin{lemma}\label{lem:subgradient}
For the update of $\Ykp$ in \eqref{eq:apg}, given any $k\in\bbN$, define
\beqn
\begin{aligned}
\Gkp
\eqdef \sfrac{1}{\gamma}\pa{\Zak-\Ykp} - \nabla \calF(\Zbk) + \nabla \calF(\Ykp)  .
\end{aligned}
\eeqn
Then, we have $\Gkp \in \partial \Phi(\Ykp)$, and 
\beq\label{eq:subg-1}
\norm{\Gkp}
\leq \bPa{\sfrac{1}{\gamma} + L}\Dt{k+1} + \pa{\sfrac{\ak}{\gamma}+\bk L}\Dt{k}  .
\eeq
\end{lemma}
\begin{proof}
From the definition of proximity operator and the update of $\Ykp$ \eqref{eq:apg}, we have $\Zak - \gamma \nabla \calF(\Zbk)  - \Ykp \in  \gamma \partial \calR(\Ykp)$. Adding $\gamma \nabla \calF(\Ykp)$ to both sides, we obtain
\[
\Gkp = \sfrac{ \Zak - \gamma \nabla \calF(\Zbk)  - \Ykp + \gamma \nabla \calF(\Ykp) }{\gamma}
\in  \partial \Phi(\Ykp)   .
\]
Applying further the triangle inequality together with the Lipschitz continuity of $\nabla \calF$, we get
\[
\begin{aligned}
\norm{\Gkp}  
&= \norm{ \tfrac{1}{\gamma}\pa{\Zak-\Ykp} - \nabla \calF(\Zbk) + \nabla \calF(\Ykp) } \\
&\leq \sfrac{1}{\gamma}\norm{\Zak-\Ykp} + L \norm{\Zbk-\Ykp}  
\leq \sfrac{1}{\gamma}\pa{\Dt{k+1} +{\ak} \Dt{k} } + L \pa{\Dt{k+1} + {\bk} \Dt{k} }  ,
\end{aligned}
\]
which concludes the proof.
\end{proof}

\begin{lemma}\label{lem:descent}
For Algorithm \ref{alg:apg}, given the parameters $\gamma, \ak, \bk$, the following inequality holds:
\beq\label{eq:descent}
\Phi(\Ykp) + \ubeta \Dt{k+1}^2
\leq
\Phi(\Yk) + \oalpha \Dt{k}^2   .
\eeq
\end{lemma}
\begin{proof}
Define the function
\beqn
\calLk(Y)
\eqdef \gamma \calR(Y) + \sfrac{1}{2}\norm{Y-\Zak}^2 + \gamma \iprod{Y}{\nabla \calF(\Zbk)}  .
\eeqn
It can be shown that the update of $\Ykp$ in \eqref{eq:apg} is equivalent to 
\beq\label{eq:ifb-L}
\Ykp \in
\argmin_{Y\in\bbR^{n}} \calLk(Y)  ,
\eeq
which means that $\calLk(\Ykp) \leq \calLk(\Yk)$, which means
\[
\calR(\Ykp) + \sfrac{1}{2\gamma}\norm{\Ykp-\Zak}^2 + \iprod{\Ykp}{\nabla \calF(\Zbk)}  
\leq
\calR(\Yk) + \sfrac{1}{2\gamma}\norm{\Yk-\Zak}^2 + \iprod{\Yk}{\nabla \calF(\Zbk)}.
\]
Therefore, we get
\beq\label{eq:Rxkp-leq-Rxk}
\begin{aligned}
\calR(\Yk)
&\geq \calR(\Ykp) + \sfrac{1}{2\gamma}\norm{\Ykp-\Zak}^2 + \iprod{\Ykp - \Yk}{\nabla \calF(\Zbk)} - \sfrac{1}{2\gamma}\norm{\Yk-\Zak}^2 \\
&= \calR(\Ykp) + \sfrac{1}{2\gamma}\norm{\Ykp - \Yk + \Yk-\Zak}^2 + \iprod{\Ykp - \Yk}{\nabla \calF(\Zbk)} - \sfrac{1}{2\gamma}\norm{\Yk-\Zak}^2 \\
&= \calR(\Ykp) + \sfrac{1}{2\gamma}\Dt{k+1}^2 + \iprod{\Ykp - \Yk}{\nabla \calF(\Zbk)} - \sfrac{\ak}{\gamma} \iprod{\Yk-\Ykp}{\Yk - \Ykm} \\
&= \calR(\Ykp) +  \iprod{\Ykp-\Yk}{\nabla \calF(\Yk)}  + \sfrac{1}{2\gamma}\Dt{k+1}^2  \\&\qquad - \sfrac{\ak}{\gamma} \iprod{\Yk-\Ykp}{\Yk - \Ykm}  + \iprod{\Ykp-\Yk}{\nabla \calF(\Zbk) - \nabla \calF(\Yk)}  .
\end{aligned}
\eeq
Since $\nabla \calF$ is $L$-Lipschitz, then
\[
\iprod{\nabla \calF(\Yk)}{\Ykp-\Yk}
\geq \calF(\Ykp) - \calF(\Yk) - \sfrac{L}{2}\Dt{k+1}^2  .
\]
For the inner product $\iprod{\Yk-\Ykp}{\Yk - \Ykm}$, applying the Pythagorean relation $2\iprod{c_1-c_2}{c_1-c_3} = \norm{c_1-c_2}^2 + \norm{c_1-c_3}^2 - \norm{c_2-c_3}^2$, we get 
\beq\label{eq:young-1}
\begin{aligned}
\iprod{\Yk-\Ykp}{\Yk - \Ykm}
&= \sfrac{1}{2} \Pa{ \norm{\Yk-\Ykp}^2 + \norm{\Yk-\Ykm}^2 - \norm{\Ykp-\Ykm}^2 }  \\
&\leq \sfrac{1}{2} \Pa{ \norm{\Yk-\Ykp}^2 + \norm{\Yk-\Ykm}^2  }  .
\end{aligned}
\eeq
Using further Young's inequality with $\nu > 0$ we obtain
\beq\label{eq:young-2}
\begin{aligned}
\iprod{\Ykp-\Yk}{\nabla \calF(\Zbk) - \nabla \calF(\Yk)} 
&\geq - \bPa{ \sfrac{\nu}{2}\Dt{k+1}^2 + \sfrac{1}{2\nu}\norm{\nabla \calF(\Zbk) - \nabla \calF(\Yk)}^2 }   \\
&\geq - \bPa{ \sfrac{\nu}{2}\Dt{k+1}^2 + \sfrac{\bk^2 L^2}{2\nu} \Dt{k}^2 }  .
\end{aligned}
\eeq
Combining the above $3$ inequalities with \eqref{eq:Rxkp-leq-Rxk} yields
\beq\label{eq:Rxkp-leq-Rxk-}
\begin{aligned}
\calR(\Yk)
&\geq \calR(\Ykp) +  \calF(\Ykp) - \calF(\Yk) - \sfrac{L}{2}\Dt{k+1}^2  + \sfrac{1}{2\gamma}\Dt{k+1}^2  \\&\qquad - \sfrac{\ak}{2\gamma} \norm{\Yk-\Ykp}^2 - \sfrac{\ak}{2\gamma} \norm{\Yk-\Ykm}^2  -  \sfrac{\nu}{2}\Dt{k+1}^2 - \sfrac{\bk^2 L^2}{2\nu} \Dt{k}^2  ,
\end{aligned}
\eeq
which leads to
\[
\Phi(\Ykp) + \sfrac{1 - \gamma L - \ak - \nu}{2\gamma}\Dt{k+1}^2 
\leq
\Phi(\Yk) + \sfrac{\gamma\bk^2 L^2 + \nu\ak}{2\nu\gamma} \Dt{k}^2 .
\]
Owing to the definition of $\ubeta$ and $\oalpha$ we conclude the proof.
\end{proof}

Define $\calH$ the product space $\calH \eqdef \bbR^{n} \times \bbR^{n}$ and $\Zk = (\Yk, \Ykm) \in \calH$. 
Then given $\Zk$, define the function
\beqn
\Psi(\Zk) 
\eqdef \Phi(\Yk) +  \oalpha \Dt{k}^2  ,
\eeqn
which is is a KL function if $\Phi$ is.
%
Denote $\calC_{\Yk}, \calC_{\Zk}$ the set of cluster points of sequences $\seq{\Yk}$ and $\seq{\Zk}$ respectively, and $\crit(\Psi) = \ba{Z = (Y, Y) \in \calH: Y \in \crit(\Phi) }$.

\begin{lemma}\label{lem:Dtk^2-in-l1}
For Algorithm \ref{alg:apg}, choose $\nu, \gamma, \ak, \bk$ such that \eqref{eq:ineq-parameters} holds. 
If $\Phi$ is bounded from below, then
\begin{enumerate}[label = {\rm (\roman{*})}]
\item
$\sum_{\kinN} \Dtk^2 < \pinf$;
\item
The sequence $\Psi(\Zk) $ is monotonically decreasing and convergent;
\item
The sequence $\Phi(\Yk)$ is convergent.
\end{enumerate}
\end{lemma}
\begin{proof}
Define $\delta = \ubeta - \oalpha   > 0$, from Lemma \ref{lem:descent}, we have
\beqn
\delta \Dt{k+1}^2 
\leq \Pa{\Phi(\Yk) - \Phi(\Ykp)} + \oalpha \pa{\Dt{k}^2 - \Dt{k+1}^2}   .
\eeqn
Let $Y_{-1} = Y_{0}$ and the above inequality over $k$:
\[
\begin{aligned}
\delta \msum_{\kinN} \Dt{k+1}^2
&\leq  \msum_{\kinN}\Pa{\Phi(\Yk) - \Phi(\Ykp)}  +  \msum_{\kinN} \oalpha \pa{\Dt{k}^2 - \Dt{k+1}^2}  \\
&\leq  \Phi(Y_{0})  +  \oalpha \msum_{\kinN} \pa{\Dt{k}^2 - \Dt{k+1}^2}   
= \Phi(Y_{0})  +  \oalpha \Dt{0}^2
= \Phi(Y_{0})  ,
\end{aligned}
\]
which means, as $\Phi(Y_{0})$ is bounded,
\[
\msum_{\kinN} \Dt{k+1}^2 
\leq \sfrac{\Phi(Y_{0})}{\delta}
< \pinf  .
\]

From Lemma \ref{lem:descent}, by pairing terms on both sides of \eqref{eq:descent}, we get
\beqn
\Psi(\Zkp) + \pa{\ubeta - \oalpha} \Dt{k+1}^2   
\leq \Psi(\Zk)  .
\eeqn
Since we assume $\ubeta - \oalpha > 0$, hence $\Psi(\Zk)$ is monotonically non-increasing. 
The convergence of $\Phi(\Yk)$ is straightforward.
\end{proof}

\begin{lemma}\label{lem:subsequence-convergence}
For Algorithm \ref{alg:apg}, choose $\nu, \gamma, \ak, \bk$ such that \eqref{eq:ineq-parameters} holds.  
If $\Phi$ is bounded from below and $\sequence{\Yk}$ is bounded, then $\Yk$ converges to a critical point of $\Phi$.
\end{lemma}
\begin{proof}
Since $\sequence{\Yk}$ is bounded, there exists a subsequence $\sequence{\Ykj}$ and cluster point $\Ybar$ such that $\Ykj \to \Ybar$ as $j \to \infty$. 
Next we show that $\Phi(\Ykj) \to \Phi(\Ybar)$ and that $\Ybar$ is a critical point of $\Phi$.

Since $\calR$ is lsc, then $\liminf_{j\to\infty} \calR(\Ykj) \geq \calR(\Ybar)$.
From \eqref{eq:ifb-L}, we have $\calL_{k_{j}-1}(\Ykj) \leq \calL_{k_{j}-1}(\Ybar)$ and thus 
\[
\begin{aligned}
\calR(\Ybar)
&\geq \calR(\Ykj) + \sfrac{1}{2\gamma}\norm{\Ykj-U_{k_{j}-1}}^2 + \iprod{\Ykj - \Ybar}{\nabla \calF(V_{k_{j}-1})} - \sfrac{1}{2\gamma}\norm{\Ybar-U_{k_{j}-1}}^2 \\
&= \calR(\Ykj) + \sfrac{1}{2\gamma} \pa{\norm{\Ykj - \Ybar}^2 + 2\iprod{\Ykj - \Ybar }{\Ybar-U_{k_{j}-1}}} + \iprod{\Ykj - \Ybar}{\nabla \calF(V_{k_{j}-1})}.
\end{aligned}
\]
Taking the limit of the above inequality and using $\Dt{k}^2\to 0$, $\Ykj \to \Ybar$, we get
 $\limsup_{j\to\infty} \calR(\Ykj) \leq \calR(\Ybar)$. As a result, $\lim_{k\to\infty} \calR(\Ykj) = \calR(\Ybar)$.
Since $\calF$ is continuous, then $\calF(\Ykj) \to \calF(\Ybar)$, hence $\Phi(\Ykj) \to \Phi(\Ybar)$.

Furthermore, owing to Lemma \ref{lem:subgradient}, $G_{k_j} \in \partial \Phi(\Ykj)$, and (i) of Lemma \ref{lem:Dtk^2-in-l1} we have $G_{k_j} \to 0$ as $k \to \infty$.
Therefore, as $j \to \infty$, we have
\[
G_{k_j} \in \partial \Phi(\Ykj) ,~~ (\Ykj, G_{k_j}) \to (\Ybar, 0)  \qandq  \Phi(\Ykj) \to \Phi(\Ybar) .
\]
Hence $0 \in \partial \Phi(\Ybar)$, \ie $\Ybar$ is a critical point.
\end{proof}


%
\begin{proof}[Proof of Theorem \ref{thm:convergence-apg}]
Putting together the above lemmas, we draw the following useful conclusions:
\begin{enumerate}[label={\rm({\bf C.\arabic{*}})}, ref = {\rm\bf C.\arabic{*}}, leftmargin=1.25cm]
\item \label{it:rrr1}
Denote $\delta = \ubeta - \oalpha$, then $\Psi(\Zkp) + \delta\Dt{k+1}^2 \leq \Psi(\Zk)$;
\item \label{it:rrr2}
Define
\[
W_{k}
\eqdef 
\begin{pmatrix}
\Gk + 2  \oalpha (\Yk-\Ykm)  \\
2 \oalpha (\Ykm-\Yk) 
\end{pmatrix} ,
\]
then we have $W_{k} \in \partial \Psi(\Zk)$. Owing to Lemma \ref{lem:subgradient}, there exists a $\sigma > 0$ such that $\norm{W_{k}} \leq \sigma (\Dt{k} + \Dt{k-1})$; 
\item \label{it:rrr3}
if $\Ykj$ is a subsequence such that $\Ykj \to \Ybar$, then $\Psi(\Zk) \to \Psi(\Zbar)$ where $\Zbar = (\Ybar, \Ybar)$.
\item \label{it:rrr4}
$\calC_{\Zk} \subseteq \crit(\Psi)$;
\item \label{it:rrr5}
$\lim_{k\to\infty} \dist\pa{\Zk, \calC_{\Zk}} = 0$;
\item \label{it:rrr6}
$\calC_{\Zk}$ is non-empty, compact and connected;
\item \label{it:rrr7}
$\Psi$ is finite and constant on $\calC_{\Zk}$.
\end{enumerate}
Next we prove the claims of Theorem \ref{thm:convergence-apg}.
\begin{enumerate}[label={\rm (\roman{*})}, leftmargin=2em]
\item Consider a critical point of $\Phi$, $\Ybar \in \crit(\Phi)$, such that $\Zbar = (\Ybar,\Ybar) \in \calC_{\Zk}$. Then owing to \iref{it:rrr3}, we have $\Psi(\Zk) \to \Psi(\Zbar)$.

{\hspace{12pt}}Suppose there exists $K$ such that $\Psi(Z_{K}) = \Psi(\Zbar)$. Then, the descent property \iref{it:rrr1} implies that $\Psi(\Zk) = \Psi(\Zbar)$ holds for all $k \geq K$. Thus, $\Zk$ is constant for $k\geq K$, hence has finite length.

{\hspace{12pt}}On the other hand, suppose that $\psi_k \eqdef \Psi(\Zk) - \Psi(\Zbar)>0$.
Owing to \iref{it:rrr6}, \iref{it:rrr7} and Definition \ref{defn:KLp}, the KL property of $\Psi$ implies that there exist $\epsilon, \eta$ and a concave function $\varphi$, and
\beq\label{eq:calU}
\calU \eqdef 
\Ba{ S \in \calH: \dist\pa{ S, \calC_{\Zk}} < \epsilon}
\mcap
\big[\Psi(\Zbar) < \Psi(S) < \Psi(\Zbar) + \eta\big] , 
\eeq
such that for all $Z \in \calU$:
\beq\label{eq:Psi-KL}
\varphi'\Pa{\Psi(z)-\Psi(\Zbar)} \dist\Pa{0, \partial \Psi(z)} \geq 1  .
\eeq
Let $k_1 \in \bbN$ be such that $\Psi(\Zk) < \Psi(\Zbar)+\eta$ holds for all $k \geq k_1$. Owing to \iref{it:rrr5}, there exists another $k_2 \in \bbN$ such that $\dist(\Zk, \calC_{\Zk}) < \epsilon$ holds for all $k \geq k_2$.
Let $K = \max\ba{k_1, k_2}$. Then $\Zk \in \calU$ holds for all $k \geq K$.
Furthermore using \eqref{eq:Psi-KL}, we have for $k\geq K$
\beqn
\varphi'\pa{\psi_k} \dist\Pa{0, \partial \Psi(\Zk)} \geq 1  .
\eeqn
Note that since $\varphi$ is concave, $\varphi'$ is decreasing. As $\Psi(\Zk)$ is decreasing too, we have
\[
\varphi\pa{\psi_k} - \varphi\pa{\psi_{k+1}}
\geq \varphi'\pa{\psi_k} \Pa{\Psi(\Zk) - \Psi(\Zkp)}  
\geq \sfrac{\Psi(\Zk) - \Psi(\Zkp)}{\dist\pa{0, \partial \Psi(\Zk)}}  .
\]
From \iref{it:rrr1}, since $\dist\pa{0, \partial \Psi(\Zk)} \leq \norm{\wk}$, we get
\[
\varphi\pa{\psi_k} - \varphi\pa{\psi_{k+1}}
\geq \sfrac{\Psi(\Zk) - \Psi(\Zkp)}{\norm{\wk}}  
\geq \sfrac{\Psi(\Zk) - \Psi(\Zkp)}{ \sigma (\Dt{k} + \Dt{k-1}) }  .
\]
Moreover, \iref{it:rrr2} yields $\Psi(\Zk) - \Psi(\Zkp) \geq \delta \Dt{k+1}^2$ and thus
\beqn
\varphi\pa{\psi_k} - \varphi\pa{\psi_{k+1}}
\geq \sfrac{ \delta \Dt{k+1}^2 }{ \sigma (\Dt{k} + \Dt{k-1}) }  ,
\eeqn
which yields
\beq\label{eq:epsk-phik}
\Dt{k+1}^2 
\leq  \Pa{ \sfrac{\sigma}{\delta} \pa{\varphi\pa{\psi_k} - \varphi\pa{\psi_{k+1}} } }  (\Dt{k} + \Dt{k-1})  .
\eeq
Taking the square root of both sides and applying Young's inequality  we further obtain
\beq\label{eq:2Dtkp}
\begin{aligned}
\Dt{k+1}
&\leq \sfrac{1}{2} (\Dt{k} + \Dt{k-1}) +  \sfrac{2\sigma}{\delta} \Pa{\varphi\pa{\psi_k} - \varphi\pa{\psi_{k+1}} }  .
\end{aligned}
\eeq
Summing up both sides over $k$, and using $x_{0}=x_{-1}$, we get
\[
\ell \eqdef 
\msum_{k\in\bbN} \Dt{k} 
\leq \Dt{1} + \sfrac{2\sigma}{\delta} \varphi\pa{\psi_1}
< \pinf  ,
\]
which concludes the finite length property of $\Yk$.

\item Then the convergence of the sequence follows from the fact that $\sequence{\Yk}$ is a Cauchy sequence, hence convergent. Owing to Lemma \ref{lem:subsequence-convergence}, there exists a critical point $\Ysol \in \crit(\Phi)$ such that~$\lim_{k\to\infty}\Yk = \Ysol$.

\item We now turn to prove local convergence to a global minimizer. 
Note that if $\Ysol$ is a global minimizer of $\Phi$, then $\Zsol$ is a global minimizer of $\Psi$. 
Let $r > \rho > 0$ such that $\ball{r}{\Zsol} \subset \calU$ and $\eta < \delta(r-\rho)^2$. Suppose that the initial point $Y_{0}$ is chosen such that following conditions hold,
\begin{eqnarray}
\label{eq:Psineighb}
& \Psi(\Zsol) \leq \Psi(Z_{0}) < \Psi(\Zsol) + \eta \label{eq:Phi0} \\
& \norm{Y_{0}-\Ysol} + \ell(s-1) + 2\sqrt{\sfrac{\Psi(Z_{0})-\Psi(\Zsol)}{\delta}} + \qfrac{\sigma}{\delta} \varphi\pa{\psi_0} < \rho   .  \label{eq:x0}
\end{eqnarray}
The descent property \iref{it:rrr1} of $\Psi$ together with \eqref{eq:Psineighb} imply that for any $k\in\bbN$, $\Psi(\Zsol) \leq \Psi(\Zkp) \leq \Psi(\Zk) \leq \Psi(Z_{0}) < \Psi(\Zsol) + \eta \label{eq:Phi0}$, and
\beq\label{eq:up-bnd-Dtkp}
\norm{\Ykp-\Yk} \leq \sqrt{ \sfrac{ \Psi(\Zk) - \Psi(\Zkp) }{\delta} }  \leq \sqrt{ \sfrac{ \Psi(\Zk) - \Psi(\Zsol) }{\delta} }.
\eeq
Therefore, given any $k\in\bbN$, if we have $\Yk \in \ball{\rho}{\Ysol}$, then
\beq\label{eq:xk}
\begin{aligned}
\norm{\Ykp-\Ysol} 
&\leq \norm{\Yk-\Ysol} + \norm{\Ykp-\Yk}
\leq \norm{\Yk-\Ysol} +  \sqrt{ \sfrac{ \Psi(\Zk) - \Psi(\Zsol) }{\delta} }  \\
&\leq \rho + (r - \rho) = r  ,
\end{aligned}
\eeq
which means that $\Ykp \in \ball{r}{\Ysol}$.

{\hspace{12pt}}For any $k\in\bbN$, define the following partial sum $p_{k} \eqdef \sum_{j=k-2}^{k-1} \sum_{i=1}^{j} \Dt{i}$. 
Note that $p_{k} = 0$ for $k=1$, and $\lim_{k\to+\infty} p_{k} = \ell$.  
Next we prove the following claims through induction: for $k\in\bbN$
\begin{eqnarray}
& \Yk \in \ball{\rho}{\Ysol}  \label{eq:xkk} \\
& \msum_{j=1}^{k}\Dt{j+1} + \Dt{k+1}
    \leq \Dt{1} + p_{k} + \sfrac{\sigma}{\delta} \Pa{\varphi\pa{\psi_{1}} - \varphi\pa{\psi_{k+1}} }  . \label{eq:Dtkk}
\end{eqnarray}
From \eqref{eq:up-bnd-Dtkp} we have
\beq\label{eq:x1-x0}
\norm{Y_{1}-Y_{0}}
\leq \sqrt{ \sfrac{ \Psi(Z_{0}) - \Psi(\Zsol) }{\delta} } .
\eeq
Applying the triangle inequality we then obtain
\[
\norm{Y_{1}-\Ysol}
\leq \norm{Y_{0}-\Ysol} + \norm{Y_{1}-Y_{0}}
\leq \norm{Y_{0}-\Ysol} + \sqrt{ \sfrac{ \Psi(Z_{0}) - \Psi(\Zsol) }{\delta} }
< \rho  ,
\]
which means $Y_{1} \in \ball{\rho}{\Ysol}$. 
Now, taking $\kappa = 1$ in \eqref{eq:2Dtkp} yields, for any $k \in \bbN$,
\beq\label{eq:2Dtkp-2}
2\Dt{k+1}
\leq  (\Dt{k} + \Dt{k-1}) +  \sfrac{\sigma}{\delta} \pa{\varphi\pa{\psi_k} - \varphi\pa{\psi_{k+1}} }  .
\eeq
Let $k = 1$. Since $Y_{0}=Y_{-1}$, we have
\[
2\Dt{2}
\leq  \Dt{1} +  \sfrac{\sigma}{\delta} \pa{\varphi\pa{\psi_{1}} - \varphi\pa{\psi_{2}} }  .
\]
Therefore, \eqref{eq:xkk} and \eqref{eq:Dtkk} hold for $k = 1$. 

{\hspace{12pt}}Now assume that they hold for some $k > 1$. Using the triangle inequality and \eqref{eq:Dtkk}, 
\[
\begin{aligned}
\norm{\Ykp-\Ysol}
&\leq \norm{Y_{0}-\Ysol} + \Delta_1 + \msum_{j=1}^{k}\Dt{j+1}  \\
&\leq \norm{Y_{0}-\Ysol} + 2\Delta_1 + p_{k} + \sfrac{\sigma}{\delta} \pa{\varphi\pa{\psi_{1}} - \varphi\pa{\psi_{k+1}} }  \\
&\leq \norm{Y_{0}-\Ysol} + 2\Delta_1 + \ell + \sfrac{\sigma}{\delta} \pa{\varphi\pa{\psi_{1}} - \varphi\pa{\psi_{k+1}} }  \\
\eqref{eq:x1-x0}&\leq \norm{Y_{0}-\Ysol} + 2\sqrt{ \sfrac{ \Psi(Z_{0}) - \Psi(\Zsol) }{\delta} } + \ell + \sfrac{\sigma}{\delta} \pa{\varphi\pa{\psi_{1}} - \varphi\pa{\psi_{k+1}} }  .
\end{aligned}
\]
As $\varphi(\psi) \geq 0$ and $\varphi'(\psi) > 0$ for $\psi \in ]0,\eta[$, and in view of \eqref{eq:x0}, we arrive at
\[
\norm{\Ykp-\Ysol} \leq \norm{Y_{0}-\Ysol} + 2\sqrt{ \sfrac{ \Psi(Z_{0}) - \Psi(\Zsol) }{\delta} } + \ell + \sfrac{\sigma}{\delta}\varphi\pa{\psi_{0}} < \rho
\]
whence we deduce that \eqref{eq:xkk} holds at $k+1$. 
Now, taking \eqref{eq:2Dtkp-2} at $k+1$ gives 
\beq\label{eq:2Dtkpp}
\begin{aligned}
2\Dt{k+2}
&\leq (\Dt{k+1} + \Dt{k}) +  \sfrac{\sigma}{\delta} \pa{\varphi\pa{\psi_{k+1}} - \varphi\pa{\psi_{k+2}} }  \\
&\leq \Dt{k+1} + (\Dt{k} + \Dt{k-1}) +  \sfrac{\sigma}{\delta} \pa{\varphi\pa{\psi_{k+1}} - \varphi\pa{\psi_{(k+2}} }   .
\end{aligned}
\eeq
Adding both sides of \eqref{eq:2Dtkpp} and \eqref{eq:Dtkk} we get
\[
\begin{aligned}
\msum_{j=1}^{k+1}\Dt{j+1} + \Dt{k+2}
&\leq \Dt{1} + p_{k} + (\Dt{k} + \Dt{k-1}) + \sfrac{\sigma}{\delta} \pa{\varphi\pa{\psi_{1}} - \varphi\pa{\psi_{k+2}} }  \\  
&= \Dt{1} + p_{k+1} + \sfrac{\sigma}{\delta} \pa{\varphi\pa{\psi_{1}} - \varphi\pa{\psi_{k+2}} }  ,
\end{aligned}
\]
meaning that \eqref{eq:Dtkk} holds at $k+1$. This concludes the induction proof. 

{\hspace{12pt}}In summary, the above result shows that if we start close enough from $\Ysol$ (so that \eqref{eq:Psineighb}-\eqref{eq:x0} hold), then the sequence $\seq{\Yk}$ will remain in the neighbourhood $\ball{\rho}{\Ysol}$ and thus converges to a critical point $\Ybar$ owing to Lemma~\ref{lem:subsequence-convergence}. Moreover, $\Psi(\Zk) \to \Psi(\Zbar) \geq \Psi(\Zsol)$ by virtue of \iref{it:rrr3}. 
Now we need to show that $\Psi(\Zbar) = \Psi(\Zsol)$. Suppose that $\Psi(\Zbar) > \Psi(\Zsol)$. As $\Psi$ has the KL property at $\Zsol$, we have
\[
\varphi'\Pa{\Psi(\Zbar)-\Psi(\Zsol)} \dist\Pa{0, \partial \Psi(\Zbar)} \geq 1  .
\]
But this is impossible since $\varphi'(s) > 0$ for $s \in ]0,\eta[$, and $\dist\Pa{0, \partial \Psi(\Zbar)} = 0$ as $\Zbar$ is a critical point. Hence we have $\Psi(\Zbar) = \Psi(\Zsol)$, 
which means $\Phi(\Ybar) = \Phi(\Ysol)$, \ie the cluster point $\Ybar$ is actually a global minimizer. This concludes the proof. \qedhere
\end{enumerate}
\end{proof}

\section{Proof of local linear convergence}
\label{sec:proof-local}

Before presenting the proof for local linear convergence, in Figure \ref{fig:rate} below we provide the comparison of theoretical estimation and practical observation. The size of the problem is $\bbR^{32\times 32}$, which is small as larger size will make the rate estimation very slow. It can be observed that our theoretical rate estimation is very tight given that the red line and the black one are parallel to each other.

\begin{figure}[!ht]
\centering
\begin{minipage}{0.4\textwidth}
  \centering
\includegraphics*[width =  \textwidth ]{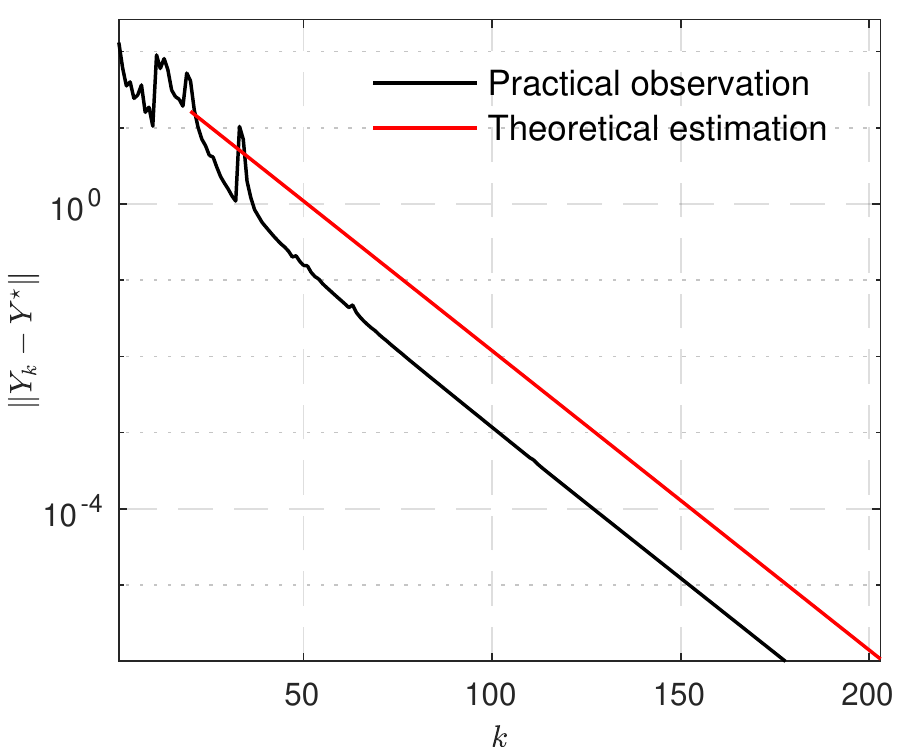} 
        { $\ak\equiv0$ }
\end{minipage}
\begin{minipage}{0.05\textwidth} $~$ \end{minipage}
\begin{minipage}{0.4\textwidth}
  \centering
\includegraphics*[width =  \textwidth ]{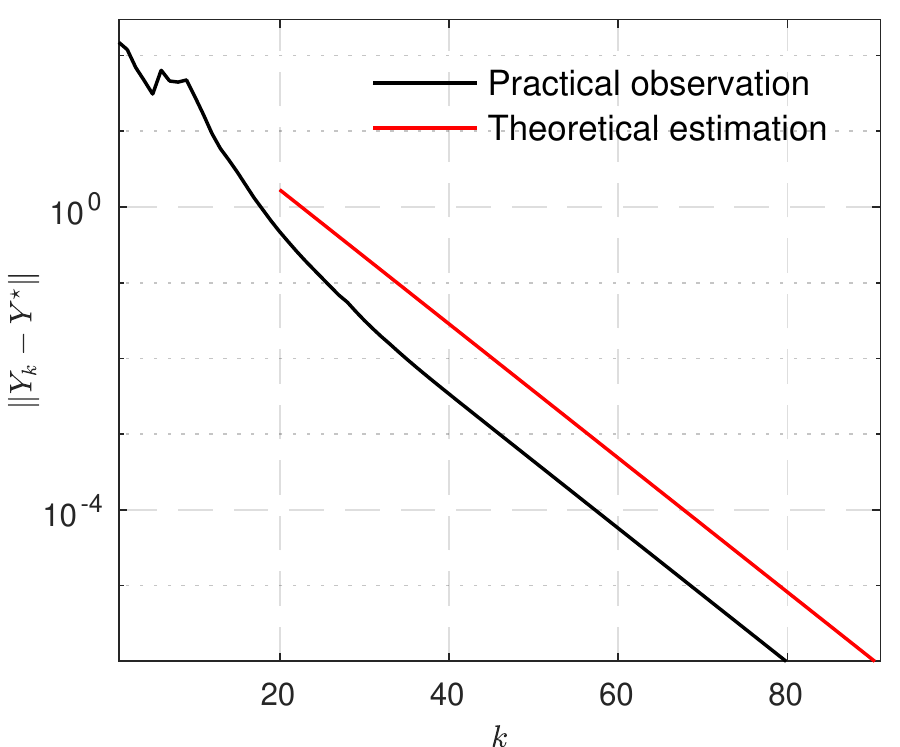} 
        { $\ak\equiv a = 0.5$  }
\end{minipage}%
\caption{ Local linear convergence of Algorithm \ref{alg:apg}.}
 \label{fig:rate}
\end{figure}

Since we are in the non-convex setting, we need the prox-regularity of the non-convexity. 
A lower semi-continuous function $\calR$ is $r$-prox-regular at $\xbar \in \dom(\calR)$ for $\vbar \in \partial \calR(\xbar)$ if $\exists r > 0$ such that $\calR(x') > \calR(x) + \iprod{v}{x'-x} - \frac{1}{2r}\norm{x-x'}^2$ $\forall x,x'$ near $\xbar$, $\calR(x)$ near $\calR(\xbar)$ and $v\in\partial \calR(x)$ near $\vbar$. 

To prove Theorem \ref{thm:supp-iden}, we rely on a so-called partial smoothness concept. 
Let $\calM \subset \bbR^{n}$ be a $C^2$-smooth submanifold, let $\tanSp{\calM}{x}$ the tangent space of $\calM$ at any point $x \in \calM$.

 \begin{definition}\label{dfn:psf}
 The function $\calR : \bbR^{n} \to \bbR \cup \ba{\pinf}$ is \emph{$C^2$-partly smooth at $\xbar \in \calM$ relative to $\calM$ for $\vbar \in \partial \calR(\xbar) \neq \emptyset$} if $\calM$ is a $C^2$-submanifold around $\xbar$, and
 \begin{enumerate}[label={\rm (\roman{*})}]
 \item {\pa{Smoothness}}: \label{PS:C2} 
 $\calR$ restricted to $\calM$ is $C^2$ around $\xbar$;
 \item {\pa{Regularity}}: \label{PS:Regular} 
$\calR$ is regular at all $x\in\calM$ near $\xbar$ and $\calR$ is $r$-prox-regular at $\xbar$ for $\vbar$;
 \item {\pa{Sharpness}}: \label{PS:Sharp} 
$\tanSp{\calM}{\xbar} = \LinHull\pa{\partial \calR(x)}^\perp$;
 \item {\pa{Continuity}}: \label{PS:DiffCont} 
 The set-valued mapping $\partial \calR$ is continuous at $\xbar$ relative to $\calM$.
 \end{enumerate}
 \end{definition}

We denote the class of partly smooth functions at $x$ relative to $\calM$ for $v$ as $\PSF{x,v}{\calM}$. Partial smoothness was first introduced in \citep{LewisPartlySmooth} and its directional version stated here is due to \citep{LewisPartlyTiltHessian,drusvyatskiy2013optimality}. Prox-regularity is sufficient to ensure that the partly smooth submanifolds are locally unique \citep[Corollary~4.12]{LewisPartlyTiltHessian}, \citep[Lemma~2.3 and Proposition~10.12]{drusvyatskiy2013optimality}.

\begin{proof}[Proof of Theorem \ref{thm:supp-iden}]
First we have
\begin{itemize}
\item
$\calY_{L}$ is a the set of fixed-rank matrices, hence it is partly smooth. 
\item
Since $\calS$ is a subspace, hence it is partly smooth at $S^\star$ relative to any $W \in (\calS)^\bot$. 
\end{itemize}
Under the conditions of Theorem \ref{thm:convergence-apg}, there exists a critical point $\Ysol$ such that $\Yk \to \Ysol$ and $\Phi(\Yk) \to \Phi(\Ysol)$.

Convergence properties of $\sequence{\Yk}$ (Theorem~\ref{thm:convergence-apg}) entails $\norm{\Zak-\Yk} \to 0$ and $\norm{\Zbk-\Ysol} \to 0$. In turn,
\beqn
\dist\Pa{-\nabla \calF(\xsol), \partial \calR(\Ykp)} 
\leq \sfrac{1}{\gamma} \norm{\Zak - \Ykp} + \norm{ \Zbk - \Ysol } \to 0  . 
\eeqn
%
Altogether, this shows that the conditions of \citep[Theorem~4.10]{LewisPartlyTiltHessian} or \citep[Proposition~10.12]{drusvyatskiy2013optimality} are fulfilled on $\calR$ at $\Ysol$ for $-\nabla \calF(\Ysol)$, and the identification result follows, that is 
\[
(\calY_{r, k} , \calY_{\alpha, k} ) \in \calY_{L} \times \calS
\]
for all $k$ large enough, and we conclude the proof. 
\qedhere
\end{proof}

\paragraph{Tangent space $T_{\calX}^{\Xsol}$}

Given $\Xsol \in \calX$, the tangent space simply reads $K X = 0$. Let $E$ be the kernel of $K$, then we have the projection operator onto $KX=0$ reads
\[
\proj_{T_{\calX}^{\Xsol}} = E(E^TE)^{-1}E^T   .
\]

\paragraph{Tangent space of $\calY_{L}$}
Let $\bbM = M_{m,n}(\bbR)$ be the space of $m\times n$ matrices with the classical inner product  $\iprod{A}{B} = \trace\pa{A^TB}$. 
The set of matrices with fixed rank $r$,
\[
\calY_{L} = \Ba{X\in \bbM: \rank\pa{X} = r} ,
\]
is a smooth manifold around any matrix $L\in\calY_{L}$. 
Given $L^\star$, with the help of the singular value decomposition $L=U\Sigma V^T$, the tagent space at $L$ to $\calY_{L}$ is 
\[
T_{\calY_{L}}^{L^\star} = \Ba{H\in \bbM: u_{i}^THv_{j} = 0,~\mathrm{for~all}~ r< i\leq m, r< j \leq n} .
\]
Let $U = [u_1,u_2,\cdots,u_m]$, $V = [v_1,v_2,\cdots,v_n]$ and $\Sigma$ be diagonal matrix with singular value written in decreasing order.

%
%
%

Denote
\[
\calL = \bBa{ L\in \bbM: X = u_{i}^Tv_{j},~\mathrm{for~all}~ \ba{i,j}_{1\leq i\leq m, 1\leq j\leq n} \setminus \ba{i,j}_{r< i\leq m, r< j \leq n} } ,
\]
then $\calL$ forms the basis of $\calT$ and $\mathrm{dim}(\calL) = mn-r^2$, there for define
\[
Z = [L_1(:); L_2(:); \dotsm ; L_{mn-r^2}(:)],~~ L_i \in \calL ,
\]
and 
\[
\proj_{T_{\calY_{L}}^{L^\star}} = Z(Z^TZ)^{-1}Z^T ,
\]
then $\proj_{T_{\calY_{L}}^{L^\star}}$ is the explicit form of the projection operator of projecting onto subspace $T_{\calY_{L}}^{L^\star}$.

\paragraph{Tangent space of $\calS$}

Given $S^\star \in \calS$, denote the tangent space as $T_{\calS}^{S^\star}$. Let $\mathrm{vec}(S^\star)$ be the vector form of $S^\star$, then we haves
\[
\proj_{T_{\calS}^{S^\star}} = \mathrm{diag}\Pa{ \abs{\mathrm{vec}(S^\star)} > 0 }  .
\]

Finally, we have
\[
\proj_{T_{\calY}^{\Ysol}}
= \begin{bmatrix}  \proj_{T_{\calS}^{S^\star}} & \\ & \proj_{T_{\calY_{L}}^{L^\star}}  \end{bmatrix}  .
\]

\begin{proof}[Proof of Theorem \ref{thm:local-rate}]
From \eqref{eq:apg}, when $\ak, \bk\equiv0$, we have thats
\[
\Ykp = \proj_{\calY}\Pa{\Yk - \gamma \pa{\Yk - \proj_{\calX}(\Yk)}}   .
\]
Let $\Ysol$ be a critical point that $\Yk$ converges to, then
\[
\Ysol = \proj_{\calY}\Pa{\Ysol - \gamma \pa{\Ysol - \proj_{\calX}(\Ysol)}}   .
\]
Denote $\Xk = \proj_{\calX}(\Yk)$ and $\Xsol = \proj_{\calX}(\Ysol)$, we have
\[
\begin{aligned}
\Xk-\Xsol
= \proj_{T_{\calX}^{\Xsol}} (\Xk-\Xsol) 
= \proj_{T_{\calX}^{\Xsol}} \proj_{\calX}(\Yk -\Ysol) 
&= \proj_{T_{\calX}^{\Xsol}} (\Yk -\Ysol)   \\
&= \proj_{T_{\calX}^{\Xsol}} \proj_{T_{\calY}^{\Ysol}}(\Yk -\Ysol) + o(\norm{\Yk-\Ysol})  .
\end{aligned}
\]
Consider the difference of the above two equations, owing to Lemma \ref{lem:proj-M}, we get
\beqn
\begin{aligned}
\Ykp - \Ysol
&= \proj_{\calY}\Pa{\Yk - \gamma \pa{\Yk - \proj_{\calX}(\Yk)}} - \proj_{\calY}\Pa{\Ysol - \gamma \pa{\Ysol - \proj_{\calX}(\Ysol)}}  + o(\norm{\Yk-\Ysol})  \\
&= \proj_{\calY}\Pa{ (1-\gamma)\Yk + \gamma \proj_{\calX}(\Yk) } - \proj_{\calY}\Pa{ (1-\gamma)\Ysol + \gamma \proj_{\calX}(\Ysol) }  + o(\norm{\Yk-\Ysol})  \\
&= \proj_{T_{\calY}^{\Ysol}} \Pa{ (1-\gamma)\Yk + \gamma \proj_{\calX}(\Yk) - (1-\gamma)\Ysol - \gamma \proj_{\calX}(\Ysol) }  + o(\norm{\Yk-\Ysol})  \\
&= \proj_{T_{\calY}^{\Ysol}} \Pa{ (1-\gamma)(\Yk-\Ysol) + \gamma (\Xk-\Xsol)  }  + o(\norm{\Yk-\Ysol})  \\
&= \proj_{T_{\calY}^{\Ysol}} \Pa{ (1-\gamma) \Id + \gamma \proj_{T_{\calX}^{\Xsol}}  } \proj_{T_{\calY}^{\Ysol}} (\Yk-\Ysol)  + o(\norm{\Yk-\Ysol})   ,
\end{aligned}
\eeqn
which means
\[
\Ykp - \Ysol  = \calP(\Yk-\Ysol) + o(\norm{\Yk-\Ysol})  .
\]
Note that $\calP$ is symmetric positive semi-definite, hence all its eigenvalues are real and lie in $[0, 1]$. 

Now, assume that $\bk=\ak \equiv a $, then we have from \eqref{eq:apg}
\[
\begin{gathered}
\Zk = (1+a)\Yk - a \Ykm    ,  \\
\Ykp = \proj_{\calY}\Pa{\Zk - \gamma \pa{\Zk - \proj_{\calX}(\Zk)}}  .
\end{gathered}
\]
Follow the derivation of $\Ykp - \Ysol $ above, we get
\[
\begin{aligned}
\Ykp - \Ysol  
&= (1+a)\calP(\Yk-\Ysol) - a \calP(\Ykm-\Ysol)  + o(\norm{\Yk-\Ysol})  \\
&= \begin{bmatrix} (1+a)\calP & - a\calP \end{bmatrix} \begin{pmatrix} \Yk-\Ysol \\ \Ykm - \Ysol \end{pmatrix}  + o(\norm{\Yk-\Ysol})  .
\end{aligned}
\]
Plus the definition of $D_k$ and the fact that $o(\norm{\Yk-\Ysol}) = o(\norm{D_{k}})$, we obtain
\[
D_{k+1}
= \calQ D_{k} + o(\norm{D_{k}})   .
\]

Owing to \citep[Chapter 6]{liang2016convergence}, if $\rho_{_{\calP}} < 1$, then so is $\rho_{_{\calQ}} < 1$, and the linear convergence result follows. \qedhere

\end{proof}

\section{Table of baseline methods \label{sec:table}}
\begin{table*}[!h]
\footnotesize
\begin{center}
\begin{tabular}{|c|c|c|c|c|}
\hline
{ \bf Algorithm} & { \bf Abbreviation} &{ \bf Appearing in Experiment }& { \bf Reference}\\
\hline
\hline
 \begin{tabular}{c}
 Inexact Augmented Lagrange \\
Method of Multipliers 
 \end{tabular}
& iEALM & Fig. \ref{syntheticdata1}, \ref{shadow_removal}, \ref{BG_full} &\citep{LinChenMa}  \\
\hline
Accelerated Proximal Gradient & APG & Fig. \ref{shadow_removal}, \ref{BG_full} &\citep{APG}  \\
\hline
Singular Value Thresholding & SVT & Table \ref{inlier_data} &\citep{caicandesshen}  \\
\hline
 \begin{tabular}{c}
 Grassmannian Robust Adaptive \\
Subspace Tracking Algorithm
 \end{tabular}
 & GRASTA& Fig. \ref{Bg_subsampled}  & \citep{grasta} \\
 \hline
Go Decomposition & GoDec & Fig. \ref{BG_full}, \ref{qunatitative_BG}  & \citep{godec}\\
\hline
Robust PCA Gradient Descent & RPCA GD & Fig. \ref{syntheticdata2}, \ref{shadow_removal}, \ref{BG_full}, \ref{Bg_subsampled}, \ref{syntheticdata3}, \ref{qunatitative_BG}, \ref{qunatitative_BG_1}   & \citep{RPCAgd}\\
\hline
Robust PCA Nonconvex Feasibility & RPCA NCF &  Fig. \ref{syntheticdata1}, \ref{shadow_removal}, \ref{BG_full}, \ref{Bg_subsampled}, \ref{qunatitative_BG}, \ref{qunatitative_BG_1} ,  \ref{inlier_detect} & \citep{duttahanzely}\\
\hline
Robust stochastic PCA Algorithms & \begin{tabular}{c}
 SGD, R-SGD1,  R-SGD2  \\
Inc,~R-Inc,~MD,~R-MD
 \end{tabular}& Fig.  \ref{inlier_detect},  Table \ref{inlier_data} & \citep{rspca} \\
\hline
\end{tabular}
\end{center}
\caption{\small{Algorithms compared in this paper.}}\label{algo}
\end{table*}

\end{document}